\documentclass[a4paper,11pt]{article}

\usepackage{amsmath,amsthm,amsfonts,amssymb,stmaryrd,amscd}
\usepackage{imakeidx}
\usepackage{graphicx}
\usepackage[french, english]{babel}

\usepackage{pifont}
\usepackage{mathrsfs}

\input xy
\xyoption{all}

\usepackage[dvipsnames]{xcolor}
\usepackage{hyperref}
\makeindex[title=Index of Notation]

\usepackage[refpage]{nomencl}
\makenomenclature

\usepackage{ifthen}
  \renewcommand{\nomgroup}[1]{%
  \item[\bfseries
  \ifthenelse{\equal{#1}{G}}{Groupoids, deformation and blowup spaces}{%
  \ifthenelse{\equal{#1}{B}}{Fiber bundles}{%
   \ifthenelse{\equal{#1}{H}}{$\mathbf{C^*}$-Algebras}{%
  \ifthenelse{\equal{#1}{I}}{$\mathbf{KK}$-elements}{%
   \ifstrequal{#1}{O}{Other Symbols}{}}}}}%
    ]}

\numberwithin{equation}{section}

{\theoremstyle{definition}
\newtheorem{definition}{Definition}[section]
\newtheorem{definitions}{Definitions}[section]

\newtheorem{notation}[definition]{Notation}

\newtheorem{remark}[definition]{Remark}
\newtheorem{remarks}[definition]{Remarks}
\newtheorem{example}[definition]{Example}

}
\newtheorem{proposition}[definition]{Proposition}
\newtheorem{proposition-definition}[definition]{Proposition-Definition}
\newtheorem{lemma}[definition]{Lemma}
\newtheorem{theorem}[definition]{Theorem}
\newtheorem{fact}[definition]{Fact}
\newtheorem{corollary}[definition]{Corollary}

\newcommand{\cB}{\mathcal{B}}

\newcommand{\cE}{\mathcal{E}}
\newcommand{\cF}{\mathcal{F}}
\newcommand{\cG}{\mathcal{G}}

\newcommand{\cL}{\mathcal{L}}

\newcommand{\cV}{\mathcal{V}}
\newcommand{\cD}{\mathcal{D}}
\newcommand{\cK}{\mathcal{K}}

\newcommand{\cN}{\mathcal{N}}

\newcommand{\cP}{\mathcal{P}}

\newcommand{\cS}{\mathcal{S}}
\newcommand{\X}{\mathcal{X}}

\newcommand{\R}{\mathbb{R}}
\newcommand{\bK}{\mathbb{K}}

\newcommand{\bP}{\mathbb{P}}
\newcommand{\bS}{\mathbb{S}}
\newcommand{\N}{\mathbb{N}}
\newcommand{\C}{\mathbb{C}}
\newcommand{\G}{\mathbb{G}}

\newcommand{\Z}{\mathbb{Z}}

\newcommand{\gA}{\mathfrak{A}}

\newcommand{\scE}{\mathscr{E}}
\newcommand{\scS}{\mathscr{S}}

\newcommand{\ronde}{\mathaccent'027}

\newcommand{\ev}{{\rm ev}}

\newcommand{\rM}{{\ronde M}}
\newcommand{\rG}{G_{\rM}^{\rM}} 
\newcommand{\rN}{{\ronde N}}
\newcommand{\rcN}{{\ronde \cN}}

\newcommand{\resp}{{\it resp.}\/ }
\newcommand{\id}{{\hbox{id}}}

\newcommand{\ind}{{{\mathrm{ind}}}}
\newcommand{\wind}{\widetilde\ind}

\newcommand{\Blup}{Blup}
\newcommand{\SBlup}{SBlup}
\newcommand{\DSBlup}{\cD SBlup}
\newcommand{\DNC}{DNC}
\newcommand{\wDNC}{{\widetilde{\DNC}}}

\newcommand{\ie}{{\it i.e.}\/ }
\newcommand{\eg}{{\it e.g.}\/ }
\newcommand{\cf}{{\it cf.}\/ }

\newcommand{\EBM}{E_{BM}}
\newcommand{\Bott}{Bott}

\newcommand{\Cn}{{\sf{C}}}

\newcommand{\PT}{Poisson-trace }
\newcommand{\BMT}{Boutet de Monvel type}

\newcommand{\boundarysymb}{\Sigma_{bound}^{C^*}(G,V)}
\newcommand{\Boundarysymb}{\Sigma_{bound}^{\Psi^*}(G,V)}
\newcommand{\fish}{{\includegraphics[width=3mm]{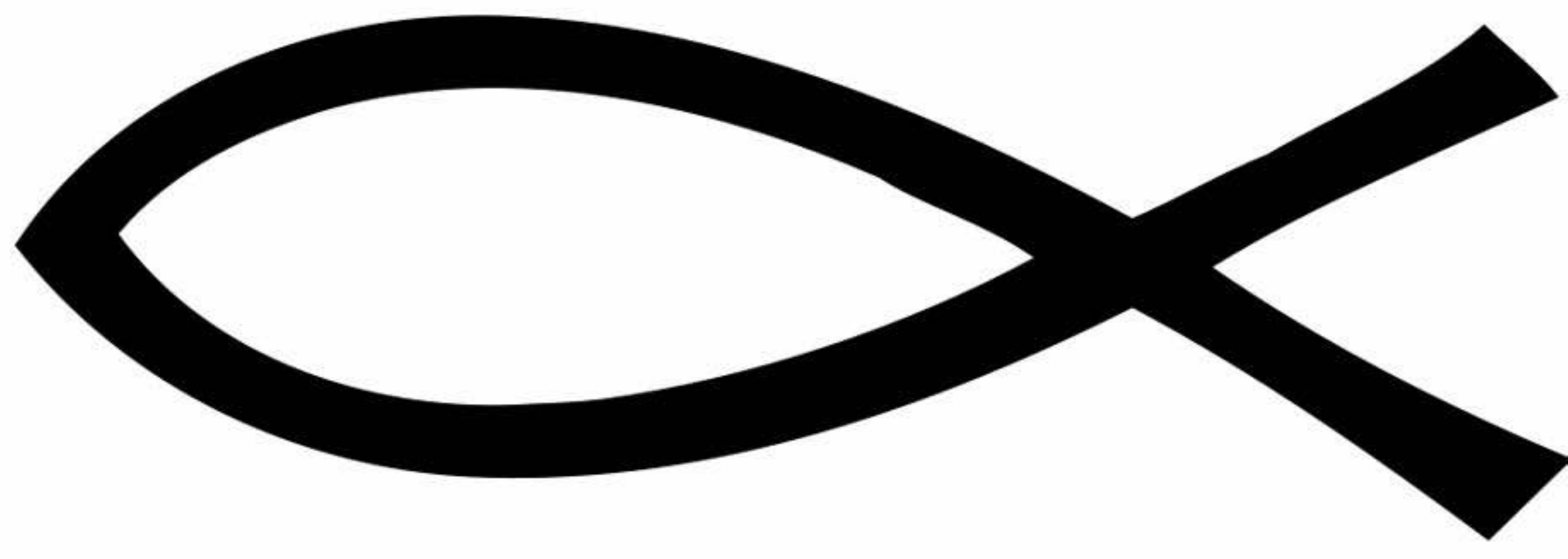}}}
\newcommand{\trace}{{\includegraphics[width=3mm]{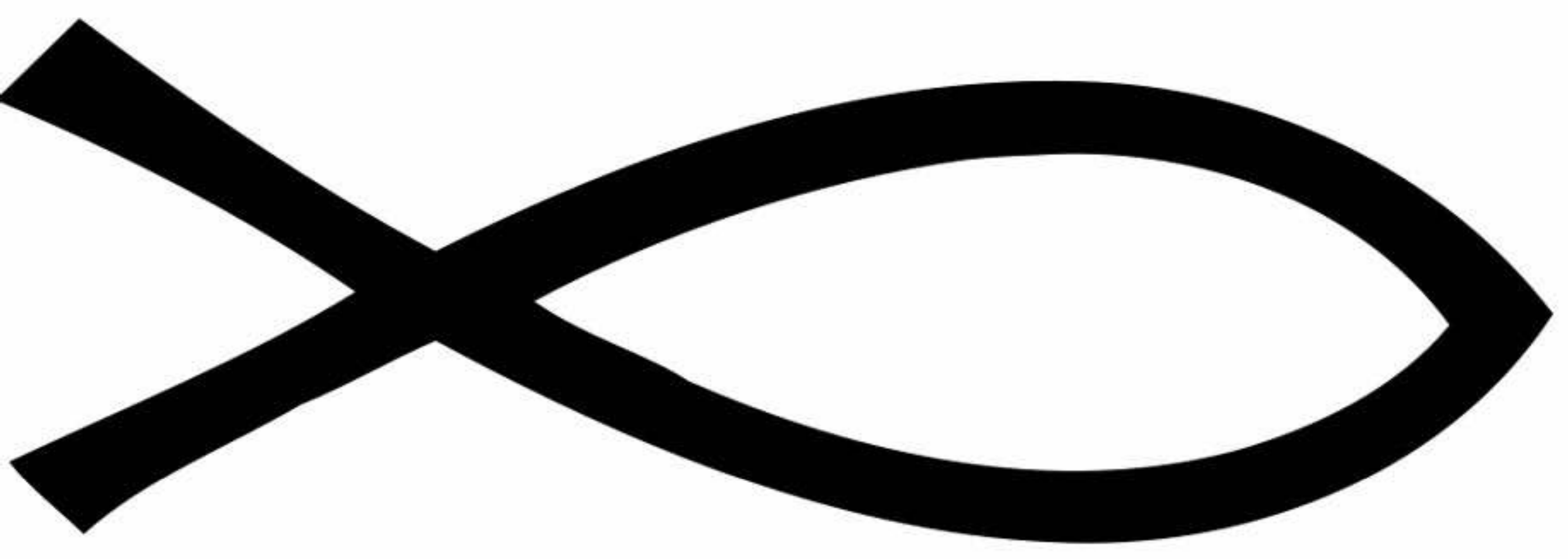}}}
\newcommand{\Green}{\rotatebox{90}{\small $\between$}} 
\newcommand{\pdo}{\psi}

\newcommand{\VBG}{$\cV\cB$ groupoid}
\newcommand{\VBGs}{$\cV\cB$ groupoids}
\newcommand{\Ept}{\scE _{PT}}
\newcommand{\Eptb}{\scE _{PT}^{V}}

\newcommand{\gag}{{_{ga}G}}

\reversemarginpar

\usepackage{color}

\everymath{\displaystyle}

\begin{document}

\begin{center}
{\renewcommand{\thefootnote}{\*}
{\Large\bf  Blowup constructions for Lie groupoids and a Boutet de Monvel type calculus}\footnote{{ The authors were partially supported by ANR-14-CE25-0012-01 (SINGSTAR).\\ AMS subject classification: Primary 58H05,  19K56. Secondary 58B34, 22A22, 46L80,19K35, 47L80.} }}
\setcounter{footnote}{0}

\bigskip

{\sc by Claire Debord and Georges Skandalis}

\end{center}

{\footnotesize
Universit\' e Clermont Auvergne
\vskip-2pt LMBP, UMR 6620 - CNRS
\vskip-2pt Campus des C\'ezeaux, 
\vskip-2pt 3, Place Vasarely 
\vskip-2pt TSA 60026  CS 60026 
\vskip-2pt 63178 Aubi\`ere cedex, France
\vskip-2pt claire.debord@math.univ-bpclermont.fr

\vskip 2pt Universit\'e Paris Diderot, Sorbonne Paris Cit\'e
\vskip-2pt  Sorbonne Universit\'es, UPMC Paris 06, CNRS, IMJ-PRG
\vskip-2pt  UFR de Math\'ematiques, {\sc CP} {\bf 7012} - B\^atiment Sophie Germain 
\vskip-2pt  5 rue Thomas Mann, 75205 Paris CEDEX 13, France
\vskip-2pt skandalis@math.univ-paris-diderot.fr
}

\vspace{1cm}

 \begin{abstract} 
We present natural and general ways of building Lie groupoids, by using the classical procedures of blowups and of deformations to the normal cone. Our constructions are seen to recover many known ones involved in index theory.  
The deformation and blowup groupoids obtained give rise to several extensions of $C^*$-algebras and to full index problems. We compute the corresponding K-theory maps. Finally, the blowup of a manifold sitting in a transverse way in the space of objects of a Lie groupoid leads to a calculus, quite similar to the Boutet de Monvel calculus for manifolds with boundary.
 \end{abstract}

\tableofcontents

\renewcommand\theenumi{\alph{enumi}}
\renewcommand\labelenumi{\rm {\theenumi})}
\renewcommand\theenumii{\roman{enumii}}

\section{Introduction}

Let $G\rightrightarrows M$ be a Lie groupoid.  The Lie groupoid $G$ comes with its natural family of elliptic pseudodifferential operators. For example\begin{itemize}
\item  if the groupoid $G$ is just the pair groupoid $M\times M$, the associated calculus is the ordinary (pseudo)differential calculus on $M$;
\item if the groupoid $G$ is a family groupoid $M\times_B M$ associated with a fibration $p:M\to B$, the associated  (pseudo)differential operators are families of operators acting on the fibers of $p$ (those of \cite{AtSing4});
\item if the groupoid $G$ is the holonomy groupoid of a foliation, the associated  (pseudo)differential operators are longitudinal operators as defined by Connes in \cite{ConnesLNM};

\item if the groupoid $G$ is the monodromy groupoid \ie the groupoid of homotopy classes (with fixed endpoints) of paths in a (compact) manifold $M$,  the associated  (pseudo)differential operators are the $\pi_1(M)$-invariant operators on the universal cover of $M$...
\end{itemize}
The groupoid $G$ defines therefore a class of partial differential equations. 

\medskip Our study will focus here on the corresponding index problems on $M$. The index takes place naturally in the $K$-theory of the $C^*$-algebra of $G$.

 \bigskip
 Let then $V$ be a submanifold of $M$. We will consider $V$ as bringing a singularity into the problem: it forces operators of $G$ to ``slow down'' near $V$, at least in the normal directions. Inside $V$, they should only propagate along a sub-Lie-groupoid $\Gamma\rightrightarrows V$ of $G$.
 
 This behavior is very nicely encoded by a groupoid $\SBlup_{r,s}(G,\Gamma)$ obtained by using a blow-up construction of the inclusion $\Gamma \to G$. 

\bigskip

The blowup construction and the deformation to the normal cone are well known constructions in algebraic geometry as well as in differential geometry. Let $X$ be a submanifold  of a manifold $Y$. Denote by $N_X^Y$ the normal bundle.
\begin{itemize}
\item The \emph{deformation to the normal cone} of $X$ in $Y$ is a smooth manifold $\DNC(Y,X)$ obtained by naturally gluing $N_X^Y\times \{0\}$ with $Y\times \R^*$.
\item The \emph{blowup} of $X$ in $Y$ is a smooth manifold $\Blup(Y,X)$ where  $X$ is inflated to the projective space $\bP N_X^Y$. It is obtained by gluing $Y\setminus X$ with $\bP N_X^Y$ in a natural way. We will mainly consider its variant the \emph{spherical blowup} $\SBlup(Y,X)$ (which is a manifold with boundary) in which  the sphere bundle $\bS N_X^Y$ replaces the projective bundle $\bP N_X^Y$. 
\end{itemize}

\bigskip The first use of deformation groupoids in connection with index theory appears in \cite{ConnesNCG}.  A. Connes showed there that the analytic index on a compact manifold $V$ can be described using a groupoid, called the ``tangent groupoid". This groupoid is obtained as a deformation to the normal cone of the diagonal inclusion of $V$ into the pair groupoid $V\times V$. 

\medskip  Since Connes' construction, deformation groupoids were used by many authors in various contexts. 
\begin{itemize}
\item This idea of Connes was extended in  \cite{HilsSkMorph} by considering the same construction of a deformation to the normal cone for smooth immersions which are groupoid morphisms. The groupoid obtained was used in order to define the wrong way functoriality for immersions of foliations (\cite[section 3]{HilsSkMorph}). An analogous construction for submersions of foliations was also given in a remark (\cite[remark 3.19]{HilsSkMorph}).

\item In \cite{MP, NWX} Monthubert-Pierrot and Nistor-Weinstein-Xu considered the deformation to the normal cone of the inclusion $G^{(0)}\to G$ of the space of units of a smooth groupoid $G$. This generalization of Connes' tangent groupoid was  called the adiabatic groupoid of \(G\) and denoted by $G_{ad}$. It was shown that this adiabatic groupoid still encodes the analytic index associated with $G$.

\item Many other important articles use this idea of deformation groupoids. We will briefly discuss some of them in the sequel of the paper. It is certainly out of the scope of the present paper to review them all... 
\end{itemize}

\bigskip Let us briefly present the objectives of our paper.

\subsection*{The groupoids $\DNC(G,\Gamma)$ and $\SBlup_{r,s}(G,\Gamma)$.} In the present paper, we give a systematic construction of deformations to the normal cone and define the blow-up deformations of groupoids. More precisely, we use the functoriality of these two constructions and note that any smooth subgroupoid $\Gamma\rightrightarrows V$ of a Lie groupoid $G\rightrightarrows M$ gives rise to a deformation to the normal cone Lie groupoid $\DNC(G,\Gamma)\rightrightarrows \DNC(M,V)$ and to a blowup  Lie groupoid $\Blup_{r,s}(G,\Gamma)\rightrightarrows \Blup(M,V)$ as well as its variant the spherical blowup  Lie groupoid $\SBlup_{r,s}(G,\Gamma)\rightrightarrows \SBlup(M,V)$.

\subsection*{Connecting maps and index maps} 

These groupoids give rise to connecting maps and to index problems that will be the main object of our study here. 

\paragraph{Connecting maps.} The (restriction $\DNC_+(G,\Gamma)$ to $\R_+$ of the) deformation groupoid $\DNC(G,\Gamma)$ is the disjoint union of an open subgroupoid $G\times \R_+^*$ and a closed subgroupoid $\cN_{\Gamma}^{G}\times \{0\}$.
 
The blowup groupoid $\SBlup_{r,s}(G,\Gamma)$ is the disjoint union of an open subgroupoid which is the restriction $G_{\rM}^{\rM}$ of $G$ to $\rM=M\setminus V$ and a boundary which is a groupoid $\cS N_{\Gamma}^G$ which is fibered over $\Gamma$, \ie a \VBG\ (in the sense of Pradines \cf \cite{Pra,Mack}).

This decomposition gives rise to exact sequences of $C^*$-algebras that we wish to ``compute'': 
$$0\longrightarrow C^*(G_{\rM}^{\rM})\longrightarrow C^*(\SBlup_{r,s}(G,\Gamma))\longrightarrow C^*(\cS N_\Gamma^G)\longrightarrow 0\eqno {E^{\partial}_{\SBlup}}$$
and
$$0\longrightarrow C^*(G\times \R_+^*)\longrightarrow C^*(\DNC_+(G,\Gamma))\longrightarrow C^*(\cN_{\Gamma}^{G})\longrightarrow 0\eqno {E^{\partial}_{\DNC_+}}$$

\paragraph{Full index maps.} Denote by $\Psi^*(\DNC_+(G,\Gamma))$ and $\Psi^*(\SBlup_{r,s}(G,\Gamma))$ the $C^*$-algebra of order $0$ pseudodifferential operators on the Lie groupoids $\DNC_+(G,\Gamma)$ and $\SBlup_{r,s}(G,\Gamma)$ respectively.  The above decomposition of groupoids give rise to extensions  of groupoid $C^*$-algebras of pseudodifferential type 
$$0\longrightarrow C^*(G_{\rM}^{\rM})\longrightarrow \Psi^*(\SBlup_{r,s}(G,\Gamma))\overset{\sigma_{full}}{{\relbar\joinrel\relbar\joinrel\longrightarrow}} \Sigma_{\SBlup}(G,\Gamma)\longrightarrow 0\eqno {E^{\wind}_{\SBlup}}$$
and
$$0\longrightarrow C^*(G\times \R_+^*)\longrightarrow \Psi^*(\DNC_+(G,\Gamma))\overset{\sigma_{full}}{\relbar\joinrel\relbar\joinrel\longrightarrow} \Sigma_{\DNC_+}(G,\Gamma)\longrightarrow 0\eqno {E^{\wind}_{\DNC_+}}$$
where $\Sigma_{\DNC_+}(G,\Gamma)$ and $\Sigma_{\SBlup}(G,\Gamma)$ are  called the \emph{full symbol algebra}, and the morphisms $\sigma_{full}$ the \emph{full symbol} maps.

\paragraph{The full symbol maps.} The full symbol algebras are naturally fibered products: $$\Sigma_{\SBlup}(G,\Gamma)=C(\bS^*\gA \SBlup_{r,s}(G,\Gamma))\times_{C(\bS^*\gA\cS N_\Gamma^G)}\Psi^*(\cS N_\Gamma^G)$$
and 
$$\Sigma_{\DNC_+}(G,\Gamma)=C(\bS^*\gA \DNC_+(G,\Gamma))\times_{C(\bS^*\gA \cN_\Gamma^G)}\Psi^*(\cN_\Gamma^G).$$

Thus, the full symbol maps have two components:\begin{itemize}
\item The usual commutative symbol of the groupoid. They are morphisms: $$\Psi^*(\DNC_+(G,\Gamma))\to C(\bS^*\gA \DNC_+(G,\Gamma))\ \ \hbox{and}\ \ \Psi^*(\SBlup_{r,s}(G,\Gamma))\to C(\bS^*\gA \SBlup_{r,s}(G,\Gamma)).$$ The commutative symbol takes its values in the algebra of continuous fonctions on the sphere bundle of the algebroid of the Lie groupoids (with boundary) $\DNC_+(G,\Gamma))$ and $\SBlup_{r,s}(G,\Gamma)$.
\item The restriction to the boundary: $$\sigma_\partial:\Psi^*(\SBlup_{r,s}(G,\Gamma))\to \Psi^*(\cS N_\Gamma^G) \mbox{ and } \Psi^*(\DNC_+(G,\Gamma))\to \Psi^*(\cN_\Gamma^G)\ . $$
\end{itemize}

\paragraph{Associated $KK$-elements.} Assume that the groupoid $\Gamma$ is amenable. Then the groupoids $\cN_{\Gamma}^G$ and $\cS N_{\Gamma}^{G}$ are also amenable, and exact sequences $E^{\partial}_{\SBlup}$ and $E^{\partial}_{\DNC_+}$ give rise to connecting elements $\partial_{\SBlup}^{G,\Gamma}\in KK^1(C^*(\cS N_{\Gamma}^G),C^*(G_{\rM}^{\rM}))$ and  $\partial_{\DNC_+}^{G,\Gamma}\in KK^1(C^*(\cN_{\Gamma}^{G}),C^*(G\times \R_+^*))$ (\cf \cite{Kasparov1980}). Also, the full symbol $C^*$-algebras $ \Sigma_{\SBlup}(G,\Gamma)$ and $\Sigma_{\DNC_+}(G,\Gamma)$ are nuclear and we also get $KK$-elements $\wind_{\SBlup}^{G,\Gamma}\in KK^1(\Sigma_{\SBlup}(G,\Gamma),C^*(G_{\rM}^{\rM}))$ and $\wind_{\DNC_+}^{G,\Gamma}\in KK^1(\Sigma_{\DNC_+}(G,\Gamma),C^*(G\times \R_+^*))$.

If $\Gamma$ is not amenable, these constructions can be carried over in $E$-theory (of maximal groupoid $C^*$-algebras).

\paragraph{Connes-Thom elements}
We will establish the following facts.\begin{enumerate}
\item  There is a natural Connes-Thom element $\beta \in KK^1(C^*(\SBlup_{r,s}(G,\Gamma)),C^*(\DNC_+(G,\Gamma)))$. This element restricts to very natural elements $\beta'\in KK^1(C^*(G_{\rM}^{\rM}),C^*(G\times \R_+^*))$ and $\beta''\in KK^1(C^*(\cS N_{\Gamma}^G),C^*(\cN_{\Gamma}^{G}))$. 

These elements extend to elements  $\beta_\Psi \in KK^1(\Psi^*(\SBlup_{r,s}(G,\Gamma)),\Psi^*(\DNC_+(G,\Gamma)))$ and $\beta_\Sigma \in KK^1(\Sigma_{\SBlup}(G,\Gamma)),\Sigma_{\DNC_+}(G,\Gamma)))$.

We have $\partial_{\SBlup}^{G,\Gamma}\otimes \beta'=-\beta''\otimes \partial_{\DNC_+}^{G,\Gamma}$ (\cf facts \ref{fact1} and \ref{fact2}) and $\wind_{\SBlup}^{G,\Gamma}\otimes \beta'=-\beta_\Sigma \otimes \wind_{\DNC_+}^{G,\Gamma}$ (fact \ref{fact7}).

\item If $M\setminus V$ meets all the orbits of $G$, then $\beta'$ is $KK$-invertible. Therefore, in that case, $\partial_{\DNC_+}^{G,\Gamma}$ determines  $\partial_{\SBlup}^{G,\Gamma}$ and $\wind_{\DNC_+}^{G,\Gamma}$ determines  $\wind_{\SBlup}^{G,\Gamma}$.

\item We will say that $V$ is \emph{$\gA G$-small} if the transverse action of $G$ on $V$ is nowhere $0$, \ie if for every $x\in V$, the image by the anchor of the algebroid $\gA G$ of $G$ is not contained in $T_xV$ (\cf definition \ref{notation6}). In that case, $\beta'$, $\beta''$, $\beta_\Sigma$ are $KK$-invertible: the connecting elements $\partial_{\DNC_+}^{G,\Gamma}$ and  $\partial_{\SBlup}^{G,\Gamma}$ determine each other, and the full index maps $\wind_{\DNC_+}^{G,\Gamma}$ and $\wind_{\SBlup}^{G,\Gamma}$  determine each other
\end{enumerate}

\paragraph{Computation.}
If $\Gamma=V$, then $C^*(\cN_{V}^{G})$ is $KK$-equivalent to $C_0(N_V^G)$ using a Connes-Thom isomorphism and the element  $\partial^{G,\Gamma}_{\DNC_+}$  is the Kasparov product of the inclusion of $N_V^G$ in the algebroid $\gA G=N_M^G$ of $G$ (using a tubular neighborhood) and the index element $\ind_G\in KK(C_0(\gA^* G),C^*(G))$ of the groupoid $G$ (prop. \ref{ComputeDNC}.\ref{noiseuse}). Of course, if $V$ is $\gA G$-small, we obtain the analogous result for $\partial_{\SBlup}^{G,\Gamma}$.

The computation of the corresponding full index is also obtained in the same way in prop.  \ref{ComputeDNC}.\ref{ComputeDNCIndex}).

\paragraph{Full index and relative $K$-theory.}
Assume that $\Gamma$ is just a $\gA G$-small submanifold $V$ of $M$. We will actually obtain a finer construction by using relative $K$-theory. It is a general fact that relative $K$-theory gives more precise index theorems than connecting maps (\cf \eg \cite{AtSing1, SchickCIRM, MSS1, MSS2, AndrSk3}). In particular, the relative $K$-theory point of view allows to take into account symbols from a vector bundle to another one. 

Let $\psi:C_0(\SBlup(M,V))\to \Psi^*(\SBlup_{r,s}(G,\Gamma))$ be the natural inclusion and consider the morphism $\mu_{\SBlup}=\sigma_{full}\circ \psi:C_0(\SBlup(M,V))\to \Sigma_{\SBlup}(G,V)$. The relative index theorem computes the map $\ind _{rel}:K_*(\mu_{\SBlup}\circ \psi )\to K_*(C^*(\rG))$:\begin{itemize}
\item  the relative $K$-group $K_*(\mu_{\SBlup})$ is canonically isomorphic to $K_{*}(C_0(\gA^*\rG))$; 
\item under this isomorphism $\ind_{rel}$ identifies with the index map of the groupoid $\rG$.
\end{itemize}
We prove an analogous result for the morphism $\mu_{\DNC}:C_0(\DNC_+(M,V)\to \Sigma_{\DNC_+}(G,V)$

\medskip
In fact, most of the computations involved come from a quite more general situation studied in section \ref{ne}. There we consider a groupoid $G$ and a partition of $G^{(0)}$ into an open and a closed saturated subset and study the connecting elements of the associated exact sequences.

\subsection*{A Boutet de Monvel type calculus.}
Let $H$ be a Lie groupoid. In \cite{DS1}, extending ideas of Aastrup, Melo, Monthubert and Schrohe \cite{AMMS}, we studied the \emph{gauge adiabatic groupoid} $H_{ga}$: the crossed product of the adiabatic groupoid of $H$ by the natural action of $\R_+^*$. We constructed a bimodule $\scE_H$ giving a Morita equivalence between the algebra of order $0$ pseudodifferential operators on $H$ and a natural ideal in the convolution $C^*$-algebra $C^*(H_{ga})$ of this gauge adiabatic groupoid. 

The gauge adiabatic groupoid $H_{ga}$ is in fact  a blowup groupoid, namely $\SBlup_{r,s}(H\times (\R\times \R),H^{(0)})$ (restricted to the clopen subset $H^{(0)}\times \R_+$ of $\SBlup(H^{(0)}\times \R,H^{(0)})=H^{(0)}\times (\R_-\sqcup \R_+)$). 

\medskip
Let now $G\rightrightarrows M$ be a Lie groupoid and let $V$ be a submanifold $M$ which is \emph{transverse} to the action of $G$ (see def. \ref{E2Kdix}). We construct a \PT bimodule: it is a $C^*(\SBlup_{r,s}(G,V)),\Psi^*(G_V^V)$ bimodule  $\Ept (G,V)$, which is a full $\Psi^*(G_V^V)$ Hilbert module.
When $G$ is the direct product of $G_V^V$ with the pair groupoid $\R\times \R$ the \PT bimodule coincide with $\scE_{G_V^V}$ constructed in \cite{DS1}. In the general case, thanks to a convenient (spherical) blowup construction, we construct a linking space between the groupoids $\SBlup_{r,s}(G,V)$ and $(G_V^V)_{ga}=\SBlup_{r,s}(G_V^V\times (\R\times \R),V)$. This linking space defines a $C^*(\SBlup_{r,s}(G,V)),C^*((G_V^V)_{ga})$-bimodule $\scE(G,V)$ which is a Morita equivalence of groupoids when $V$ meets all the orbits of $G$. The \PT-bimodule is then the composition of $\scE_{G_V^V}$ with $\scE(G,V)$.

Denote by $\Psi^*_{BM}(G;V)$ the \emph{Boutet de Monvel type} algebra consisting of matrices $R=\begin{pmatrix}
 \Phi&P\\
 T&Q
\end{pmatrix}$ with $\Phi\in \Psi^*(\SBlup_{r,s}(G,V)),\ P\in \Ept(G,V),\ T\in \Ept^*(G,V)$ and $Q\in \Psi^*(G_V^V)$, and $C^*_{BM}(G;V)$ its ideal - where $\Phi\in C^*(\SBlup_{r,s}(G,V))$.  This algebra has obvious similarities with the one involved in the Boutet de Monvel calculus for manifold with boundary \cite{MB1}. We will examine its relationship with these two algebras in a forthcoming paper.

We still have two natural symbol maps: the classical symbol $\sigma_c:\Psi^*_{BM}(G,V)\to C(\bS^*\gA G)$ given by $\sigma_c\begin{pmatrix}
 \Phi&P\\
 T&Q
\end{pmatrix}=\sigma_c(\Phi)$ and the boundary symbol $r_V$ which is restriction to the boundary.

We have an exact sequence:
$$0\to C^*(G_{\rM\coprod V}^{\rM\sqcup V})\to \Psi^*_{BM}(G;V)\overset{\sigma_{BM}}{\relbar\joinrel\relbar\joinrel\longrightarrow} \Sigma_{BM}(G,V) \to 0$$
where $\Sigma_{BM}(G,V)=\Psi^*_{BM}(G;V)/C^*(G_{\rM\sqcup V}^{\rM\sqcup V})$ and  $\sigma_{BM}$ is defined using both $\sigma_c$ and $r_V$.

We may note that $\Psi^*(\SBlup_{r,s}(G,V))$  identifies with the full hereditary subalgebra of $\Psi^*_{BM}(G,V)$ consisting of elements of the form $ \begin{pmatrix}
 \Phi&0\\
 0&0
\end{pmatrix} $. We thus obtain Boutet de Monvel type index theorems for the connecting map of this exact sequence  - as well as for the corresponding relative $K$-theory.

\bigskip
The paper is organized as follows:
\begin{itemize}
\item In section 2 we recall some classical facts, constructions and notation involving groupoids.
\item Section 3 is devoted to the description and computation of various $KK$-elements associated with groupoid $C^*$-algebras. The first and second type are encountered in the situation where a Lie groupoid $G$ can be cut in two pieces $G=G\vert_W\sqcup G\vert_ F$ where $W$ is an open saturated subset of the units $G^{(0)}$ and $F=G^{(0)}\setminus W$. They are respectively built from exact sequences of $C^*$-agebras of the form:
$$0\longrightarrow C^*(G_{W})\longrightarrow C^*(G)\longrightarrow C^*(G_F)\longrightarrow 0\eqno E_\partial$$ and $$
0\longrightarrow C^*(G_{W})\longrightarrow  \Psi^*(G)\longrightarrow \Sigma^W(G)\longrightarrow  0\eqno E_{\widetilde {\ind}_{full}}$$
The other $KK$-elements are Connes-Thom type elements arising when a $\R$-action is involved in those situations.
\item In section 4 we review two geometric constructions: deformation to the normal cone and blowup, and their functorial properties. 
\item In section 5, using this functoriality, we study deformation to the normal cone and blowup in the Lie groupoid context. We present examples which recover groupoids constructed previously by several authors.
\item In section 6, applying the results obtained in section 3, we compute the connecting maps and index maps of the groupoids constructed in section 5.
\item In section 7, we describe the above mentioned Boutet de Monvel type calculus.
\item Finally, in the appendix, we make a few remarks on \VBGs. In particular, we give a presentation of the dual \VBG\ $E^*$ of a \VBG\ $E$ and show that $C^*(E)$ and $C^*(E^*)$ are isomorphic. 
\item Our constructions involved a large amount of notation, that we tried to choose as coherent as possible. We found it however helpful to list several items in an index at the end of the paper.
\end{itemize} 

\medskip \begin{notation}\label{notation1.3}
We will use the following notation:

\begin{itemize}
\item If $E$ is a real vector bundle over a manifold (or over a locally compact space) $M$, the corresponding projective bundle \(\bP (E) \) is the bundle over \(M\) whose fiber over a point \(x\) of \(M\) is the projective space \(\bP(E_x)\). The bundle  \(\bP (E) \) is simply the quotient of \(E\setminus M\) by the natural action of \(\R^* \) by dilation. The quotient of  \(E\setminus M\) under the  action of \(\R_+^* \) by dilation is the (total space of the) sphere bundle \(\bS (E) \).

\item If $E$ is a real vector bundle over a manifold (or a locally compact space) $M$, we will denote by $\ronde B^*E$, $B^*E$ and $\bS^*E$ the total spaces of the fiber bundles of open balls, closed ball and spheres of the dual vector bundle $E^*$ of $E$. If $F\subset M$ is a closed subset of $M$, we will denote by $B^*_FE$ the quotient of $B^*E$ where we identify two points $(x,\xi)$ and $(x,\eta)$ for $x\in F$ and $\bS^*_FE$ the image of $\bS^*E$ in $B^*_FE$. 
\end{itemize}
\end{notation}

\medskip \textbf{Acknowledgements.} We would like to thank Vito Zenobi for his careful reading and for pointing out quite a few typos in an earlier version of the manuscript.

\nomenclature[B, 02]{$\bP (E)$, $\bS (E)$ }{The projective and sphere bundles associated to a real vector bundle $E$ over $M$, whose fiber over $x\in M$ are respectively the projective space $\bP(E_x)$ and the sphere $\bS(E_x)$}
\nomenclature[B, 04]{$\ronde B^*E$, $B^*E$, $\bS^*E$}{The total spaces of the fiber bundles of open balls, closed ball and spheres of the dual vector bundle $E^*$ of $E$}
\nomenclature[B, 05]{$B^*_FE$}{The quotient of $B^*E$ where we identify two points $(x,\xi)$ and $(x,\eta)$ for $x\in F$, $F$ being a closed subset of the zero section $M$ of the bundle $E$}
\nomenclature[B, 06]{$\bS^*_FE$}{The image of $\bS^*E$ in $B^*_FE$}

\section{Some quite classical constructions involving groupoids}

\subsection{Some classical notation}

Let $G$ be a Lie groupoid. We denote by $G^{(0)}$ its space of objects and $r,s:G\to G^{(0)}$ the range and source maps. 

\nomenclature[G, 01]{$G\overset{r,s}{\rightrightarrows} G^{(0)}$}{A Lie groupoid with source  $s$, range $r$ and space of units $G^{(0)}$}

The algebroid of $G$ is denoted by $\gA G$, and its anchor by $\natural:\gA G\to TG^{(0)}$ its anchor. Recall that (the total space of) $\gA G$ is the normal bundle  $N_{G^{(0)}}^G$ and the anchor map is induced by \((dr-ds)\).

\nomenclature[G, 02]{ $\gA G$}{The Lie algebroid of the groupoid $G$}

We denote by $\gA^*G$ the dual bundle of $\gA G$ and by $\bS^* \gA G$ the sphere bundle of $\gA^*G$.

\begin{itemize}
\item We denote by $C^*(G)$  its (full or reduced) $C^*$-algebra. We denote by $\Psi^*(G)$ the $C^*$-algebra of order $\le0$ (classical, \ie polyhomogeneous) pseudodifferential operators on $G$ vanishing at infinity on $G^{(0)}$ (if $G^{(0)}$ is not compact). More precisely, it is the norm closure in the multiplier algebra of $C^*(G)$ of the algebra of classical pseudodifferential operators on $G$ with compact support in $G$.

\nomenclature[H, 01]{$C^*(G)$}{The (either maximal or reduced) $C^*$-algebra of the groupoid $G$}
\nomenclature[H, 02]{$\Psi^*(G)$}{The $C^*$-algebra of order $\le0$ pseudodifferential operators on $G$ vanishing at infinity on $G^{(0)}$}

We have an exact sequence of $C^*$-algebras $0\to C^*(G)\to \Psi^*(G)\to C_0(\bS^* \gA G)\to 0$.

As mentioned in the introduction, our constructions involve connecting maps associated to short exact sequences of groupoid $C^*$-algebras, therefore they make sens a priori for the full $C^*$-algebras, and give rise to $E$-theory elements (\cite{Connes-Higson}). Nevertheless, in many interesting situations, the quotient $C^*$-algebra will be the $C^*$-algebra of an amenable groupoid, thus the corresponding exact sequence is semi-split as well as for the reduced and the full $C^*$-algebras and it defines moreover a $KK$-element. In these situations $C^*(G)$ may either be the reduced or the full $C^*$-algebra of the groupoid $G$ and we have preferred to leave the choice to the reader.

\item For any maps $f:A\to G^{(0)}$ and $g:B\to G^{(0)}$, define $$G^f=\{(a,x)\in A\times G;\ r(x)=f(a)\} , \ G_g=\{(x,b)\in G \times B;\ s(x)=g(b)\}$$ and $$G^f_g=\{(a,x,b)\in A\times G\times B;\ r(x)=f(a),\ s(x)=g(b)\}\ .$$  In particular for $A,B\subset G^{(0)}$, we put $G^A=\{x\in G;\ r(x)\in A\}$ and $G_A=\{x\in G;\ s(x)\in A\}$; we also put $G_A^B=G_A\cap G^B$.

Notice that \(A\) is a \emph{saturated} subset of $G^{(0)}$ if and only if $G_A=G^A=G_A^A$.

\nomenclature[G, 05]{$G^f$, $G_g$, $G^f_g$ }{If $f:A\to G^{(0)}$ and $g:B\to G^{(0)}$ are maps, $G^f=\{(a,x)\in A\times G;\ r(x)=f(a)\}$, $G_g=\{(x,b)\in G \times B;\ s(x)=g(b)\}$ and $G^f_g=G^f\cap G_g$}
\nomenclature[G, 03]{$G^A$, $G_B$, $G^A_B$}{If $A$ and $B$ are subsets of $G^{(0)}$, $G^A=\{x\in G;\ r(x)\in A\}$,  $G_B=\{x\in G;\ s(x)\in B\}$ and $G_A^B=G_A\cap G^B$ }

\item We denote by $G_{ad}$ the adiabatic groupoid of $G$, (\cite{MP,NWX}), it is obtained by using the deformation to the normal cone construction for the inclusion of $G^{(0)}$ as a Lie subgroupoid of $G$ (see section \ref{red} and \ref{sefaire} below for a complete description). Thus: 
\[G_{ad}=G\times \R^* \cup \gA G\times \{0\} \rightrightarrows G^{(0)}\times \R \ .\]  
If $X$ is a locally closed saturated subset of $M\times \R$, we will denote sometimes by $G_{ad}(X)$ the restriction $(G_{ad})_X^X$ of $G_{ad}$ to $X$: it is a locally compact groupoid.

In the sequel of the paper, we let  $G_{ad}^{[0,1]}=G_{ad}(G^{(0)}\times[0,1])$ and  $G_{ad}^{[0,1)}=G_{ad}(G^{(0)}\times[0,1))$ \ie \[G_{ad}^{[0,1]}=G\times (0,1] \cup \gA G\times \{0\} \rightrightarrows G^{(0)}\times[0,1] \ \ \ \hbox{and}\ \ \ G_{ad}^{[0,1)}=G\times (0,1) \cup \gA G\times \{0\} \rightrightarrows G^{(0)}\times[0,1).\] 

\end{itemize}

\nomenclature[B]{$N_V^M$}{The normal bundle of a submanifold $V$ of a manifold $M$}

\nomenclature[G, 07]{ $G_{ad}$, $G_{ad}^{[0,1]}$, $G_{ad}^{[0,1)}$}{The adiabatic groupoid of $G$ and its restriction respectively to $G^{(0)}\times[0,1] $ and to $G^{(0)}\times[0,1)$ }
\nomenclature[G, 08]{$G_{ad}(X)$}{The restriction of $G_{ad}$ to a locally closed saturated subset $X$ of $G^{(0)}\times [0,1]$}

\smallskip
\begin{remark}\label{Mnotmanifold}
Many manifolds and groupoids that occur in our constructions have boundaries or corners. In fact all the groupoids we consider sit naturally inside Lie groupoids \emph{without boundaries} as restrictions to closed \emph{saturated} subsets. This means that we consider subgroupoids $G_V^V=G_V$ of a Lie groupoid $G\overset{r,s}{\rightrightarrows} G^{(0)}$ where $V$ is a closed subset of $G^{(0)}$. Such groupoids, have a natural algebroid, adiabatic deformation, pseudodifferential calculus, \emph{etc.} that are restrictions to $V$ and $G_V$ of the corresponding objects on $G^{(0)}$ and $G$. We chose to give our definitions and constructions for Lie groupoids for the clarity of the exposition. The case of a longitudinally smooth groupoid over a manifold with corners is a straightforward generalization using a convenient restriction.
\end{remark}

\subsection{Transversality  and Morita equivalence}

Let us recall the following definition (see \eg \cite{Zenobi} for details):

 \begin{definition}\label{E2Kdix}
Let $G\overset{r,s}{\rightrightarrows} M$ be a Lie groupoid with set of objects $G^{(0)}=M$. Let $V$ be a manifold. A smooth map $f:V\to M$ is said to be  \emph{transverse} to (the action of the groupoid) $G$  if for every $x\in V$, $df_x(T_xV)+\natural _{f(x)}\gA _{f(x)}G=T_{f(x)}M$.
\\
An equivalent condition is that the map \((\gamma,y) \mapsto r(\gamma)\)  defined on the fibered product \(G_f=G\underset{s,f}{\times} V\) is a submersion $G_f\to M$.
\\
A submanifold $V$ of $M$ is \emph{transverse} to  $G$ if the inclusion $V\to M$ is transverse to $G$ - equivalently, if for every $x\in V$, the composition $q_x=p_x\circ \natural_x:\gA_x G\to (N_V^M)_x=T_xM/T_xV$ is onto.
\end{definition}

\begin{remark}\label{soeur}
Let $V$ be a (locally) closed submanifold of $M$ transverse to a groupoid $G\overset{r,s}{\rightrightarrows} M$. Denote by $N_V^M$ the (total space) of the normal bundle of $V$ in $M$. Upon arguing locally, we can assume that $V$ is compact. 

By the transversality assumption the anchor $\natural:\gA G_{|V}\to TM_{|V}$ induces a surjective bundle morphism $\gA G_{|V}\to N_V^M$. Choosing a subbundle $W'$ of the restriction $\gA G_{|V}$ such that $W'\to N_V^M$ is an isomorphism and using an exponential map, we thus obtain a submanifold $W\subset G$  such that $r:W\to M$ is a diffeomorphism onto an open neighborhood of $V$ in $M$ and $s$ is a submersion from $W$ onto $V$.   Replacing $W$ by a an open subspace, we may assume that $r(W)$ is a tubular neighborhood of $V$ in $M$, diffeomorphic to $N_V^M$. The map $W\times _{V}G_V^V\times _VW\to G$ defined by $(\gamma_1,\gamma_2,\gamma_3)\mapsto \gamma_1\circ\gamma_2\circ\gamma_3^{-1}$ is a diffeomorphism and a groupoid isomorphism from the pull back groupoid (see next section) $(G_V^V)_s^s=W\times _{V}G_V^V\times _VW$ onto the open subgroupoid $G_{r(W)}^{r(W)}$ of $G$.

\end{remark}

\subsubsection*{Pull back}

If $f:V\to M$ is transverse to a Lie groupoid $G\overset{r,s}{\rightrightarrows} M$, then the \emph{pull back groupoid} $G_f^f$ is naturally a Lie groupoid (a submanifold of $V\times G\times V$).

\smallskip
If $f_i:V_i\to M$ are transverse to $G$ (for $i=1,2$) then we obtain a Lie groupoid $G_{f_1\sqcup f_2}^{f_1\sqcup f_2}\rightrightarrows V_1\sqcup V_2$. The \emph{linking manifold} $G_{f_2}^{f_1}$ is a clopen submanifold. We denote by $C^*(G_{f_2}^{f_1})$ the closure in $C^*(G_{f_1\sqcup f_2}^{f_1\sqcup f_2})$ of  the space of functions (half densities) with support in $G_{f_2}^{f_1}$; it is a $C^*(G_{f_1}^{f_1})-C^*(G_{f_2}^{f_2})$ bimodule. 

\begin{fact}\label{otum}
The bimodule $C^*(G_{f_2}^{f_1})$ is full if all the $G$ orbits meeting $f_2(V_2)$ meet also $f_1(V_1)$.
\end{fact}

\subsubsection*{Morita equivalence}\label{Moreq1}

Two Lie groupoids $G_1\overset{r,s}{\rightrightarrows} M_1$ and $G_2\overset{r,s}{\rightrightarrows} M_2$ are \emph{Morita equivalent} if there exists a groupoid $G\overset{r,s}{\rightrightarrows} M$ and smooth maps $f_i:M_i\to M$ transverse to $G$ such that the pull back groupoids $G_{f_i}^{f_i}$ identify to $G_i$ and $f_i(M_i)$ meets all the orbits of $G$.

Equivalently, a Morita equivalence is given by a linking manifold $X$ with extra data: surjective smooth submersions $r:X\to G_1^{(0)}$ and $s:X\to G_2^{(0)}$ and compositions $G_1\times_{s,r}X\to X$,  $X\times_{s,r} G_2\to X$, $X\times_{r,r} X\to G_2$ and $X\times_{s,s} X\to G_1$ with natural associativity conditions (see \cite{MRW} for details). In the above situation, $X$ is the manifold $G_{f_2}^{f_1}$ and the extra data are the range and source maps and the composition rules of the groupoid $G_{f_1\sqcup f_2}^{f_1\sqcup f_2}\rightrightarrows M_1\sqcup M_2$ (see \cite{MRW}).

If the map $r:X\to G_1^{(0)}$ is surjective but $s:X\to G_2^{(0)}$ is not necessarily surjective, then $G_1$ is Morita equivalent to the restriction of $G_2$ to the open saturated subspace $s(X)$. We say that $G_1$ is \emph{sub-Morita} equivalent to $G_2$.

\subsection{Semi-direct products}\label{croise}

\begin{description}\item[Action of a groupoid on a space.] Recall  that an action of a groupoid $G \overset {r,s}\rightrightarrows  G ^{(0)}$ on a space $V$ is given by a map $p:V\to G ^{(0)}$ and the action $G \times_{s,p}V\to V$ denoted by $(g,x)\mapsto g.x$ with the requirements $p(g.x)=r(g),\ g.(h.x)=(gh).x$ and $u.x=x$ if $u=p(x)$.

In that case, we may form the crossed product groupoid $V\rtimes G$: 
\begin{itemize}
\item as a set  $V\rtimes G $ is the fibered product $V\times _{p,r} G $;
\item the unit space $(V\rtimes G)^{(0)}$ is $V$. The range and source maps are $r(x,g)=x$ and $s(x,g)=g^{-1}.x$;
\item the composition is given by $(x,g)(y,h)=(x,gh)$ (with $g.y=x$).
\end{itemize} 
If $G $ is a Lie groupoid, $M$ is a manifold and if all the maps defined are smooth and $p$ is a submersion, then $V\rtimes G $ is a Lie groupoid.

\item[Action of a group on a groupoid.] Let $\Gamma$ be a Lie group acting on a Lie groupoid $G\overset{r,s}{\rightrightarrows} M$ by Lie groupoid automorphisms. The set $G\times \Gamma$ is naturally a Lie groupoid $G\rtimes \Gamma \overset{r_{\rtimes},s_{\rtimes}}{ \rightrightarrows} M$ we put $r_{\rtimes}(g,\gamma)=r(g)$, $s_{\rtimes}(g,\gamma)=\gamma^{-1}(s(g))$ and, when $(g_1,\gamma_1)$ and $(g_2,\gamma_2)$ are composable, their product is $(g_1,\gamma_1)(g_2,\gamma_2)=(g_1\, \gamma_1(g_2),\gamma_1\gamma_2)$.

Note that the semi-direct product groupoid $G\rtimes \Gamma$ is canonically isomorphic to the quotient $\G/\Gamma$ of the product $\G=G\times (\Gamma\times \Gamma)$ of $G$ by the pair groupoid $\Gamma \times \Gamma$ where the $\Gamma$ action on $\G$ is the diagonal one: $\gamma\cdot (g,\gamma_1,\gamma_2)=(\gamma (g),\gamma^{-1}\gamma_1,\gamma^{-1} \gamma_2)$.

\item[Free and proper action of a group on a groupoid.]
When the action of $\Gamma$ on $G$ (and therefore on its closed subset $M=G^{(0)}$) is free and proper, we may define the quotient groupoid $G/\Gamma \overset{r,s}{\rightrightarrows} M/\Gamma$.

In that case, the groupoid $G/\Gamma$ acts on $M$ and the groupoid $G$ identifies with the action groupoid $M\rtimes (G/\Gamma)$. Indeed, let $p:M\to M/\Gamma$ and $q:G\to G/\Gamma$ be the quotient maps. If $x\in M$ and $h\in G/\Gamma$ are such that $s(h)=p(x)$, then there exists a unique $g\in G$ such that $q(g)=h$ and $s(g)=x$; we put then $h.x=r(g)$. It is then immediate that $\varphi:G\to M\times _{p,r}(G/\Gamma) $ given by $\varphi (g)=(r(g),q(g))$ is a groupoid isomorphism.

The groupoid $G/\Gamma$ is Morita equivalent to $G\rtimes \Gamma$: indeed one easily identifies $G\rtimes \Gamma$ with the pull back groupoid $(G/\Gamma)_q^q$ where $q:M\to M/\Gamma$ is the quotient map.

\smallskip Note also that in this situation the action of \(\Gamma\) on \(G\) leads to an action of $\Gamma$ on the Lie algebroid \(\gA G\) and \(\gA(G/\Gamma)\) identifies with \(\gA G /\Gamma\).

\begin{remark}\label{RemarqueToutePropre}
As the Lie groupoids we are considering need not be Hausdorff, the properness condition has to be relaxed. We will just assume that the action is \emph{locally proper}, \ie that every point in $G$ has a $\Gamma$-invariant neighborhood on which the action of $\Gamma$ is proper.
\end{remark}

\item[{Action of a groupoid on a groupoid.}] Recall  that an action of a groupoid $G \overset {r,s}\rightrightarrows  G ^{(0)}$ on a groupoid $H\overset {r_H,s_H}\rightrightarrows H^{(0)}$ is  by groupoid automorphisms (\cf \cite{Brown}) if $G $ acts on $H^{(0)}$ through a map $p_0:H^{(0)}\to G ^{(0)}$, we have $p=p_0\circ r_H=p_0\circ s_H$ and $g.(xy)=(g .x)(g.y)$.

In that case, we may form the crossed product groupoid $H\rtimes G =\G$: 
\begin{itemize}
\item as a set  $H\rtimes G $ is the fibered product $H\times _{p,r} G $;
\item the unit space $\G^{(0)}$ of $\G=H\rtimes G $ is $H^{(0)}$. The range and source maps are $r_{\G}(x,g)=r_H(x)$ and $s_{\G}(x,g)=g^{-1}.s_H(x)$;
\item the composition is given by $(x,g)(y,h)=(x(g.y),gh)$.
\end{itemize} 
If $G $ and $H$ are Lie groupoids and if all the maps defined are smooth and $p$ is a submersion, then $\G=H\rtimes G $ is a Lie groupoid.
\end{description}

\subsection{Index maps for Lie groupoids}

Recall (\cf \cite{MP, NWX}) that if $G$ is any Lie groupoid, the index map is an element in $KK(C_0(\gA^*G),C^*(G))$ which can be constructed thanks to the adiabatic groupoid $G_{ad}^{[0,1]}$ of $G$  as $$\ind_{G}=[ev_0]^{-1}\otimes [ev_1]$$ where  $$ev_0:C^*(G_{ad}^{[0,1]})\to C^*(G_{ad}(0))\simeq C_0(\gA^*G) \  \mbox{ and } \ ev_1:C^*(G_{ad}^{[0,1]})\to C^*(G_{ad}(1))\simeq C^*(G)$$ are the evaluation morphisms (recall that $[ev_0]$ is invertible).

\nomenclature[I, 01]{$\ind_{G}$}{The $KK$-element $[ev_0]^{-1}\otimes [ev_1]$, which belongs to $KK(C_0(\gA^*G),C^*(G))$, 
associated to the deformation groupoid $G_{ad}^{[0,1]}=G\times (0,1] \cup \gA(G)\times \{0\} \rightrightarrows G^{(0)}\times[0,1]$}
\nomenclature[I, 02]{$\wind_G$}{The connecting element, which belongs to $KK^1(C(S^*\gA \cG),C^*(\cG))$ associated to the short exact sequence $0\to C^*(G)\to \Psi^*(G)\to C(S^*\gA G)\to 0$}

It follows quite immediately that the element $\wind_G\in KK^1(C(\bS^*\gA G),C^*(G))$ corresponding to the pseudodifferential exact sequence $$0\to C^*(G)\to \Psi^*(G)\to C(\bS^*\gA G)\to 0 \eqno E_{\Psi^*(G)}$$ is the composition $\wind_G=\ind_G\otimes q_{\gA^*G}$ where $q_{\gA^*G}\in KK^1(C(\bS^*\gA G),C_0(\gA^* G))$ corresponds to the pseudodifferential exact sequence for $\gA G$ which is $$0\to C_0(\gA^* G)\to C(B^* \gA G)\to C(\bS^* \gA G)\to 0\  \eqno  E_{\Psi^*(\gA G)} $$ This connecting element is immediately seen to be the element of $KK(C_0(\bS^*\gA G\times \R_+^*),C_0(\gA^* G))$ associated to the inclusion of $\bS^*\gA G\times \R_+^*$ as the open subset $\gA^* G\setminus G^{(0)}$ - where $ G^{(0)}$ sits in  $\gA^* G$ as the zero section.

\section{Remarks on exact sequences, Connes-Thom elements, connecting maps and index maps}\label{ne}

The first part of this section is a brief reminder of some quite classical facts about connecting elements associated to short exact sequences of $C^*$-algebras. 

The second part is crucial for  our main results of section 6: given a Lie groupoid and an open saturated subset of its unit space, we consider connecting maps and full index maps, compare them, compute them in some cases... In particular, we study a Fredholm realizability problem generalizing works of Albin and Melrose (\cite{AlMel1, AlMel2}) and study index maps using relative $K$-theory.

In the last part we study a proper action of $\R_+^*$ on a Lie groupoid $G$ with an open saturated subset wich is $\R_+^*$-invariant. We compare the connecting maps and the index maps of $G$ with those of $G/\R_+^*$, using Connes-Thom morphisms.

\subsection{A (well known) remark on exact sequences}

We will use the quite immediate (and well known) result:

\begin{lemma} \label{elabet}
Consider a commutative diagram of semi-split exact sequences of $C^*$-algebras
$$\xymatrix{
0\ar[r] &J_1\ar[r]\ar[d]_{f_J}& A_1\ar[r]^{q_1}\ar[d] _{f_A}&B_1\ar[r]\ar[d]_{f_B}& 0\\
0\ar[r] &J_2\ar[r]& A_2\ar[r] ^{q_2}&B_2\ar[r]& 0
}
$$ \begin{enumerate}
\item We have $\partial_1 \otimes [f_J]=[f_B]\otimes \partial_2$ where $\partial _i$ denotes the element in $KK^1(B_i,J_i)$ associated with the exact sequence $$\xymatrix{
0\ar[r] &J_i\ar[r]& A_i\ar[r] &B_i\ar[r]& 0.
}
$$
\item  If two of the vertical arrows are $KK$-equivalences, then so is the third one.
\end{enumerate}

\begin{notation} When $f:A\rightarrow B$ is a morphism of $C^*$-algebra, we will denote the corresponding mapping cone by $\Cn_f =\{(x,h)\in A\oplus B[0,1)\ ; \ h(0)=f(x)\}$.
\nomenclature[H, 06]{$\Cn_f$}{The mapping cone of a morphism $f:A\rightarrow B$ of $C^*$-algebra}
\end{notation}

\begin{proof}
\begin{enumerate}
\item See \eg \cite{CuSk}. Let $\Cn_{q_i}$ be the mapping cone of $q_i$ and $j_i:B_i(0,1)\to \Cn_{q_i}$ and $e_i:J_i\to \Cn_{q_i}$ the natural (excision) morphisms. The \emph{excision morphism} $e_i$ is $K$-invertible and $\partial_i=[j_i]\otimes [e_i]^{-1}$. 
\item For every separable $C^*$-algebra $D$, by applying the ``five lemma'' to the diagram $$\xymatrix{
...\ar[r] &KK^*(D,J_1)\ar[r]\ar[d] & KK^*(D,A_1)\ar[r]\ar[d]&KK^*(D,B_1)\ar[r]\ar[d]& KK^{*+1}(D,J_1)\ar[r]\ar[d]&...\\
...\ar[r] &KK^*(D,J_2)\ar[r]& KK^*(D,A_2)\ar[r] &KK^*(D,B_2)\ar[r]& KK^{*+1}(D,J_2)\ar[r]&...
}
$$we find that all vertical arrows are invertible. Applying this to $D=J_2$ (\resp $A_2$, $B_2$) we find a one sided inverse to  $[f_J]$  (\resp $f_A$, $f_B$). Applying this again to $D=J_1$ (\resp $A_1$, $B_1$), it follows that this inverse is two-sided.
\qedhere
\end{enumerate}
\end{proof}
\end{lemma}

\nomenclature[I, 01]{$[f]$}{The $KK$-element, in $KK(A,B)$ associated to a morphism of C$^*$-algebra $f:A\to B$}

\subsection{Saturated open subsets, connecting maps and full index map}\label{gueule}

In this section, we let $G\rightrightarrows M$ be a Lie groupoid and $F$ be a closed subset of $M$ saturated for $G$. Put $W=M\setminus F$. Denote by $G_W$ the open subgroupoid $G_W=G_W^W$ of $G$ and $G_F$ its complement. If $F$ is not a submanifold, then $G_F$ is not a Lie groupoid, but as explained in remark \ref{Mnotmanifold}, we still can define $\Psi^*(G_F)$ (it is the quotient $\Psi^*(G)/\Psi^*(G_{W})$) the symbol map, \textit{etc.}

\nomenclature[H, 03]{$\Psi^*(G_F)$}{The quotient $\Psi^*(G)/\Psi^*(G_{W})$ where $F$ is a closed subset of $G^{(0)}$ saturated for $G$ and $W=G^{(0)}\setminus F$}
\nomenclature[H, 06]{$\Sigma^W(G)$}{The quotient $\Psi^*(G)/C^*(G_W)$}
\nomenclature[I, 04]{$\wind_{full}^W(G)$}{The connecting element, which belongs to $KK^1(\Sigma^W(G),C^*(G_W))$ associated to the short exact sequence $
0\longrightarrow C^*(G_{W})\longrightarrow  \Psi^*(G)\longrightarrow \Sigma^W(G)\longrightarrow  0$}

\nomenclature[I, 03]{$\partial_G^W$}{The connecting element, which belongs to $KK^1(C^*(G\vert_F),C^*(G\vert_{W}))$, 
associated to the short exact sequence $\xymatrix{0\ar[r] &C^*(G\vert_{W})\ar[r]&C^*(G)\ar[r]& C^*(G\vert_F)\ar[r]&0}$ where $W$ is a saturated open subset of $G^{(0)}$ and $F=G^{(0)}\setminus W$}

\medskip Define the \emph{full symbol algebra} $\Sigma^W(G)$ to be the quotient $\Psi^*(G)/C^*(G_W)$.

\medskip In this section we will be interested in the description of elements $\partial_G^W\in KK^1(C^*(G_F),C^*(G_W))$ and $\wind_{full}^W(G)\in KK^1(\Sigma^W(G),C^*(G_W))$ associated to the exact sequences $$0\longrightarrow C^*(G_{W})\longrightarrow C^*(G)\longrightarrow C^*(G_F)\longrightarrow 0\eqno E_\partial$$ and $$
0\longrightarrow C^*(G_{W})\longrightarrow  \Psi^*(G)\longrightarrow \Sigma^W(G)\longrightarrow  0
.\eqno E_{\widetilde {\ind}_{full}}
$$ To that end, it will be natural to assume that the restriction $G_F$ of $G$ to $F$ is amenable - so that the above sequences are exact and semi-split for the reduced as well as the full groupoid algebra. 

At some point, we wish to better control the $K$-theory of the $C^*$-algebras $C^*(G_F)$ and $\Sigma^W(G)$. We will assume that the index element $\ind_{G_F}\in KK(C_0((\gA ^*G)_{|F}),C^*(G_F))$ is invertible. This assumption is satisfied in our main applications in section \ref{Section6}.

 \subsubsection{Connecting map and index}
 
Assume that the groupoid $G_F$ is amenable. We have a diagram 

$$\xymatrix{&&&0\ar[d]&0\ar[d] \\
E_\partial : & 0\ar[r] &C^*(G_W)\ar[r]\ar@{=}[d]&C^*(G)\ar[r]\ar[d]& C^*(G_F) \ar[r]\ar[d]^j&0 \\
E_{\widetilde {\ind}_{full}} : & 0\ar[r] &C^*(G_W)\ar[r] &\Psi^*(G)\ar[d]\ar[r] &\Sigma^{W}(G)\ar[r]\ar[d]&0 \\
&&&C_0(\bS^*\gA G)\ar[d]\ar@{=}[r]&C_0(\bS^*\gA G)\ar[d]  \\
&&&0&0}$$

 It follows that $\partial _G^W=j^*(\wind_{full}^W(G))$ (proposition \ref{elabet}).

\subsubsection{Connecting maps}

\begin{proposition}\label{ConnectMapAbstract}
Let $\partial_G^W \in KK^1(C^*(G_F),C^*(G_{W}))$ be the element  associated with the exact sequence $$0\longrightarrow C^*(G_{W})\longrightarrow C^*(G)\longrightarrow C^*(G_F)\longrightarrow 0.$$

Similarly, let  $\partial_{\gA G}^W\in KK^1(C_0((\gA ^*G)_{|F}),C_0((\gA ^*G)_{|W}))$ be associated with the exact sequence $$0\longrightarrow C_0((\gA ^*G)_{|W})\longrightarrow C_0(\gA ^*G)\longrightarrow C_0((\gA ^*G)_{|F})\longrightarrow 0.$$ We have $\partial_{\gA G}^W\otimes \ind_{G_W}=\ind_{G_F}\otimes \partial_G^W$.

In particular, if the index element $\ind_{G_F}\in KK(C_0((\gA ^*G)_{|F}),C^*(G_F))$ is invertible, then the element $\partial_G^W$  is the composition $\ind_{G_F}^{-1}\otimes \partial_{\gA G}^W\otimes \ind_{G_W}$.
\begin{proof}
 
Indeed, we just have to apply twice proposition \ref{elabet} using the adiabatic deformation $G_{ad}^{[0,1]}$ and the diagram:
$$\xymatrix{
0\ar[r]&C_0((\gA ^*G)_{|W}))\ar [r] & C_0(\gA ^*G)\ar [r] &C_0((\gA ^*G)_{|F}))\ar[r]& 0\\
0\ar[r]& C^*(G_{ad}(W\times [0,1])\ar [u]^{ev_0}\ar [r]\ar[d]_{ev_1} & C^*(G_{ad})\ar [u]^{ev_0}\ar [r]\ar[d]_{ev_1}  &C^*(G_{ad}(F\times [0,1])) \ar [u]^{ev_0}\ar[d]_{ev_1} \ar[r]& 0\\
0\ar[r]& C^*(G_{W})\ar [r] & C^*(G)\ar [r] &C^*(G_F)\ar[r]& 0
}$$
\end{proof}
\end{proposition}

\subsubsection{A general remark on the index}

In the same way as the index $\ind_G\in KK(C_0(\gA^*G),C^*(G))$ constructed using the adiabatic groupoid is more primitive and to some extent easier to handle than $\wind_G\in KK^1(C_0(\bS^*\gA G),C^*(G))$ constructed using the exact sequence of pseudodifferential operators, there is in this ``relative'' situation a natural more primitive element.

Denote by $\gA_W G=G_{ad}(F\times [0,1)\cup W\times \{0\})$ the restriction of $G_{ad}$ to the saturated locally closed subset $F\times [0,1)\cup W\times \{0\}$. Note that, since we assume that $G_F$ is amenable, and since $\gA  G$ is also amenable (it is a bundle groupoid), the groupoid $\gA_W G$ is amenable.

\nomenclature[G, 09]{$\gA_W G$}{The restriction of the adiabatic groupoid $G_{ad}$ to $F\times [0,1)\cup W\times \{0\}$ where $F$ is a closed subset of $G^{(0)}$ saturated for $G$ and $W=G^{(0)}\setminus F$}

Similarly to \cite{DebLesc, DLR}, we define the \emph{noncommutative algebroid} of $G$ relative to $F$ to be $C^*(\gA _W G)$. Note that by definition we have: $$C^*(G_{ad}^{ [0,1)})/C^*(G_{ad}(W\times (0,1))=C^*(G_{ad}(F\times [0,1)\cup W\times \{0\})=C^*(\gA _W G)$$

We have an exact sequence $$0\to C^*(G_W\times (0,1])\longrightarrow C^*(G_{ad}(F\times [0,1)\cup W\times [0,1]))\overset{ev_0}\longrightarrow C^*(\gA _W G)\to 0,$$ 
where $ev_0 : C^*(G_{ad}(F\times [0,1)\cup W\times [0,1])) \to C^*(G_{ad}(F\times [0,1)\cup W\times \{0\})=C^*(\gA _W G)$ is the restriction morphism.
As $C^*(G_W\times (0,1])$ is contractible the $KK$-class $[ev_0]\in KK( C^*(G_{ad}(F\times [0,1)\cup W\times [0,1])),C^*(\gA _W G))$ is invertible. Let as usual $ev_1:C^*(G_{ad}(F\times [0,1)\cup W\times [0,1]))\to C^*(G_W)$ be the evaluation at $1$. We put:
$$\ind_G^W=[ev_0]^{-1}\otimes [ev_1] \in KK(C^*(\gA _W G),C^*(G_W))\ .$$

\nomenclature[I, 05]{$\ind_G^W$}{The $KK$-element $[ev_0]^{-1}\otimes [ev_1]$, which belongs to $KK(C^*(\gA _W G),C^*(G_W))$, 
associated to the groupoid $G_{ad}(F\times [0,1)\cup W\times [0,1])=G_W\times (0,1]\sqcup \gA _W G$}

Recall from \cite[Rem 4.10]{DS1} and \cite[Thm. 5.16]{DS3} that there is a natural action of $\R$ on $\Psi^*(G)$ such that $\Psi^*(G)\rtimes \R$ is an ideal in $C^*(G_{ad}^{[0,1)})$ (using a homeomorphism of $[0,1)$ with $\R_+$). This ideal is the kernel of the composition $C^*(G_{ad}^{[0,1)})\overset {\ev_0}{\longrightarrow}C_0(\gA^* G)\to C(M)$. 

Recall that the restriction to $C^*(G)$ of the action of $\R$ is inner. It follows that $C^*(G_W)\subset \Psi^*(G)$ is invariant by the action of $\R$ - and $C^*(G_W)\rtimes \R=C^*(G_W)\otimes C_0(\R)=C^*(G_{ad}(W\times (0,1)))$. 

We thus obtain an action of $\R$ on $\Sigma^W(G)=\Psi^*(G)/C^*(G_W)$ and an inclusion $i:\Sigma^W(G)\rtimes \R\hookrightarrow C^*(\gA_WG)$.

\begin{proposition}\label{indGW}
The element $\widetilde{\ind}_{full}^W\in KK^1(\Sigma^W(G),C^*(G_W))$ corresponding to the exact sequence
$$
0\longrightarrow C^*(G_{W})\longrightarrow  \Psi^*(G)\longrightarrow \Sigma^W(G)\longrightarrow  0
.\eqno E_{\widetilde {\ind}_{full}}
$$ is the Kasparov product of:\begin{itemize}
\item the Connes-Thom element $[th]\in KK^1(\Sigma^W(G),\Sigma^W(G)\rtimes \R)$;
\item the inclusion $i:\Sigma^W(G)\rtimes \R\hookrightarrow C^*(\gA _W G)$;
\item the index $\ind_G^W=[ev_0]^{-1}\otimes [ev_1]$ defined above.
\end{itemize}
\begin{proof} 
By naturality of the Connes Thom element, it follows that $$\widetilde{\ind}^W_{full}\otimes [B]=-[th]\otimes  [\partial]$$ where $\partial\in KK^1(C^*(\gA _W G),C^*(G_W\times (0,1)))$ is the $KK^1$-element corresponding with the exact sequence $$\xymatrix{
0\ar[r]& C^*(G_W)\rtimes \R \ar [r]   & \Psi^*(G)\rtimes \R \ar [r]  &\Sigma^W(G)\rtimes \R    \ar[r]& 0\\
}$$ and $[B]\in KK^1(C^*(G_W),C^*(G_W)\rtimes \R)$ is the Connes-Thom element. Note that, since the action is inner, $[B]$ is actually the Bott element.

By the diagram 
$$\xymatrix{
0\ar[r]& C^*(G_W)\rtimes \R \ar [r]\ar[d]  & \Psi^*(G)\rtimes \R \ar [r]\ar[d]   &\Sigma^W(G)\rtimes \R \ar[d]^{i}  \ar[r]& 0\\
0\ar[r]& C^*(G_{W}\times (0,1))\ar [r] & C^*(G_{ad}^{[0,1)})\ar [r] &C^*(\gA_W G)\ar[r]& 0
}$$
we deduce that $[\partial]=i^*[\partial']$ where $\partial'$ corresponds to the second exact sequence.

Finally, we have a diagram
$$\xymatrix{&0\ar[d]&0\ar[d]& &\\
0\ar[r]& C^*(G_{W}\times (0,1)) \ar [r]\ar[d]  & C^*(G_{ad}^{[0,1)}) \ar [r]\ar[d]   & C^*(\gA_W G)\ar@{=}[d]  \ar[r]& 0\\
0\ar[r]& C^*(G_{W}\times (0,1])\ar[d]\ar [r] & C^*(G_{ad}(F\times [0,1)\cup W\times [0,1]))\ar[d]_{ev_1}\ar [r]^{\hskip 50pt ev_0} &C^*(\gA_W G)\ar[r]& 0\\
&C^*(G_W)\ar@{=}[r]\ar[d]&C^*(G_W)\ar[d]\\
&0&0}$$
where exact sequences are semisplit. Now the connecting element corresponding to the exact sequence $$\xymatrix{
0\ar[r]& C^*(G_{W}\times (0,1)) \ar [r] & C^*(G_{ad}(F\times [0,1)\cup W\times [0,1])) \ar [r]^{\hskip 23pt ev_0\oplus ev_1} &C^*(\gA_W G)\oplus C^*(G_W)\ar[r]&0}$$
 is $[\partial']\oplus [B]$ and it follows that $$[ev_0]\otimes [\partial']+[ev_1]\otimes  [B]=0.$$
 
As $\widetilde{\ind}_{full}^W\otimes  [B]=-[th]\otimes  [i]\otimes  [\partial']$ and $[\partial']=-[ev_0]^{-1}\otimes[ev_1]\otimes [B]$, the result follows from invertibility of the Bott element.
\end{proof}
\end{proposition}

\subsubsection{Full symbol algebra and index}

 Denote by $\Psi_F^*(G)$ the subalgebra $C_0(M)+\Psi^*(G_{W})$ of $\Psi^*(G)$. It is the algebra of pseudodifferential operators which become trivial (\ie multiplication operators) on $F$. Let $\Sigma_F(G)=\Psi_F^*(G)/C^*(G_{W})$ be the algebra of the corresponding symbols. It is the subalgebra $C_0(M)+C_0(\bS^* \gA G_{W})$ of $C_0(\bS^* \gA G)$ of symbols $a(x,\xi)$ with $x\in M$ and $\xi\in (\bS^* \gA G)_x$ whose restriction on $F$ does not depend on $\xi$.

\nomenclature[H, 05]{$\Psi_F^*(G)$}{The subalgebra $C_0(M)+\Psi^*(G_{W})$ of $\Psi^*(G)$}
\nomenclature[H, 07]{$\Sigma_F(G)$}{The algebra $\Psi_F^*(G)/C^*(G_{W})$}

\begin{lemma} \label{lemme3.5}Assume that the index element $\ind_{G_F}\in KK(C_0((\gA ^*G)_{|F}),C^*(G_F))$ is invertible, \ie that the $C^*$-algebra of the adiabatic groupoid $C^*(G_{ad}(F\times [0,1)))$ is $K$-contractible.
\begin{enumerate}
\item The inclusion $j_\psi:C_0(F)\to \Psi^*(G_{F})$  is a $KK$-equivalence.
\item The inclusion $j_\sigma:\Sigma_F(G) =\Psi_F^*(G)/C^*(G_{W}) \to \Sigma^W(G)$ is also a $KK$-equivalence.\label{lemme3.5b}
\end{enumerate}
\begin{proof}
\begin{enumerate}
\item Consider the diagram
$$\xymatrix{
0\ar[r]&C_0((\gA ^*G)_{|F})\ar [r] & C_{0}((B^*\gA G)_{|F})\ar [r] &C_0((\bS^*\gA G)_{|F})\ar[r]& 0\\
0\ar[r]& C^*(G_{ad}(F\times [0,1]))\ar [u]^{ev_0}\ar [r]\ar[d]_{ev_1} & \Psi^*(G_{ad}(F\times [0,1]))\ar [u]^{ev_0}\ar [r]\ar[d]_{ev_1}  &C((\bS^*\gA G)_{|F}\times[0,1]) \ar [u]^{ev_0}\ar[d]_{ev_1} \ar[r]& 0\\
 0\ar[r]& C^*(G_F)\ar [r] & \Psi^*(G_F)\ar [r] &C((\bS^*\gA G)_{|F})\ar[r]& 0
}$$
where the horizontal exact sequences are the pseudodifferential exact sequences $E_{\Psi^*(\gA G)_F}$, $E_{\Psi^*(G_{ad}(F\times [0,1]))}$ and $E_{\Psi^*(G_F)}$. Since $\ind_{G_F}$ is invertible $ev_1:C^*(G_{ad}(F\times [0,1])\to C^*(G_F)$ is a $KK$-equivalence. Hence,
the left and right vertical arrows are all $KK$-equivalences, and therefore so are the middle ones. The inclusion $C_0(F)$ in $C_0((B^*\gA G)_{|F})$ is a homotopy equivalence and therefore the inclusions $C_0(F)\to \Psi^*(G_{ad}(F\times [0,1]))$ and $C_0(F)\to \Psi^*(G_{F})$ are $KK$-equivalences.

\item 
Apply Lemma \ref{elabet} to the diagrams

$$\xymatrix{
0\ar[r] &\Psi^*(G_{W})\ar[r]\ar@{=}[d]& \Psi_F^*(G)\ar[r]\ar[d]_{J_\psi} &C_0(F)\ar[r]\ar[d]_{j_\psi}& 0\\
0\ar[r] &\Psi^*(G_{W})\ar[r]& \Psi^*(G)\ar[r] &\Psi^*(G_F)\ar[r]& 0
}
$$

$$\xymatrix{
0\ar[r] &C^*(G_{W})\ar[r]\ar@{=}[d]& \Psi_F^*(G)\ar[r]\ar[d]_{J_\psi} &\Sigma_F(G)\ar[r]\ar[d]_{j_\sigma}& 0\\
0\ar[r] &C^*(G_{W})\ar[r]& \Psi^*(G)\ar[r] &\Sigma^W(G)\ar[r]& 0
}
$$ 
we find that $J_\Psi$ and $j_\sigma$ are $K$-equivalences.
\qedhere
\end{enumerate}
\end{proof}
\end{lemma}

The diagram in lemma \ref{lemme3.5}.\ref{lemme3.5b}) shows that $\partial_F=j_\sigma^*(\wind_{full}^W(G))$ where $\partial_F\in KK^1(\Sigma_F(G),C^*(G_{W}))$ is the $KK$-element associated with the exact sequence $$\xymatrix{
0\ar[r] &C^*(G_{W})\ar[r] & \Psi_F^*(G)\ar[r] &\Sigma_F(G)\ar[r] & 0.
}
$$

\nomenclature[I, 06]{$\partial_F$}{The connecting element, which belongs to $KK^1(\Sigma_F(G),C^*(G_{W}))$ associated to the short exact sequence $\xymatrix{ 0\ar[r] &C^*(G_{W})\ar[r] & \Psi_F^*(G)\ar[r] &\Sigma_F(G)\ar[r] & 0 }$}

So, let's compute the $KK$-theory of $\Sigma_F(G)$ and the connecting element $\partial_F$.

\medskip Consider the vector bundle $\gA G$ as a groupoid (with objects $M$). It is its own algebroid - with anchor $0$. With the notation in \ref{notation1.3},
\begin{itemize}
\item $C^*(\gA G)$ identifies with $C_0(\gA^* G)$ and $C^*(\gA G_W)$ with $C_0(\gA^* G_W)$;
\item $\Psi ^*(\gA G)$  identifies with $C_0(B^*\gA G)$; it is homotopy equivalent to $C_0(M)$; 
\item the spectrum of $\Psi_F ^*(\gA G)$ is $B^*_F\gA G$ the quotient of $B^*\gA G$ where we identify two points $(x,\xi)$ and $(x,\eta)$ for $x\in F$; it is also homotopy equivalent to $C_0(M)$.
\item the algebroid of the groupoid $\gA G$ is $\gA G$ itself; therefore, $\Sigma_F(\gA G)=\Sigma_F(G)$;  its spectrum is $\bS^*_F\gA G$ which is the image of $\bS^*\gA G$ in $B^*_F\gA G$. 
\end{itemize}

We further note.
\begin{enumerate}
\item Let $k:C_0(\gA^* G_W)\to C_0(M)$ be given by $k(f)(x)=\begin{cases}f(x,0)& \hbox{if\ } x\in W\\0& \hbox{if\ } x\in F\end{cases}$. We find a commutative diagram
\(\xymatrix{
C_0(\ronde B^*\gA G_W)\ar[r]\ar[d]& C_0(B_F^*\gA G)\ar[d]\\
C_0(\gA^* G_W)\ar[r]^k& C_0(M)}\)
where the vertical arrows are homotopy equivalences. \label{commdiagr(d)}

\item the exact sequence $0\to C^*(\gA G_W)\to \Psi_F^*(\gA G)\to \Sigma_F(\gA G)\to 0$, reads $0\to C_0(\ronde B^*\gA G_W)\to C_0(B_F^*\gA G)\to C_0(S_F^*\gA G)\to 0$.
\label{suiteexacte(f)}
\end{enumerate}

We deduce using successively (\ref{suiteexacte(f)}) and (\ref{commdiagr(d)}):
\begin{proposition}\label{Sigma=cone}
\begin{enumerate}
\item The algebra $C_0(S_F^*\gA G)$ is $KK^1$-equivalent with the mapping cone of the inclusion $C_0(\ronde B^*\gA G_W)\to C_0(B_F^*\gA G)$.
\item This mapping cone is homotopy equivalent to the mapping cone of the morphism $k$.

\end{enumerate}
 \hfill$\square$

\end{proposition}

Note finally that we have a diagram
$$\xymatrix{
0\ar[r]&C_0((\gA ^*G)_{|W})\ar [r] & \Psi^*_F(\gA G)\ar [r] &\Sigma_F(\gA G)\ar[r]& 0\\
0\ar[r]& C^*(G_{ad}(W\times [0,1]))\ar [u]^{ev_0}\ar [r]\ar[d]_{ev_1} & \Psi_{F\times [0,1]}^*(G_{ad})\ar [u]^{ev_0}\ar [r]\ar[d]_{ev_1}  &\Sigma_{F\times[0,1]}(G_{ad}) \ar [u]^{ev_0}\ar[d]_{ev_1} \ar[r]& 0\\
0\ar[r]& C^*(G_W)\ar [r] & \Psi_F^*(G)\ar [r] &\Sigma_F(\gA G)\ar[r]& 0
}$$
The right vertical arrows are $KK$-equivalences, and therefore we find $\partial \otimes[ev_0]^{-1}\otimes [ev_1] =\partial_F$, where $\partial$ is the connecting element of the first horizontal exact sequence. To summarize, we have proved:
\begin{proposition} \label{avita}Assume that the index element $\ind_{G_F}\in KK(C_0((\gA ^*G)_{|F}),C^*(G_F))$ is invertible.
\begin{enumerate}
\item The inclusion $j_\sigma:\Sigma_F(G)\to \Sigma^W(G)$ is a $KK$-equivalence.
\item The analytic index $\wind_{full}^W(G) \in KK^1(\Sigma^W(G),C^*(G_W))$ corresponding to the exact sequence  $$\xymatrix{
0\ar[r] &C^*(G_{W})\ar[r]& \Psi^*(G)\ar[r] &\Sigma^W(G)\ar[r]& 0
}
$$ is the Kasparov product of \begin{itemize}
\item the element $[j_\sigma]^{-1}\in KK(\Sigma^W(G), \Sigma_F(G))$;
\item the connecting element $\partial \in KK^1(\Sigma_F(\gA G),C_0((\gA ^*G)_{|W}))$ associated with the exact sequence of (abelian) $C^*$-algebras $$\xymatrix{
0\ar[r]&C_0((\gA ^*G)_{|W})\ar [r] & \Psi^*_F(\gA G)\ar [r] &\Sigma_F(\gA G)\ar[r]& 0;}$$
\item the analytic index element $\ind_{G_W}$ of $G_W$, \ie the element $$[ev_0]^{-1}\otimes [ev_1]\in KK(C_0((\gA ^*G)_{|W}),C^*(G_W)).$$
\hfill$\square$
\end{itemize}
\end{enumerate}
\end{proposition}

\subsubsection{Fredholm realization}
Let $\sigma$ be a classical  symbol which defines an element in $K_1(C_0(\bS^*\gA G))$. A natural question is: when can this symbol be lifted to a pseudodifferential element which is invertible modulo $C^*(G_W)$?

In particular, if $G_W$ is the pair groupoid $W\times W$, this question reads: when can this symbol be extended to a Fredholm operator? Particular cases of this question were studied in \cite{AlMel1, AlMel2}.

Consider the exact sequences: $$\xymatrix{  & &  0 & 0 & & \\
E: &0\ar[r] &C^*(G_{F})\ar[r] \ar[u] & \Sigma^W(G)\ar[r]^{q} \ar[u] &C_0(\bS^*\gA G)\ar[r] \ar@{=}[d] & 0 \\  & 0\ar[r]  &C^*(G)\ar[r] \ar[u] & \Psi^*(G) \ar[u]  \ar[r] & C_0(\bS^*\gA G) \ar[r] & 0  \\  & & C^*(G_W) \ar@{=}[r] \ar[u] & C^*(G_W) \ar[u] & & & \\  & & 0 \ar[u] & 0 \ar[u] & & }
$$ 
The element $\sigma$ is an invertible element in $M_n(C_0(\bS^*\gA G)^+)$ (where $ C_0(\bS^*\gA G)^+$ is obtained by adjoining a unit to $C_0(\bS^*\gA G)$ - if $G^{(0)}$ is not compact). The question is: when can $\sigma$ be lifted to an invertible element of $M_n( \Sigma^W(G)^+)$. 

By the $K$-theory exact sequence, if this happens then the class of $\sigma$ is in the image of $K_1(\Sigma^W(G))$ and therefore its image via the connecting map of the exact sequence $E$ is $0$ in $K_0(C^*(G_{F}))$. Conversely, if  the image of $\sigma $ via the connecting map of $E$ vanishes, then the class of $\sigma$ in $K_1(C_0(\bS^*\gA G))$ is in the image of $K_1(\Sigma^W(G))$. This means that there exists $p\in \N$ and an invertible element $x\in M_{n+p}(\Sigma^W(G)^+)$ such that $q(x)$ and $\sigma \oplus 1_p$ are in the same path connected component of $GL_{n+p}( C_0(\bS^*\gA G)^+)$. \\ Now the morphism $q:M_{n+p}( \Sigma^W(G)^+)\to M_{n+p}(C_0(\bS^*\gA G)^+)$ is open and therefore the image of the connected component $GL_{n+p}(\Sigma^W(G)^+)_{(0)}$ of $1_{n+p}$ in $GL_{n+p}(\Sigma^W(G)^+)$ is an open (and therefore also closed) subgroup of $GL_{n+p}(C_0(\bS^*\gA G)^+)$. It follows immediately that $q\Big(GL_{n+p}(\Sigma^W(G)^+)_{(0)}\Big)=GL_{n+p}(C_0(\bS^*\gA G)^+)_{(0)}$. Finally $(\sigma\oplus 1_p)x^{-1}$ is in the image of $GL_{n+p}(\Sigma^W(G)^+)_{(0)}$, therefore $\sigma \oplus 1_p$ can be lifted to an  invertible element of $M_n(\Sigma^W(G)^+)$.

Let us make a few comments:

\begin{enumerate}
\item Considering the diagram
$$\xymatrix{
0\ar[r] &C^*(G_{F})\ar[r]\ar@{=}[d]& \Sigma^W(G)\ar[r]\ar[d] &C_0(\bS^*\gA G)\ar[d]\ar[r]& 0\\
0\ar[r] &C^*(G_{F})\ar[r]& \Psi^*(G_F)\ar[r] &C_0(\bS^*\gA G_F)\ar[r]& 0
}
$$ 
we find that the image of $\sigma$ in $K_0(C^*(G_{F}))$ is the index $\ind (\sigma_F)$ of the restriction $\sigma_F$ of $\sigma$ to $F$.

\item Considering the diagram
$$\xymatrix{
&0\ar[d]&0\ar[d]&0\ar[d]\\
0\ar[r] &C^*(G_W)\ar[r]\ar[d]& \Psi^*(G_W)\ar[r]\ar[d] &C_0(\bS^*\gA G_W)\ar[r]\ar[d]& 0\\
0\ar[r] &C^*(G)\ar[r]\ar[d] & \Psi^*(G)\ar[r]\ar[d] &C_0(\bS^*\gA G)\ar[d]\ar[r]& 0\\
0\ar[r] &C^*(G_F)\ar[r]\ar[d]& \Psi^*(G_F)\ar[r]\ar[d] &C_0(\bS^*\gA G_F)\ar[r]\ar[d]& 0\\
&0&0&0}
$$ 
we could also say that our question is: when is the index $\ind(\sigma)\in K_0(C^*(G))$ in the image of $K_0(C^*(G_W))$, and of course this happens if and only if the image of $\ind(\sigma)$ in $K_0(C^*(G_F))$ vanishes. Again we may notice that the image of $\ind(\sigma)$ in $K_0(C^*(G_F))$ is $\ind (\sigma_F)$.

\item Of course, the same remark holds if we start with a symbol defining a class in $K_0(C_0(\bS^*\gA G))$.
\end{enumerate}

\subsubsection{Relative K-theory and full index}

It is actually better to consider the index map in a relative $K$-theory setting. Indeed, the starting point of the index problem is a pair of bundles $E_{\pm}$ over $M$ together with a pseudodifferential operator $P$ from sections of $E_+$ to sections of $E_-$ which is invertible modulo $C^*(G_W)$. Consider the morphism $\psi:C_0(M)\to \Psi^*(G)$ which associates to a (smooth) function $f$ the order $0$ (pseudo)differential operator multiplication by $f$ and $\sigma_{full}:\Psi^*(G)\to \Sigma^{W}(G)$ the full symbol map. 

\smallskip 
Put $\mu=\sigma_{full}\circ \psi$.

By definition, for any $P\in\Psi^*(G)$,  the triple $(E_{\pm},\sigma_{full}(P))$ is an element in the relative $K$-theory of the morphism $\mu$. The index $\cdot\otimes \widetilde{\ind}_{full}^W(G)$ considered in the previous section is the composition of the morphism $K_1(\Sigma^{W}(G))\to K_0(\mu)$ 
\footnote{Recall that if $f:A\to B$ is a morphism of $C^*$-algebras, we have a natural morphism $u:K_{*+1}(B)\to K_*(f)$ corresponding to the inclusion of the suspension of $B$ in the cone of $f$.}  with the index map $\ind_{rel}:K_0(\mu)\to K_0(C^*(G_W))$ which to $(E_{\pm},\sigma_{full}(P))$ associates the class of $P$.

The morphism $\ind_{rel}$ can be thought of as the composition of the obvious morphism $K_0(\mu)\to K_0(\sigma_{full})\simeq K_0(\ker(\sigma_{full}))=K_0(C^*(G_W))$.

\bigskip Let us now compute the group $K_*(\mu)$ and the morphism $\ind_{rel}$ when the index element $\ind_{G_F}\in KK(C_0((\gA ^*G)_{|F}),C^*(G_F))$ is invertible.

\begin{proposition}\label{indrel}
Assume that the index element $\ind_{G_F}\in KK(C_0((\gA ^*G)_{|F}),C^*(G_F))$ is invertible. Then $K_*(\mu)$ is naturally isomorphic to $K_*(C_0(\gA^* G_W))$. Under this isomorphism, $\ind_{rel}$ identifies with  $\ind_{G_W}$.

\begin{proof}
We have a diagram $$\xymatrix{ 
0\ar[r]&C_0(W)\ar[r]\ar[d]^i& C_0(M)\ar[r]\ar[d]^{\mu}&C_0(F)\ar[d]^{j_\Psi}\ar[r]&0\\
0\ar[r]&C_0(\bS^*\gA G_W)\ar[r]&\Sigma^{W}(G)\ar[r]&\Psi^*(G_F)\ar[r]&0
}$$
As $j_\Psi$ is an isomorphism in $K$-theory, the map $K_*(i)\to K_*(\mu)$ induced by the first commutative square of this diagram is an isomorphism. As $K_*(i)=K_*(\Cn_i)$ and $\Cn_i=C_0(\gA^* G_W)$, we obtained the desired isomorphism $K_*(C_0(\gA^* G_W))\simeq K_*(\mu)$.

Comparing the diagrams
$$\xymatrix{ 
C_0(W)\ar[r]\ar[d]^i& C_0(M)\ar[r]\ar[d]^{\mu}&\Psi^*(G)\ar[d]^{\sigma_{full}}\\
C_0(\bS^*\gA G_W)\ar[r]&\Sigma^{W}(G)\ar@{=}[r]&\Sigma^{W}(G)
} \ \ \hbox{and}\ \  \xymatrix{ 
C_0(W)\ar[r]\ar[d]^i& \Psi^*(G_W)\ar[r]\ar[d]^{\sigma_{W}}&\Psi^*(G)\ar[d]^{\sigma_{full}}\\
C_0(\bS^*\gA G_W)\ar@{=}[r]&C_0(\bS^*\gA G_W)\ar[r]&\Sigma^{W}(G)
}$$
we find that the composition $K_*(i)\overset{\sim}{\longrightarrow} K_*(\mu)\longrightarrow K_*(\sigma_{full})$ coincides with the index \\ $K_*(i){\longrightarrow} K_*(\sigma_W)\overset{\sim}{\longrightarrow} K_*(\sigma_{full})$.
\end{proof}
\end{proposition}

\begin{remark}
We wrote the relative index map in terms of morphisms of $K$-groups. One can also write everything in terms $KK$-theory, by replacing relative $K$-theory by mapping cones, \ie construct the relative index as the element of  $KK(\Cn_{\mu},C^*(G_W))$ given as $\psi_{\Cn}^*([e]^{-1})$ where $e:C^*(G_W)\to \Cn_{\sigma_{full}}$ is the ($KK$-invertible) ``excision map'' associated with the (semi-split) exact sequence $0\to C^*(G_W)\to \Psi^*(G)\overset{\sigma_{full}}\longrightarrow \Sigma_{F}(G)\to 0$ and $\psi_{\Cn}:\Cn_{\mu}\to \Cn_{\sigma_{full}}$ is the morphism associated with $\psi$.
\end{remark}

\subsection{Connes-Thom elements and quotient of a groupoid by $\R_+^*$}\label{saison}

\subsubsection{Proper action on a manifold}

\begin{remark}[Connes-Thom elements]\label{ConnesThom}
Let $\R_+^*$ act smoothly (freely and) properly on a manifold $M$. We have a canonical invertible $KK$-element $\alpha=(H,D)\in KK^1(C_0(M),C_0(M/\R_+^*))$. \begin{itemize}
\item The Hilbert module $H$ is obtained as a completion of $C_c(M)$ with respect to the $C_0(M/\R_+^*)$ valued inner product $\langle \xi|\eta\rangle(p(x))=\int_0^{+\infty}\overline {\xi(t.x)}\,\eta(t.x)\,\frac{dt}t$ for $\xi,\eta \in C_c(M)$, where $p:M\to M/\R_+^*$ is the quotient map. 
\item The operator $D$ is $\frac1t\frac{\partial }{\partial t}$.
\end{itemize}

The inverse element $\beta\in KK^1(C_0(M/\R_+^*),C_0(M))$ is constructed in the following way: $C_0(M/\R_+^*)$ sits in the multipliers of $C_0(M)$. One may define a continuous function $f:M\to [-1,1]$ such that, uniformly on compact sets of $M$, $\lim_{t\to\pm\infty}f(e^t.x)=\pm1$. The pair $(C_0(M),f)$ is then an element in $KK^1(C_0(M/\R_+^*),C_0(M))$. To construct $f$, one may note that, by properness, we actually have a section $\varphi:M/\R_+^* \to M$ and we may thus construct a homeomorphism $\R_+^*\times M/\R_+^*\to M$ defined by $(t,x)\mapsto t.\varphi(x)$. Then put $f(e^t,x)=t(1+t^2)^{-1/2}$.

As an extension of $C^*$-algebras the element $\beta $ is given by considering  $P=(M\times \R_+)/\R_+^*$ (where $\R_+^*$ acts -properly - diagonally). Then $M$ sits as an open subset $(M\times \R_+^*)/\R_+^*$ and we have an exact sequence $0\to C_0(M)\to C_0(P)\to C_0(M/\R_+^*)\to 0$.

\subsubsection{Proper action on a groupoid}
Let now $\R_+^*$ act smoothly (locally \cf remark \ref{RemarqueToutePropre}) properly on a Lie groupoid $G\rightrightarrows M$. The groupoid $G/\R_+^*$ acts on $M$, and the element $\alpha$ is $G$ invariant - and $\beta $ is almost $G$ invariant in the sense of \cite{Legall}. In other words, we obtain elements $\alpha\in KK^1_{G/\R_+^*}(C_0(M),C_0(M/\R_+^*))$ and $\beta\in KK_{G/\R_+^*}^1(C_0(M/\R_+^*),C_0(M))$ which are inverses of each other in Le Gall's equivariant $KK$-theory for groupoids. 

Using the descent morphism of Kasparov (\cite{Kasparov1988}) and Le Gall (\cite{Legall}), we obtain elements $j_{G/\R_+^*}(\alpha)\in KK^1(C^*(G),C^*(G/\R_+^*))$ and $j_{G/\R_+^*}(\beta)\in KK^1(C^*(G/\R_+^*),C^*(G))$ that are also inverses of each other. 

Note also that the element $\beta_G=j_{G/\R_+^*}(\beta) $ is the connecting element of the extension of groupoid $C^*$-algebras $0\to C^*(G)\longrightarrow C^*(\G)\overset{ev_0}{\longrightarrow} C^*(G/\R_+^*)\to 0,$ where $\G=(G\times \R_+)/\R_+^*$  and $ev_0$ comes from the evaluation at \(G\times \{0\}\). Using the pseudodifferential operators on the groupoid $\G$, we obtain a $KK$-element $\beta_G^\Psi\in KK^1(\Psi^*(G/\R_+^*),\Psi^*(G))$. We obtain a commutative diagram
$$\xymatrix{&0\ar[d]&0\ar[d]&0\ar[d]&\\
0\ar[r]&C^*(G)\ar[r]\ar[d]&C^*(\G)\ar[r]\ar[d]&C^*(G/\R_+^*)\ar[r]\ar[d]&0\\
0\ar[r]&\Psi^*(G)\ar[r]\ar[d]&\Psi^*(\G)\ar[r]\ar[d]&\Psi^*(G/\R_+^*)\ar[r]\ar[d]&0\\
0\ar[r]&C_0(\bS^* \gA G)\ar[r]\ar[d]&C_0(\bS^* \gA \G)\ar[r]\ar[d]&C_0(\bS^* \gA (G/\R_+^*))\ar[r]\ar[d]&0\\
&0&0&0&
}$$
The third horizontal  exact sequence corresponds to the proper action of $\R_+^*$ on \(\bS^*\gA G\). In fact $\bS^* \gA  \G$ is homeomorphic (using a cross section) to $(\bS^*\gA (G/\R_+^*))\times \R_+$.
As the connecting elements of the  first and third horizontal (semi-split) exact sequences are invertible, it follows that $C^*(\G)$ and $C_0(\bS^* \gA  \G)$ are $K$-contractible, whence so is $\Psi^*(G)$ and therefore $\beta_G^\Psi$ is a $KK$-equivalence. 
Hence we have obtained:
\begin{proposition} If $\R_+^*$ acts smoothly (locally) properly on the Lie groupoid \(G\), the connecting elements $\beta_G \in KK^1(C^*(G/\R_+^*),C^*(G))$,  $\beta_{\bS^*\gA G} \in KK^1(C_0(\bS^*\gA(G/\R_+^*),C_0(\bS^*\gA G))$ and $\beta_G^\Psi\in KK^1(\Psi^*(G/\R_+^*),\Psi^*(G))$ are $KK$-equivalences. 
\end{proposition}

\subsubsection{Closed saturated subsets and connecting maps}\label{ctoi}
If $W$ is an open saturated subset in $M$ for the actions of $G$ and of $\R_+^*$ and $F=M\setminus W$, one compares the corresponding elements. We then obtain a diagram 
$$\xymatrix{0\ar[r] &C^*(G_{W}^{W}/\R_+^*)\ar[r]\ar@{-}[d]^{\beta'}&C^*(G/\R_+^*)\ar[r]\ar@{-}[d]^{\beta}& C^*(G_F^F/\R_+^*)\ar[r]\ar@{-}[d]^{\beta''}&0\\
0\ar[r] &C^*(G_{W}^{W})\ar[r]&C^*(G)\ar[r]& C^*(G_F^F)\ar[r]&0}$$
where the horizontal arrows are morphisms and the vertical ones $KK^1$-equivalences.

Using the deformation groupoid  $\cG=G^F_F\times [0,1)\cup G\times \{0\}$  which is the restriction of the groupoid \(G\times [0,1)\rightrightarrows M\times [0,1]\) to the closed saturated subset \(F\times [0,1)\cup M\times \{0\}\), we obtain:

\begin{proposition}\label{etK1}
If $G_F^F$ is amenable, $\partial_{G/\R_+^*}^W\otimes \beta'=-\beta''\otimes \partial_G^W\in KK(C^*(G_F^F/\R_+^*),C^*(G_{W}^{W}))$ where $\partial_{G}^W\in KK^1(C^*(G_F^F),C^*(G_{W}^{W}))$ and $\partial_{G/\R_+^*}^W\in KK^1(C^*(G_F^F/\R_+^*),C^*(G_{W}^{W}/\R_+^*))$ denote the $KK$-elements associated with the above exact sequences.
\begin{proof}
Indeed, the connecting map of a semi-split exact sequence $0\to J\to A\overset p\to A/J\to 0$ is obtained as the $KK$-product of the morphism $A/J(0,1)\to \Cn_p$ with the $KK$-inverse of the morphism $J\to \Cn_p$. The $-$ sign comes from the fact that we have naturally elements of $KK(C^*(G_F^F/\R_+^*\times (0,1)^2),C^*(G_{W}^{W}))$ which are equal but with opposite orientations of $(0,1)^2$.
\end{proof}
\end{proposition} 
 
Note also that the same holds for $\Psi^*$ in place of $C^*$.
\end{remark}

\subsubsection{Connes-Thom invariance of the full index}

Let $W$ be as above: an open subset of $M$ saturated for $G$ and invariant under the action of $\R_+^*$. One compares the corresponding $\widetilde \ind_{full}$ elements. Indeed, we have a diagram 
$$\xymatrix{&0\ar[r] &C^*(G_{W}^{W}/\R_+^*)\ar[r]\ar@{-}[d]^{\beta^{G_W}}&\Psi^*(G/\R_+^*)\ar[r]\ar@{-}[d]^{\beta_\Psi^G}& \Sigma^{W/\R_+^*}(G/\R_+^*)\ar[r]\ar@{-}[d]^{\beta_{\Sigma}^{(G,W)}}&0\\
E_{\wind_{full}}:& 0\ar[r] &C^*(G_{W}^{W})\ar[r]&\Psi^*(G)\ar[r]& \Sigma^{W}(G)\ar[r]&0}$$
where the horizontal arrows are morphisms and the vertical ones $KK^1$-elements. As $\beta^{G_W}$ and $\beta_\Psi^G$ are invertible, we deduce as in prop. \ref{etK1}: 

\begin{proposition}\label{betasigma}
\begin{enumerate}
\item The element $\beta_{\Sigma}^{(G,W)}$ is invertible. 
\item We have $\beta_{\Sigma}^{(G,W)}\otimes \wind_{full}^W(G)=-\wind_{full}^{W/\R_+^*}(G/\R_+^*)\otimes \beta^{G_W}.$\hfill$\square$

\end{enumerate}
\end{proposition}

\section{Two classical geometric constructions: Blowup and deformation to the normal cone}

One of the main object in our study is a Lie groupoid $G$ based on a groupoid restricted to a half space. This corresponds to the inclusion of a hypersurface $V$ of $G^{(0)}$ into $G$ and gives rise to the ``gauge adiabatic groupoid'' $\gag$. The construction of $\gag$ is in fact a particular case of the blowup construction corresponding to the inclusion of a Lie subgroupoid into a groupoid. In this section, we will explain this general construction. We will give a more detailed description in the case of an inclusion $V\to G$ when $V$ is a submanifold of $G^{(0)}$.

Let $Y$ be a manifold and $X$ a locally closed submanifold (the same constructions hold if we are given an injective immersion $X\to Y$). Denote by $N_X^Y$ the (total space) of the normal bundle of $X$ in $Y$. 

\subsection{Deformation to the normal cone}\label{sec:DNC}

\nomenclature[G, 10]{$\DNC(Y,X)$}{The deformation to the normal cone of the inclusion of a submanifold $X$ in a manifold $Y$, $\DNC(Y,X)=Y\times \R^* \cup N_X^Y$}

The deformation to the normal cone $\DNC(Y,X)$ is obtained by gluing $N_X^Y\times \{0\}$ with $Y\times \R^*$. The smooth structure of $\DNC(Y,X)$ is described by use of any exponential map $\theta:U'\to U$ which is a diffeomorphism from an open neighborhood $U'$ of the $0$-section in $N_X^Y$ to an open neighborhood $U$ of $X$. The map $\theta $ is required to satisfy $\theta (x,0)=x$ for all $x\in X$ and $p_x\circ d\theta_x=p'_x$ where $p_x:T_xY\to (N_X^Y)_x=(T_xY)/(T_xX)$ and $p'_x:T_xN_X^Y\simeq (N_X^Y)_x\oplus (T_xX) \to (N_X^Y)_x$ are the projections. The manifold structure of $\DNC(Y,X)$ is described by the requirement that:\begin{enumerate}
\item the inclusion $Y\times \R^*\to \DNC(Y,X)$ and
\item the map $\Theta:\Omega'=\{((x,\xi),\lambda)\in N_X^Y\times \R;\ (x,\lambda\xi)\in U'\}\to \DNC(Y,X)$ defined by $\Theta((x,\xi),0)=((x,\xi),0)$ and  $\Theta((x,\xi),\lambda)=(\theta(x,\lambda \xi),\lambda)\in Y\times \R^*$ if $\lambda \ne 0$.
\end{enumerate}
are diffeomorphisms onto open subsets of $\DNC(Y,X)$.

It is easily shown that $\DNC(Y,X)$ has indeed a smooth structure satisfying these requirements and that this smooth structure does not depend on the choice of $\theta$. (See for example \cite{PCR} for a detailed description of this structure).

In other words, $\DNC(Y,X)$ is obtained by gluing $Y\times \R^*$ with $\Omega'$ by means of the diffeomorphism $\Theta:\Omega'\cap (N_X^Y\times \R^*)\to U\times \R^*$.

Let us recall the following facts which are essential in our construction.

\begin{definitions}
\begin{description}
\item[The gauge action of $\R^*$.] The group $\R^*$ acts on $\DNC(Y,X)$ by $\lambda.(w,t)=(w,\lambda t)$ and $\lambda.((x,\xi),0)=((x,\lambda^{-1}\xi),0)$ \big(with $\lambda,t\in \R^*$, $w\in Y$, $x\in X$ and $\xi\in (N_X^Y)_x$\big).
\item[Functoriality.] Given a commutative diagram of smooth maps $$\xymatrix{X\ar@{^{(}->}[r]\ar[d]_{f_X}&Y\ar[d]^{f_Y}\\X'\ar@{^{(}->}[r]&Y'}$$ where the horizontal arrows are inclusions of submanifolds, we naturally obtain a smooth map $\DNC(f):\DNC(Y,X)\to \DNC(Y',X')$. This map is defined by $\DNC(f)(y,\lambda)=(f_Y(y),\lambda)$ for $y\in Y$ and $\lambda \in \R_*$ and $\DNC(f)(x,\xi,0)=(f_X(x),f_N(\xi),0)$ for $x\in X$ and $\xi\in (N_X^Y)_x=T_xY/T_xX$ where $f_N:N_x\to (N_{X'}^{Y'})_{f_X(x)}=T_{f_X(x)}Y'/T_{f_X(x)}X'$ is the linear map induced by the differential $(df_Y)_x$ at $x$. This map is of course equivariant with respect to the gauge action of $\R^*$. 
\end{description}
\end{definitions}

\begin{remarks}
Let us make a few remarks concerning the DNC construction.\label{atcha}
\begin{enumerate}
\item The map equal to identity on \(X\times \R^*\) and sending \(X\times \{0\}\) to the zero section of \(N_X^Y\) leads to an embedding of \(X\times \R\) into $\DNC(Y,X)$, we may often identify  \(X\times \R\) with its image in $\DNC(Y,X)$. As $\DNC(X,X)=X\times \R$, this corresponds to the naturality of the diagram $$\xymatrix{X\ar@{^{(}->}[r]\ar[d]_{}&X\ar[d]^{}\\X\ar@{^{(}->}[r]&Y}$$

\item  We have a natural smooth map $\pi:\DNC(Y,X)\to Y\times \R$ defined by $\pi(y,\lambda)=(y,\lambda)$ (for $y\in Y$ and $\lambda\in \R^*$) and $\pi((x,\xi),0)=(x,0)$ (for $x\in X\subset Y$ and $\xi\in (N_X^Y)_x$ a normal vector).\label{Claireetc'estGianniquil'adit} This corresponds to the naturality of the diagram $$\xymatrix{X\ar@{^{(}->}[r]\ar[d]_{}&Y\ar[d]^{}\\Y\ar@{^{(}->}[r]&Y}$$

\item If $Y_1$ is an open subset of $Y_2$ such that $X\subset Y_1$, then $\DNC(Y_1,X)$ is an open subset of $\DNC(Y_2,X)$ and $\DNC(Y_2,X)$ is the union of the open subsets $\DNC(Y_1,X)$ and $Y_2\times \R^*$. This reduces to the case when $Y_1$ is a tubular neighborhood - and therefore to the case where $Y$ is (diffeomorphic to) the total space of a real vector bundle over $X$.  In that case one gets \(\DNC(Y,X)=Y\times \R\) and the gauge action of $\R^*$ on \(\DNC(Y,X)=Y\times \R\) is given by $\lambda.((x,\xi),t)= ((x,\lambda^{-1}\xi), \lambda t)$ (with $\lambda\in \R^*$, \(t\in \R\), $x\in X$ and $\xi\in Y_x$).

\item \label{zebut}More generally, let $E$ be (the total space of) a real vector bundle over $Y$. Then $\DNC(E,X)$ identifies with the total space of the pull back vector bundle $\hat\pi^*(E)$ over $\DNC(Y,X)$, where $\hat\pi$ is the composition of $\pi:\DNC(Y,X)\to Y\times \R$ (remark \ref{Claireetc'estGianniquil'adit}) with the projection $Y\times \R\to Y$. The gauge action of $\R^*$ is $\lambda.(w,\xi)= (\lambda .w,\lambda^{-1}\xi)$ for $w\in \DNC(Y,X)$ and $\xi\in E_{\hat\pi(w)}$.

\item \label{fegor}
Let $X_1$ be a (locally closed) smooth submanifold of a smooth manifold $Y_1$ and let $f:Y_2\to Y_1$ be a smooth map transverse to $X_1$. Put $X_2=f^{-1}(X_1)$. Then the normal bundle $N_{X_2}^{Y_2}$ identifies with the pull back of $N_{X_1}^{Y_1}$ by the restriction $X_2\to X_1$ of $f$. It follows that $\DNC(Y_2,X_2)$ identifies with the fibered product $\DNC(Y_1,X_1)\times_{Y_1}Y_2$. 

\item \label{fegor2} More generally, let $Y,Y_1,Y_2$ be smooth manifolds and $f_i:Y_i\to Y$ be smooth maps. Assume that $f_1$ is transverse to $f_2$. Let $X\subset Y$ and $X_i\subset Y_i$ be (locally closed) smooth submanifolds. Assume that $f_i(X_i)\subset X$ and that the restrictions $g_i:X_i\to X$ of $f_i$ are transverse also. We thus have a diagram 
$$\xymatrix{X_1 \ar@{^{(}->}[d]\ar[r]^{g_1}&X\ar@{^{(}->}[d]&X_2\ar@{^{(}->}[d]\ar[l]_{g_2}\\  Y_1\ar[r]^{f_1}&Y&Y_2\ar[l]_{f_2}}$$
Then the maps $\DNC(f_i):\DNC(Y_i,X_i)\to \DNC(Y,X)$ are transverse and the deformation to the normal cone of fibered products $\DNC(Y_1\times_YY_2,X_1\times_XX_2)$ identifies with the fibered product $\DNC(Y_1,X_1)\times_{\DNC(Y,X)}\DNC(Y_2,X_2)$. 
\end{enumerate}
\end{remarks}

\subsection{Blowup constructions}

\nomenclature[G, 13]{$\Blup(Y,X)$}{The blowup of the inclusion of a submanifold $X$ in a manifold $Y$, $\Blup(Y,X)=Y\setminus X \cup \bP(N_X^Y)$}
\nomenclature[G, 14]{$\SBlup(Y,X)$}{The spherical blowup of the inclusion of a submanifold $X$ in a manifold $Y$, $\SBlup(Y,X)=Y \setminus X \cup \bS(N_X^Y)$}

The blowup $\Blup(Y,X)$ is a smooth manifold which is a union of $Y\setminus X$ with the (total space) $\bP(N_X^Y)$ of the projective space of the normal bundle $N_X^Y$ of $X$ in $Y$. We will also use the ``spherical version'' $\SBlup(Y,X)$ of $\Blup(Y,X)$ which is a manifold with boundary obtained by gluing $Y\setminus X$ with the (total space of the)  sphere bundle $\bS(N_X^Y)$. We have an obvious smooth onto map $\SBlup(Y,X)\to \Blup(Y,X)$ with fibers $1$ or $2$ points. These spaces are of course similar and we will often give details in our constructions to the one of them which is the most convenient for our purposes.

\medskip 
We may view $\Blup(Y,X)$ as the quotient space of a submanifold of the deformation to the normal cone $\DNC(Y,X)$ under the gauge action of $\R^*$.

Recall that the group $\R^*$ acts on $\DNC(Y,X)$ by $\lambda.(w,t)=(w,\lambda t)$ and $\lambda.((x,\xi),0)=((x,\lambda^{-1}\xi),0)$ (with $\lambda,t\in \R^*$, $w\in Y$, $x\in X$ and $\xi\in (N_X^Y)_x$). This action is easily seen to be free and (locally \cf remark \ref{RemarqueToutePropre}) proper on the open subset $\DNC(Y,X)\setminus X\times \R$ (see remark \ref{mondo} below). 

\begin{notation}
For every locally closed subset $T$ of $\R$ containing $0$, we define $\DNC_T(Y,X)=Y\times (T\setminus \{0\})\cup N_X^Y\times \{0\}=\pi^{-1}(Y\times T)$ (with the notation of remark \ref{atcha}.\ref{Claireetc'estGianniquil'adit}). It is the restriction of $\DNC(Y,X)$ to $T$. We put  $\DNC_+(Y,X)=\DNC_{\R_+}(Y,X)=Y\times \R_+^*\cup N_X^Y\times \{0\}$.
\end{notation}

\nomenclature[G, 11]{$\DNC_T(Y,X)$}{The restriction $Y\times (T\setminus \{0\})\cup N_X^Y\times \{0\}$ of $\DNC(Y,X)$ to a  closed subset $T$ of $\R$ containing $0$}
\nomenclature[G, 12]{$\DNC_+(Y,X)$}{The restriction $\DNC_{\R_+}(Y,X)$}

\begin{definition}
We put $$\Blup(Y,X)=\Big(\DNC(Y,X)\setminus X\times \R\Big)/\R^*$$
and $$\SBlup(Y,X)=\Big(\DNC_+(Y,X)\setminus X\times \R_+\Big)/\R_+^*.$$
\end{definition}

\begin{remark}\label{mondo}
With the notation of section \ref{sec:DNC}, $\Blup(Y,X)$ is thus obtained by gluing $Y\setminus X=((Y\setminus X)\times \R^*)/\R^*$, with $(\Omega'\setminus (X\times \R))/\R^*$ using the map $\Theta$ which is equivariant with respect to the gauge action of $\R^*$.

Choose a euclidean metric on $N_X^Y$. Let $\bS=\{((x,\xi),\lambda)\in \Omega';\ \|\xi\|=1\}$. The map $\Theta$ induces a diffeomorphism of $\bS/\tau$ with an open neighborhood $\widetilde \Omega$ of $\bP(N_X^Y)$ in $\Blup(Y,X)$ and $\tau $ is the map $((x,\xi),\lambda)\mapsto ((x,-\xi),-\lambda)$.

In this way, with a Riemannian metric on $Y$, we may naturally associate a Riemannian metric on $\Blup(Y,X)$ (using a partition of the identity to glue the metric of $Y\setminus X$ with that of $\widetilde \Omega$).
\end{remark}

Since $\hat\pi:\DNC(Y,X)\to Y$ is invariant by the gauge action of $\R^*$, we obtain a natural smooth map $\tilde \pi:\Blup(Y,X)\to Y$ whose restriction to $Y\setminus X$  is the identity  and  whose restriction to $\bP(N_X^Y)$ is the canonical projection $\bP(N_X^Y)\to X\subset Y$. This map is easily seen to be  proper. 

\begin{remark} \label{Cdedans} Note that, according to remark \ref{atcha}.\ref{zebut}), $\DNC(Y,X)$ canonically identifies with the open subset $\Blup(Y\times \R,X\times \{0\})\setminus \Blup(Y\times \{0\},X\times \{0\})$ of $\Blup(Y\times \R,X\times \{0\})$. Thus, one may think at $\Blup(Y\times \R,X\times \{0\})$ as a ``local compactification'' of $\DNC(Y,X)$ (since the map $\Blup(Y\times \R,X\times \{0\})\to Y\times \R$ is proper).
\end{remark}

\begin{example} In the case where $Y$ is a real vector bundle over $X$, $\Blup(Y,X)$ identifies non canonically with an open submanifold of the bundle of projective spaces $\bP(Y\times \R)$ over $X$. Indeed, in that case $\DNC(Y,X)=Y\times \R$; choose a euclidian structure on the bundle $Y$. Consider the smooth involution $\Phi$ from $(Y\setminus X)\times \R$ onto itself which to $(x,\xi,t)$ associates $(x,\frac{\xi}{\|\xi\|^2},t)$ (for $x\in X,\ \xi\in Y_x,\ t\in \R$). This map transforms the gauge action of $\R^*$ on $\DNC(Y,X)$ into the action of $\R^*$ by dilations on the vector bundle $Y\times \R$ over $X$ and thus defines a diffeomorphism of $\Blup(Y,X)$ into its image which is the open set $\bP(Y\times \R)\setminus X$ where $X$ embeds into $\bP(Y\times \R)$ by mapping $x\in X$ to the line $\{(x,0,t),\ t\in \R\}$.\end{example}

\begin{remark}\label{RemarqueBienPropre}
Since we will apply this construction to morphisms of groupoids that need not be proper, we have to relax properness as in remark \ref{RemarqueToutePropre}: we will say that $f:Y\to X$ is \emph{locally proper} if every point in $X$ has a neighborhood $V$ such that the restriction $f^{-1}(V)\to V$ of $f$ is proper. In particular, if $Y$ is a non Hausdorff manifold and $X$ is a locally closed submanifold of $Y$, then the map $\Blup(Y\times \R,X\times \{0\})\to Y\times \R$ is locally proper
\end{remark}

\subsubsection*{Functoriality}

\begin{definition}[Functoriality] \label{def:nature:blup} Let $$\xymatrix{X\ar@{^{(}->}[r]\ar[d]_{f_X}&Y\ar[d]^{f_Y}\\X'\ar@{^{(}->}[r]&Y'}$$ be a commutative diagram of smooth maps, where the horizontal arrows are inclusions of closed submanifolds. Let  $U_f =\DNC(Y,X)\setminus \DNC(f)^{-1}(X'\times\R)$ be the inverse image by $\DNC(f)$ of the complement in $\DNC(Y',X')$ of the subset $X'\times\R$. We thus obtain a smooth map $\Blup(f):\Blup_f(Y,X)\to \Blup(Y',X')$ where $\Blup_f(Y,X)\subset \Blup(Y,X)$ is the quotient of $U_f$ by the gauge action of $\R^*$.
 \end{definition}

\nomenclature[G, 15]{$\Blup_f(Y,X)$}{The subspace of $\Blup(Y,X)$ on which $\Blup(f):\Blup_f(Y,X)\to \Blup(Y',X')$ can be defined for a smooth map $f:Y\rightarrow Y'$ (with $f(X)\subset X'$)}

In particular,
\begin{enumerate}
\item If $X\subset Y_1$ are (locally) closed submanifolds of a manifold $Y_2$, then $\Blup(Y_1,X)$ is a subma\-nifold of $\Blup(Y_2,X)$.

\item Also, if $Y_1$ is an open subset of $Y_2$ such that $X\subset Y_1$, then $\Blup(Y_1,X)$ is an open subset of $\Blup(Y_2,X)$ and $\Blup(Y_2,X)$ is the union of the open subsets $\Blup(Y_1,X)$ and $Y_2\setminus X$. This reduces to the case when $Y_1$ is a tubular neighborhood.
\end{enumerate}

\subsubsection*{Fibered products}
Let $X_1$ be a (locally closed) smooth submanifold of a smooth manifold $Y_1$ and let $f:Y_2\to Y_1$ be a smooth map transverse to $X_1$. Put $X_2=f^{-1}(X_1)$. Recall from remark \ref{atcha}.\ref{fegor} that in this situation $\DNC(Y_2,X_2)$ identifies with the fibered product $\DNC(Y_1,X_1)\times_{Y_1}Y_2$. Thus $\Blup(Y_2,X_2)$ identifies with the fibered product $\Blup(Y_1,X_1)\times_{Y_1}Y_2$.

\section{Constructions of groupoids}

\subsection{Linear groupoids}\label{section:AlgebraicBlup}

We will encounter groupoids with an extra linear structure which are special cases of \VBGs/ in the sense of Pradines \cite{Pra,Mack}. We will also need to consider the spherical and projective analogues. 

\medskip
Let $E$ be a vector space over a field $\bK$ and let $F$ be a vector sub-space. Let $r,s:E\to F$ be linear retractions of the inclusion $F\to E$.  
\subsubsection{The linear groupoid }

The space $E$ is endowed with a groupoid structure $\cE$ with base $F$.  The range and source maps are $r$ and $s$ and the product is $(x,y)\mapsto (x\cdot y)=x+y-s(x)$ for $(x,y)$ composable, \ie such that $s(x)=r(y)$.  One can easily check: 
\begin{itemize}\item Since \(r\) and \(s\) are linear retractions: $r(x\cdot y)=r(x)$ and $s(x\cdot y)=s(y)$. 
\item If $(x,y,z)$ are composable, then $(x\cdot y)\cdot z=x+y+z-(r+s)(y)=x\cdot (y\cdot z)$. 
\item The inverse of $x$ is $(r+s)(x)-x$.
\end{itemize}

\begin{remarks}\begin{enumerate}\item Note that, given \(E\) and linear retractions \(r\) and \(s\) on \(F\), \(\cE \rightrightarrows F\)  is the only possible linear groupoid structure(\footnote{A linear groupoid is a groupoid $G$ such that $G^{(0)}$ and $G$ are vector spaces and all structure maps (unit, range, source, product) are linear.}) on \(E\) . Indeed, for any \(x\in E\) one must have \(x\cdot s(x)=x\) and \(r(x) \cdot x=x\). By linearity, it follows that for every composable pair \((x,y)=(x,s(x))+(0,y-s(x))\) we have \(x\cdot y=x\cdot s(x)+0\cdot(y-s(x))=x+y-s(x)\).
\item The morphism \(r-s:E/F \rightarrow F\) gives an action of \(E/F\) on \(F\) by addition. The groupoid associated to this action is in fact $\cE $.

\item Given a linear groupoid structure on a vector space $E$, we obtain the ``dual'' linear groupoid  structure  $\cE^*$ on the dual space $E^*$ given by the subspace \(F^\perp =\{\xi \in E^*;\ \xi\vert_F=0\}\) and the two retractions $r^*,s^*:E^*\to F^\perp $ with kernels $(\ker r)^\perp$ and $(\ker s)^\perp$: for \(\xi\in E^*\) and \(x\in E\), \(r^*(\xi)(x)=\xi(x-r(x))\) and similarly \(s^*(\xi)(x)=\xi(x-s(x))\). 
\end{enumerate}\label{UnicLin}
\end{remarks}

\subsubsection{The  projective groupoid}

The multiplicative group $\bK^*$ acts on  $\cE $ by groupoid automorphisms. This action is free on the restriction  $\widetilde{\cE }=\cE  \setminus (\ker r\cup \ker s)$ of the groupoid $\cE $  to the subset $F\setminus \{0\}$ of  $\cE ^{(0)}=F$.

The projective groupoid is the quotient groupoid $\cP E=\widetilde{\cE }/\bK^*$. It is described as follows.

As a set $\cP E=\bP(E)\setminus (\bP(\ker r)\cup \bP(\ker s))$ and \(\cP^{(0)}=\bP(F)\subset \bP(E)$. The source and range maps \(r,s: \cP E\rightarrow \bP(F)\) are those induced by $r,s: E \rightarrow F$. The product of  $x,y\in \cP E$ with $s(x)=r(y)$ is the line $x\cdot y=\{u+v-s(u);\ u\in x,\ v\in y; \ s(u)=r(v)\}\). The inverse of $x\in \cP E$ is $(r+s-id)(x)$.

\begin{remarks} \begin{enumerate} 
 \item When $F$ is just a vector  line, $\cP E$ is a group. Let us describe it:
 
 we have a canonical morphism $h:\cP E\to \bK^*$ defined by $r(u)=h(x)s(u)$ for $u\in x$. The kernel of $h$ is $\bP(\ker (r-s))\setminus \bP(\ker r)$. Note that $F\subset \ker (r-s)$ and therefore $\ker (r-s)\not\subset \ker r$, whence  $\ker r\cap \ker (r-s)$ is a hyperplane in $\ker (r-s)$. The group $\ker h$ is then easily seen to be isomorphic to $\ker (r)\cap \ker (s)$. Indeed, choose a non zero vector \(w\) in \(F\); then the map which assigns to \(u\in \ker (r)\cap \ker (s)\) the line with direction \(w+u\) gives such an isomorphism onto \(\ker h\).
 
Then: \begin{itemize}
\item If $r=s$,  $\cP E$ is isomorphic to the abelian group  \(\ker(r)=\ker(s)\).
\item If $r\ne s$, choose $x$ such that $r$ and $s$ do not coincide on $x$ and let $P$ be the plane $F\oplus x$. The subgroup $\bP(P)\setminus \{\ker r\cap P,\ker s\cap P\}$ of $\cP E$ is isomorphic through $h$ with $\bK^*$. It thus defines a section of $h$. In that case $\cP E$ is the group of dilations $(\ker (r)\cap \ker (s))\rtimes \bK^*$.
\end{itemize}

\item In the general case, let $d\in \bP(F)$. Put $E^d_d=r^{-1}(d)\cap s^{-1}(d)$. \begin{itemize}
\item The stabilizer $(\cP E)_d^d$ is the group $\cP E^d_d=\bP(E^d_d)\setminus (\bP(\ker r)\cup \bP(\ker s))$ described above.

\item The orbit of a line \(d\) is the set of $r(x)$ for $x\in \cP E$ such that $s(x)=d$. It is therefore $\bP(d+r(\ker s))$.

\end{itemize}

\item the following are equivalent:
\begin{enumerate}
\item $(r,s):E\to F\times F$ is onto; 
\item $r(\ker s)=F$;  
\item $(r-s):E/F\to F$ is onto; 
\item the groupoid $\cP E$ has just one orbit.
\end{enumerate}

\item When $r=s$, the groupoid \(\cP E\) is the product of the abelian group $E/F$ by the space $\bP(F)$.

When $r\ne s$, the groupoid $\widetilde{\cE }$ is Morita equivalent to $\cE $  since \(F\setminus\{0\}\) meets all the orbits of  
$\cE $. 

If $\bK$ is a locally compact field and $r\ne s$, the smooth groupoid \(\cP E\) is Morita equivalent to the groupoid crossed-product  $\widetilde{\cE }\rtimes \bK^*$. 

In all cases, when $\bK$ is a locally compact field, \(\cP E\) is amenable.
\end{enumerate}
\end{remarks}

\subsubsection{The spherical groupoid}

If the field is $\R$, we may just take the quotient by $\R_+^*$ instead of $\R^*$. We then obtain similarly the \emph{spherical groupoid} $\cS E=\bS(E)\setminus (\bS(\ker r)\cup \bS(\ker s))$ where \(\cS^{(0)}(E)=\bS(F)\subset \bS(E)$. 

The involutive automorphism \(u \mapsto -u\) of \(E\) leads to a \(\Z/2\Z$ action, by groupoid automorphisms on \(\cS E\). Since this action is free (and proper!), it follows that the quotient groupoid \(\cP E\) and the crossed product groupoid crossed product \(\cS E\rtimes \Z/2\Z\) are Morita equivalent. Thus \(\cS E\) is also amenable.

As for the projective case, if $(r,s):E\to F\times F$ is onto, the groupoid $\cS E$ has just one orbit. The stabilizer of $d\in \bS (F)$ identifies with the group $(\ker r\cup \ker s)\rtimes \R_+^*$, and therefore the groupoid $\cS E$ is Morita equivalent to the group $(\ker r\cup \ker s)\rtimes \R_+^*$.

\subsubsection{Bundle groupoids}\label{tricfion}

We may of course perform the constructions of section \ref{section:AlgebraicBlup} (with say $\bK=\R$) when $E$ is a (real) vector bundle over a space $V$, $F$ is a subbundle and $r,s$ are bundle maps. We obtain respectively vector bundle groupoids, projective bundle groupoids and spherical bundle groupoids: $\cE $,  $(\cP E ,r,s)$  and $(\cS E,r,s)$ which are respectively families of linear, projective  and spherical  groupoids.

\begin{remarks}
\begin{enumerate}
\item A vector bundle groupoid is just given by a bundle morphism $\alpha=(r-s):E/F\to F$. It is isomorphic to the semi direct product $F\rtimes_\alpha E/F$.
\item All the groupoids defined here are amenable, since they are continuous fields of amenable groupoids (\cf \cite[Prop. 5.3.4]{ADR}).
\end{enumerate}
\end{remarks}

The analytic index element $\ind_G\in KK(C_0(\gA^*G),C^*(G))$ of a vector bundle groupoid $G$ is a $KK$-equivalence. 

The groupoid $G$ is a vector bundle $E$ over a locally compact space $X$, $G^{(0)}$ is a vector subbundle $F$ and $G$ is given by a linear bundle map $(r-s):E/F\to F$. 

\begin{proposition}[A Thom-Connes isomorphism]\label{thomconnesthom}
Let $E$ be a vector bundle groupoid. Then $C^*(E)$ is $KK$-equivalent to $C_0(E)$. More precisely, the index $\ind_E:KK(C_0(\gA^*E),C^*(E))$ is invertible.

\begin{proof}
Put $F=E^{(0)}$ and $H=E/F$. Then $H$ acts on $C_0(F)$ and $C^*(E)=C_0(F)\rtimes H$.

We use the equivariant $KK$-theory of Le Gall (\cf \cite{Legall}) $KK_H(A,B)$.

The thom element of the complex bundle $H\oplus H$ defines an invertible element $$t_H\in KK_H(C_0(X),C_0(H\oplus H)).$$ We deduce that, for every pair $A,B$ of $H$ algebras, the morphism $$\tau_{C_0(H)}:KK_H(A,B)\to KK_H(A\otimes _{C_0(X)}C_0(H),B\otimes _{C_0(X)}C_0(H))$$ is an isomorphism. Its inverse is $x\mapsto t_H\otimes \tau_{C_0(H)}(x)\otimes t_H^{-1}$.

Denote by $A_0$ the $C_0(X)$ algebra $A$ endowed with the trivial action of $H$. We have an isomorphism of $H$-algebras $u_A:C_0(H)\otimes _{C(X)} A\simeq C_0(H)\otimes _{C(X)} A_0$. 

It follows  that the restriction map  $KK_H(A,B)$ to $KK_X(A,B)$ (associated to the groupoid morphism $X\to H$) is an isomorphism - compatible of course with the Kasparov product.

Let $v_A\in KK_H(A_0,A)$ be the element whose image in $KK_X(A_0,A)$ is the identity. The descent of $j_H(v_A)\in KK(C_0(H^*)\otimes _{C(X)}A,A\rtimes H)$ is a $KK$-equivalence. The proposition follows by letting $A=C_0(F)$.
\end{proof}
\end{proposition}

\subsubsection{\VBGs}\label{VBgroupoids}

Recall from \cite{Pra, Mack} that a \VBG\ is a groupoid which is a vector bundle over a groupoid $G$. More precisely:

\begin{definition}
Let $\xymatrix{G\ar@<-3pt>[r] \ar@<1pt>[r]^{r_G,s_G}&G^{(0)}}$ be a groupoid. A \VBG\ over $G$ is a vector bundle $p:E\to G$ with a groupoid structure $\xymatrix{E\ar@<-3pt>[r] \ar@<1pt>[r]^{r_E,s_E}&E^{(0)}}$ such that all the groupoid maps are linear vector bundle morphisms.  This means that $E^{(0)}\subset E$ is a vector subbundle of the restriction of $E$ to $G^{(0)}$ and that $r_E,s_E$, $x\mapsto x^{-1}$ and the composition are linear bundle maps. We also assume that the bundle maps $r_E:E\to r_G^*(E^{(0)})$ and  $s_E:E\to s_G^*(E^{(0)})$ are surjective.
\end{definition}

We will come back to {\VBGs} in the appendix.

\subsection{Normal groupoids, deformation groupoids and blowup groupoids} \label{red}

\subsubsection{Definitions}

Let $\Gamma $ be a closed Lie subgroupoid of a Lie groupoid $G$. Using functoriality (\cf Definition \ref{def:nature:blup})  of the \(\DNC\) and \(\Blup\) construction we may construct a normal  and  a blowup groupoid. 
\begin{enumerate} 
\item The normal bundle $N_{\Gamma }^{G}$ carries a Lie groupoid structure with objects $N_{\Gamma ^{(0)}}^{G^{(0)}}$. We denote by $\cN_{\Gamma }^{G}\rightrightarrows N_{\Gamma ^{(0)}}^{G^{(0)}}$ this groupoid. The projection $\cN_{\Gamma }^{G}\to \Gamma $ is a groupoid morphism and it follows that $\cN_{\Gamma }^{G}$ is a \VBG\ over  $\Gamma $.

\item The manifold $\DNC(G,\Gamma )$ is naturally a Lie groupoid (unlike what was asserted in remark 3.19 of \cite{HilsSkMorph}). Its unit space is $\DNC(G^{(0)},\Gamma ^{(0)})$; its source and range maps are $\DNC(s)$ and $\DNC(r)$; the space of composable arrows identifies with $\DNC(G^{(2)},\Gamma ^{(2)})$ and its product with $\DNC(m)$ where $m$ denotes both products  $G^{(2)}\to G$ and is $\Gamma^{(2)}\to \Gamma$.

\nomenclature[G, 17]{$\DNC(G,\Gamma )\rightrightarrows \DNC(G^{(0)}_2,G^{(0)}_1)$}{The deformation groupoid where $\Gamma $ is a closed Lie subgroupoid of a Lie groupoid $G$}

\item\label{jaiplusdidee} The subset $\wDNC(G,\Gamma )=U_r\cap U_s$ of $\DNC(G,\Gamma )$ consisting of elements whose image by $\DNC(r)$ and $\DNC(s)$ is not in $G^{(0)}_1\times \R$ is an open subgroupoid of $\DNC(G,\Gamma )$: it is the restriction of $\DNC(G,\Gamma )$ to the open subspace $\DNC(G^{(0)},G^{(0)}_1)\setminus G^{(0)}_1\times \R$.

\item The group $\R^*$ acts on $\DNC(G,\Gamma )$ via the gauge action by groupoid morphisms. Its action on $\wDNC(G,\Gamma )$ is (locally) proper. Therefore the open subset $\Blup_{r,s}(G,\Gamma )=\wDNC(G,\Gamma )/\R^*$ of $\Blup(G,\Gamma )$  inherits a groupoid structure as well: its space of units is $\Blup(G^{(0)}_2,G^{(0)}_1)$; its source and range maps are \(\Blup(s)\) and \(\Blup(r)\) and the product is $\Blup(m)$.

\nomenclature[G, 18]{$\Blup_{r,s}(G,\Gamma )\rightrightarrows \Blup(G^{(0)},\Gamma ^{(0)})$}{The blowup groupoid $\Blup_{r}(G,\Gamma ) \cap \Blup_{s}(G,\Gamma )$ where $\Gamma $ is a closed Lie subgroupoid of a Lie groupoid $G$}

\item 
In the same way, we define the groupoid $\SBlup_{r,s}(G,\Gamma )$. It is the quotient of the restriction $\widetilde \DNC_+(G,\Gamma )$ of $\widetilde \DNC(G,\Gamma )$ to $\R_+$ by the action of $\R_+^*$. Similarly $\DSBlup_{r,s}(G,\Gamma )$ will be the quotient of $\widetilde \DNC(G,\Gamma )$ by the action of $\R_+^*$. This is the ``double'' of the Lie groupoid with boundary $\SBlup_{r,s}(G,\Gamma )$.
\nomenclature[G, 19]{$\SBlup_{r,s}(G,\Gamma )$}{ The spherical version of $\Blup_{r,s}(G,\Gamma )$}
\nomenclature[G, 20]{$\wDNC(G,\Gamma )$}{The open subgroupoid of $\DNC(G,\Gamma )$ of $\DNC(G,\Gamma )$ consisting of elements whose image by $\DNC(r)$ and $\DNC(s)$ is not in $G^{(0)}_1\times \R$}
 \nomenclature[G, 20]{$\wDNC_+(G,\Gamma )$}{The restriction of $\widetilde \DNC(G,\Gamma )$ to $\R_+$}
\end{enumerate}

An analogous result about the groupoid structure on $\Blup_{r,s}(G,\Gamma )$ in the case of $\Gamma ^{(0)}$ being a hypersurface of $G^{(0)}$ can be found in \cite[Theorem 2.8]{GualtieriLi1} (\cf also \cite{GualtieriLi2}).

\subsubsection{Algebroid and anchor}

The (total space of the) Lie algebroid $\gA \Gamma $ is a closed submanifold (and a subbundle) of $\gA G$. The Lie algebroid of $\DNC(G,\Gamma )$ is $\DNC(\gA G,\gA \Gamma )$. Its anchor map is $\DNC(\natural_{G}):\DNC(\gA G,\gA \Gamma )\to \DNC(TG^{(0)},T\Gamma ^{(0)})$.

The groupoid $\DNC(G,\Gamma )$ is the union of its open subgroupoid $G\times \R^*$ with its closed Lie sub-groupoid $\cN_{\Gamma }^{G}$. The algebroid of $G\times \R^*$ is $\gA G\times \R^*$ and the anchor is just the map $\natural _{G}\times \id:\gA G\times \R^*\to T(G^{(0)}\times \R_+^*)$.

\subsubsection{Morita equivalence}\label{Moreq2}

Let $G_1\rightrightarrows G_1^{(0)}$ and $G_2\rightrightarrows G_2^{(0)}$ be Lie groupoids, $\Gamma_1\subset G_1$ and $\Gamma_2\subset G_2$ Lie subgroupoids. A Morita equivalence of the pair $(\Gamma_1\subset G_1)$  with the pair $(\Gamma_2\subset G_2)$ is given by a pair $(X\subset Y)$ where $Y$ is a linking manifold which is a Morita equivalence between $G_1$ and $G_2$ and $X\subset Y$ is a submanifold of $Y$ such that the maps $r,s$ and products of $Y$ (see page \pageref{Moreq1}) restrict to a Morita equivalence $X$ between $\Gamma_1$ and $\Gamma_2$.

Then, by functoriality, \begin{itemize}
\item $\DNC(Y,X)$ is a Morita equivalence between $\DNC(G_1,\Gamma_1)$ and $\DNC(G_2,\Gamma_2)$, 
\item$\DNC_+(Y,X)$ is a Morita equivalence between $\DNC_+(G_1,\Gamma_1)$ and $\DNC_+(G_2,\Gamma_2)$,
\item $\Blup_{r,s}(Y,X)$ is a Morita equivalence between $\Blup_{r,s}(G_1,\Gamma_1)$ and $\Blup_{r,s}(G_2,\Gamma_2)$, 
\item $\SBlup_{r,s}(Y,X)$ is a Morita equivalence between $\SBlup_{r,s}(G_1,\Gamma_1)$ and $\SBlup_{r,s}(G_2,\Gamma_2)$...
\end{itemize}

Note that if $Y$ and $X$ are sub-Morita equivalences, the above linking spaces are also sub-Morita equivalences.

\subsubsection{Groupoids on manifolds with boundary}
\label{BoundLiePoid}

Let $M$ be a manifold and $V$ an hypersurface in $M$ and suppose that $V$ cuts $M$ into two manifolds with boundary $M=M_- \cup M_+$ with $V=M_- \cap M_+$. Then by considering a tubular neighborhood of $V$ in $M$, $\DNC(M,V)=M\times \R^* \cup \cN_V^M\times \{0\}$ identifies with $M\times \R$, the quotient $\wDNC(M,V)/\R_+^*$ identifies with two copies of $M$ and $\SBlup(M,V)$  identifies with the disjoint union $M_- \sqcup M_+$. Under this last identification, the class under the gauge action of a normal vector in 
$\cN_V^M\setminus V\times \{0\}$ pointing in the direction of $M_+$ is an element of $V\subset M_+$.

\medskip

Let $M_b$ be manifold with boundary $V$. A \emph{piece of  Lie groupoid} is the restriction $G=\widetilde G_{M_b}^{M_b}$ to $M_b$ of a Lie groupoid $\widetilde G\rightrightarrows M$ where $M$ is a neighborhood of $M_b$ and $G$ is a groupoid without boundary. Note that when the boundary $V$ is transverse to the groupoid $\widetilde G$, $G$ is in fact a manifold with corners.

\smallskip With the above notation, since $V$ is of codimention $1$ in $M$,  $\SBlup(M,V)=M_b\sqcup M_-$ where $M_-= M\setminus \rM$ is the complement in $M$ of the interior $\rM=M_b\setminus V$ of $M_b$ in $M$. 

Let then $\Gamma\rightrightarrows V$ be a Lie subgroupoid of $\widetilde G$.

We may construct $\SBlup_{r,s}(\widetilde G,\Gamma)$ and consider its restriction to the open subset $M_b$ of $\SBlup(M,V)$. We thus obtain a longitudinally smooth groupoid that will be denoted $\SBlup_{r,s}(G,\Gamma)$.

Note that the groupoid $\SBlup_{r,s}(G,\Gamma) \rightrightarrows M_b$ is the restriction to $M_b$ of a Lie groupoid $\cG \rightrightarrows M$ for which $M_b$ is saturated. Indeed $\SBlup_{r,s}(G,\Gamma)$ is an open subgroupoid of $\SBlup_{r,s}(\widetilde G,\Gamma)\rightrightarrows M_b\sqcup M_-$ which is a piece of the Lie groupoid  $\wDNC(\widetilde G,\Gamma)/\R_+^*\rightrightarrows \wDNC(M,V)/\R_+^*\simeq   M\sqcup   M$. We may then let $\cG$ be the restriction of $\wDNC( M,V)/\R_+^*$ to one of the copies of $ M$.

\smallskip

In this way, we may treat by induction a finite number of boundary components \ie a groupoid on a manifold with corners.

\begin{remarks}
\begin{enumerate}
\item If $M$ is a manifold with boundary $V$ and $G=M\times M$ is the pair groupoid, then \(\SBlup_{r,s}(G,V)\) is in fact the groupoid associated with the $0$ calculus in the sense of Mazzeo (\cf \cite{Mazzeo88, Melbook, MM}), \ie the canonical pseudodifferential calculs associated with \(\SBlup_{r,s}(G,V)\) is the Mazzeo-Melrose's $0$-calculus. Indeed, the sections of the algebroid of \(\SBlup_{r,s}(G,V)\) are exactly the vector fields of $M$ vanishing at the boundary $V$, \ie those generating the $0$-calculus.

\item In a recent paper \cite{NistorLast}, an alternative description of $\SBlup_{r,s}(G,V)$ is given under the name of {\sl edge modification} for \(G\) along the ``\(\gA G\)-tame manifold" $V$, thus in particular \(V\) is transverse to \(G\). This is essentially the gluing construction described in \ref{subsecsubmorita} below. 

\end{enumerate}

\end{remarks}

\subsection{Examples of normal groupoids,  deformation groupoids and blowup groupoids}\label{sefaire}

We examine some particular cases of inclusions of groupoids $G_1\subset G_2$. The various constructions of deformation to the normal cone and blow-up allow us to recover many well known groupoids. As already noted in the introduction, our constructions immediately extend to the case where we restrict to a closed saturated subset of a smooth groupoid, in particular for manifolds with corners.

\subsubsection{Inclusion $F\subset E$ of vector spaces}\label{uge}
Let $E$ be a real vector space - considered as a group - and $F$ a vector subspace of $E$. The inclusion of groups $F\to E$ gives rise to a groupoid $\DNC(E,F)$. Using any supplementary subspace of $F$ in $E$, we may identify the groupoid $\DNC(E,F)$ with $E\times \R\rightrightarrows \R$. Its $C^*$-algebra identifies then with $C_0(E^*\times \R)$.

More generally, if $F$ is a vector-subbundle of a vector bundle $E$ over a manifold $M$ (considered as a family of groups indexed by $M$), then the groupoid $\DNC(E,F)\rightrightarrows M\times \R$ identifies with $E\times \R$ and its $C^*$-algebra is $C_0(E^*\times \R)$.

Let $p_E:E\to M$ be a vector bundle over a manifold $M$ and let $V$ be a submanifold of $M$. Let $p_F:F\to V$ be a subbundle of the restriction of $E$ to $V$. We use a tubular construction and find an open subset $U$ of $M$ which is a vector bundle $\pi:Q\to V$. Using $\pi$, we may extend $F$ to a subbundle $F_U$ of the restriction to $F$ on $U$. Using that, we may identify $\DNC(E,F)$ with the open subset $E\times \R^*\cup p_E^{-1}(U)\times \R$ of $E\times \R$. Its $C^*$-algebra  identifies then with $C_0(E^*\times \R^*\cup p_{E^*}^{-1}(U)\times \R)$.

\subsubsection{Inclusion $G^{(0)}\subset G$: adiabatic groupoid}

The deformation to the normal cone $\DNC(G,G^{(0)})$ is the adiabatic groupoid $G_{ad}$ (\cite{MP,NWX}), it is obtained by using the deformation to the normal cone construction for the inclusion of $G^{(0)}$ as a Lie subgroupoid of $G$. The normal bundle $N_{G^{(0)}}^G$  is the total space of the Lie algebroid \(\gA (G)\) of \(G\). Note that its groupoid structure coincides with its vector bundle structure. Thus,
\[\DNC(G,G^{(0)})=G\times \R^* \cup \gA(G)\times \{0\} \rightrightarrows G^{(0)}\times \R \ .\]  

The particular case where $G$ is the pair groupoid $M\times M$ is the original construction of the ``tangent groupoid'' of Alain Connes (\cite{ConnesNCG}).

\smallskip Note that   $\Blup(G^{(0)},G^{(0)})=\emptyset=\Blup_{r,s}(G,G^{(0)})$.

\subsubsection{Gauge adiabatic groupoid}\label{subsecgag}

Start with a Lie groupoid $G\rightrightarrows V$.

\medskip

Let $G\times (\R\times \R)\overset {\tilde r,\tilde s}\rightrightarrows V\times \R$ be the product groupoid of $G$ with the pair groupoid over $\R$. \\ First notice that since \(V\times \{0\}\) is a codimension \(1\) submanifold in \(V\times \R\), \(\SBlup(V\times \R,V\times \{0\})\) is canonically isomorphic to \(V\times (\R_-\sqcup \R_+)\).
Then $\SBlup_{\tilde r,\tilde s}(G\times (\R\times \R),V\times \{(0,0)\})_{V\times \R_+}^{V\times \R_+}$ is the semi-direct product groupoid $G_{ad}(V\times \R_+)\rtimes \R_+^*$: \[\SBlup_{\tilde r,\tilde s}(G\times (\R\times \R),V\times \{(0,0)\})_{V\times \R_+}^{V\times \R_+}= G_{ad}(V\times \R_+)\rtimes \R^*\rightrightarrows V\times \R_+ \ .\]

In other words, $\SBlup_{\tilde r,\tilde s}(G\times (\R\times \R),V\times \{(0,0)\})_{V\times \R_+}^{V\times \R_+}$ is the gauge adiabatic groupoid used in \cite{DS1}.

\smallskip Indeed, as $G\times (\R\times \R)$ is a vector bundle over $G$, $\DNC(G\times (\R\times \R),V\times \{(0,0)\})\simeq \DNC(G,V)\times \R^2$ (remark \ref{atcha}.\ref{zebut}). Under this identification, the gauge action of $\R^*$ is given by $\lambda.(w,t,t')=(\lambda.w,\lambda^{-1}t,\lambda^{-1}t')$.
The maps $\DNC(\tilde s)$ and $\DNC(\tilde r)$ are respectively $(w,t,t')\mapsto (\DNC(s)(w),t')$ and $(w,t,t')\mapsto (\DNC(r)(w),t)$.  It follows that $\SBlup_{\tilde r,\tilde s}(G\times (\R\times \R),V\times \{(0,0)\})$ is the quotient by the diagonal action of $\R_+^*$ of the open subset $\DNC(G,V)\times (\R^*)^2$ of $\DNC_+(G,V)\times \R^2$. 

According to the description of the groupoid of a group action on a groupoid given in section \ref{croise} it is isomorphic to $\DNC(G,V)_+\rtimes \R_+^* \times \{-1,+1\}^2$ where $\{-1,+1\}^2$ is the pair groupoid over $\{-1,+1\}$.

\subsubsection{Inclusion of a transverse submanifold of the unit space}\label{subsecsubmorita}

Let $G$ be a Lie groupoid with set of objects  \(M=G^{(0)} \) and let $V$ be a submanifold of $M$. We now study the special case of normal and blowup groupoids $\DNC(G,V)$ and $\Blup_{r,s}(G,V)$ (as well as $\SBlup_{r,s}(G,V)$) associated to the groupoid morphism $V\to G$.  

\smallskip Put $\rM=M\setminus V$.  Let $N=N_V^{G}$ and $N'=N_V^M$ be the normal bundles. We identify $N'$ with a subbundle of $N$ by means of the inclusion $M\subset G$. The submersions $r,s:G\to M$ give rise to bundle morphisms $r^N,s^N:N\to N'$ that are sections of the inclusion $N'\to N$. By construction, using remark \ref{UnicLin}.a), the groupoid $\DNC(G,V)$ is the union of $G\times \R^*$ with the family of linear groupoids $\cN_{r^N,s^N}(N)$. It follows that $\Blup_{r,s}(G,V)$ is the union of $G_{\rM}^{\rM}$ with the family $(\cP N,r^N,s^N)$ of projective groupoids.

\smallskip 
If $V$ is transverse to $G$, the bundle map $r^N-s^N:N=N_V^{G}\to N'=N_V^M$ is surjective; it follows that \begin{itemize}
\item $\cL_{r^N,s^N}(N)$ identifies with the pull-back groupoid $\big(\gA(G_V^V)\big)_q^q$ where $q:N'\to V$ is the projection,
\item  $(\cP N,r^N,s^N)$  with the pull-back groupoid $\big(\gA(G_V^V)\rtimes \R^*\big)_\rho^\rho$ where $\rho:\bP (N')\to V$ is the projection,
\item  $(\cS N,r^N,s^N)$  with the pull-back groupoid $\big(\gA(G_V^V)\rtimes \R_+^*\big)_p^p$ where $p:\bS (N')\to V$ is the projection.
\end{itemize}

Let us give a local description of these groupoids in the neighborhood of the transverse submanifold $V$. \\
Put $\ronde{G}=G_{M\setminus V}^{M\setminus V}$. Upon arguing locally, we can assume that $V$ is compact. 

By Remark \ref{soeur}, $V$ admits a tubular neighborhood $W\simeq N_V^M$ such that $G_W^W$ is the pull back of $G_V^V$ by the retraction $q:W\to V$. 

\smallskip The normal groupoid $\DNC(G_W^W,V)$ identifies with the pull back groupoid  $(\DNC(G_V^V,V))_q^q$ of the adiabatic deformation $\DNC(G_V^V,V)=(G_V^V)_{ad}$ by the map $q:N_V^M\to V$.

\smallskip The (spherical) blowup groupoid $\SBlup_{r,s}(G_W^W,V)$ identifies with the pull back groupoid $(\DNC_+(G_V^V,V)\rtimes \R_+^*)_p^p$ of the gauge adiabatic deformation $\DNC_+(G_V^V,V)\rtimes \R_+^*=(G_V^V)_{ga}$ by the map $p:\bS N_V^M\to V$.

In order to get $\SBlup_{r,s}(G,V)$, we then may glue $(\DNC_+(G_V^V,V)\rtimes \R_+^*)_p^p$ with $\ronde{G}$ in their common open subset $\big((G_V^V)_q^q\big)_{W\setminus V}^{W\setminus V} \simeq G_{W\setminus V}^{W\setminus V}$.

\subsubsection{Inclusion $G_V^V\subset G$ for a transverse hypersurface $V$ of $G$: $b$-groupoid}

If $V$ is a hypersurface of $M$, the blowup $\Blup(M\times M,V\times V)$ is just the construction of Melrose of the $b$-space. Its open subspace $\Blup_{r,s}(M\times M,V\times V)$ is the associated groupoid of Monthubert \cite{Month1, Month2}. Moreover, if $G$ is a groupoid on $M$ and $V$ is transverse to  $G$ we can form the restriction groupoid $G_V^V\subset G$ which is a submanifold of $G$. The corresponding blow up construction $\Blup_{r,s}(G,G_V^V)$ identifies with the fibered product $\Blup_{r,s}(M\times M,V\times V)\times _{M\times M}G$ (\cf remark \ref{atcha}.\ref{fegor}).

\smallskip Iterating (at least locally) this construction, we obtain the $b$-groupoid of Monthubert for manifolds with corners - \cf \cite{Month1, Month2}.

\begin{remark}
The groupoid \(\Blup_{r,s}(G,V)\) corresponds to inflating all the distances when getting close to $V$.

\smallskip The groupoid $\Blup_{r,s}(G,G_V^V)$ is a kind of \emph{cylindric deformation groupoid} which is obtained by  pushing the boundary $V$ at infinity but keeping the distances along $V$ constant.
\end{remark}

\begin{remark}
Intermediate examples between these two are given by a subgroupoid $\Gamma\rightrightarrows V$ of $G_V^V$. 

In the case where $G=M\times M$, such a groupoid $\Gamma$ is nothing else than the holonomy groupoid $Hol(V,\cF)$ of a regular foliation $\cF$ of $V$ (with trivial holonomy groups). The groupoid $\SBlup_{r,s}(M\times M,Hol(V,\cF))$ is a holonomy groupoid of a singular foliation of $M$: the sections of its algebroid. Its leaves are $M\setminus V$ and the leaves of $(V,\cF)$. The corresponding calculus, when $M$ is a manifold with a boundary $V$ is Rochon's generalization (\cite{Rochon}) of the $\phi $ calculus of Mazzeo and Melrose (\cite{MM2}).

\smallskip Iterating (at least locally) this construction, we obtain the holonomy groupoid associated to a stratified space in \cite{DLR}.
\end{remark}

\subsubsection{Inclusion $G_V^V\subset G$ for a saturated submanifold $V$ of $G$: shriek map for immersion}

Suppose now that  \(V\) is saturated thus \(G_V^V=G_V=G^V\). \\ In such a situation the groupoid \(G_V^V\) acts on the normal bundle \(N_{G_V^V}^G=r^*(N_V^{G^{(0)}}) \) and \(\DNC(G,G_V^V)\rightrightarrows \DNC(G^{(0)},V)\) coincides with the normal groupoid of the immersion  \(\varphi: G_V^V \rightarrow G\). This construction was defined in the case of foliation groupoids in \cite[section 3]{HilsSkMorph} and was used in order to define \(\varphi_!\) as associated \(KK\)-element.

\subsubsection{Inclusion $G_1\subset G_2$ with $G_1^{(0)}=G_2^{(0)}$}

This is the case for the tangent and adiabatic groupoid discussed above. Two other kinds of this situation\footnote{Note that in this case $\Blup(G_2^{(0)},G_1^{(0)})=\emptyset$, whence $\Blup_{r,s}(G_2,G_1)=\emptyset$.} can be encountered in the literature:
\begin{enumerate}
\item The case of a subfoliation $\cF_1$ of a foliation $\cF_2$ on a manifold $M$: shriek map for submersion. As pointed out in remark 3.19 of \cite{HilsSkMorph} the corresponding deformation groupoid $\DNC(G_2,G_1)$ gives an alternative construction of the element $\varphi_!$ where $\varphi:M/\cF_1\to M/\cF_2$ is a submersion of leaf spaces.

\item The case of a subgroup of a Lie group. \begin{itemize}
\item If $K$ is a maximal compact subgroup of a reductive Lie group $G$, the connecting map associated to the exact sequence of $\DNC(G,K)$ is the Dirac extension mapping the twisted $K$-theory of $K$ to the $K$-theory of $C^*_r(G)$ (see \cite{HigsonMackey}).
\item In the case where $\Gamma$ is a dense (non amenable) countable subgroup of a compact Lie group $K$, the groupoid $\DNC(K,\Gamma)$ was used in \cite{HLS} in order to produce a Hausdorff groupoid for which the Baum-Connes map is not injective.
\end{itemize}

\end{enumerate}

\subsubsection{Wrong way functoriality}

Let $f:G_1\to G_2$ be a morphism of Lie groupoids. If $f$ is an (injective) immersion the construction of  $\DNC_+(G_2,G_1)$ gives rise to a short exact sequence
$$0\longrightarrow C^*(G_2\times \R_+^*)\longrightarrow C^*(\DNC_+(G_2,G_1))\longrightarrow C^*(\cN_{G_1}^{G_2})\longrightarrow 0.$$ 
and consequently to a 
connecting map from the $K$-theory of the $C^*$-algebra of the groupoid $\cN_{G_1}^{G_2}$, which is a \VBG\ over $G_1$, to the $K$-theory of $C^*(G_2)$. This wrong way functoriality map will be discussed in the next section.

\medskip More generally let $\cG=G_1^{(0)}\times G_2\times G_1^{(0)}$ be the product of $G_2$ by the pair groupoid of $G_1^{(0)}$. Assume that the map $x\mapsto (r(x),f(x),s(x))$ is an immersion from $G_1\to \cG$. 

The above construction gives a map from $K_*(C^*(\cN_{G_1}^{\cG}))$ to $K_*(C^*(\cG))$ which is isomorphic to $K_*(C^*(G_2))$ since the groupoids $G_2$ and $\cG$ are canonically equivalent.

\section{\texorpdfstring{The $C^*$-algebra of a deformation and of a blowup groupoid, full symbol and index map}{}}\label{Section6}

Let $G\rightrightarrows M$ be a Lie groupoid and $\Gamma\rightrightarrows V$ a Lie subgroupoid of $G$. The groupoids $\DNC_+(G,\Gamma)$ and $\SBlup_{r,s}(G,\Gamma)$ that we constructed admit the closed saturated subsets $N_V^M\times \{0\}$ and $\bS N_V^M$ respectively. 
In order to shorten the notation we put $\rM=M\setminus V$ and the corresponding full symbol algebras \begin{itemize}
\item $\Sigma_{\DNC_+}(G,\Gamma)=\Sigma^{M\times \R_+^*}(\DNC_+(G,\Gamma))$;
\item $\Sigma_{\wDNC_+}(G,\Gamma)=\Sigma^{\rM\times \R_+^*}(\wDNC_+(G,\Gamma))$;
\item  $\Sigma_{\SBlup}(G,\Gamma)=\Sigma^{\rM}(\SBlup_{r,s}(G,\Gamma))$.
\end{itemize}

\nomenclature[H, 20]{$\Sigma_{\DNC_+}(G,\Gamma)$}{The algebra $\Sigma^{M\times \R_+^*}(\DNC_+(G,\Gamma))$}
\nomenclature[H, 21]{$\Sigma_{\wDNC_+}(G,\Gamma)$}{The algebra $\Sigma^{\rM\times \R_+^*}(\wDNC_+(G,\Gamma))$}
\nomenclature[H, 22]{$\Sigma_{\SBlup}(G,\Gamma)$}{The algebra $\Sigma^{\rM}(\SBlup_{r,s}(G,\Gamma))$}

They give rise to the exact sequences 
$$0\longrightarrow C^*(G_{\rM}^{\rM})\longrightarrow C^*(\SBlup_{r,s}(G,\Gamma))\longrightarrow C^*(\cS N_\Gamma^G)\longrightarrow 0\eqno {E^\partial_{\SBlup}}$$
and
$$0\longrightarrow C^*(G\times \R_+^*)\longrightarrow C^*(\DNC_+(G,\Gamma))\longrightarrow C^*(\cN_{\Gamma}^{G})\longrightarrow 0\eqno {E^\partial_{\DNC_+}}$$
of groupoid $C^*$-algebras as well as index type ones 
$$0\longrightarrow C^*(G_{\rM}^{\rM})\longrightarrow \Psi^*(\SBlup_{r,s}(G,\Gamma))\longrightarrow \Sigma_{\SBlup}(G,\Gamma)\longrightarrow 0\eqno {E^{\wind}_{\SBlup}}$$
and
$$0\longrightarrow C^*(G\times \R_+^*)\longrightarrow \Psi^*(\DNC_+(G,\Gamma))\longrightarrow \Sigma_{\DNC_+}(G,\Gamma)\longrightarrow 0\eqno {E^{\wind}_{\DNC_+}}$$
We will compare the exact sequences given by $\DNC$ and by $\SBlup$.

If $V$ is $\gA G$-small (see notation \ref{notation6} below), we will show that, in a sense, $\DNC$ and $\SBlup$ give rise to equivalent exact sequences - both for the ``connecting'' ones and for the ``index'' ones.

We will then compare these elements with a coboundary construction.

We will compute these exact sequences when $\Gamma=V\subset M$. Finally, we will study a refinement of these constructions using relative $K$-theory.

\subsection{``DNC'' versus ``Blup''}\label{section:DNCversusBlup}

Let $\Gamma$ be a submanifold and a subgroupoid of a Lie-groupoid $G$. We will further assume that the groupoid $\Gamma$ is amenable. Put $M=G^{(0)}$ and $V=\Gamma^{(0)}$. Put also $\rM=M\setminus V$ and let $\rcN_{\Gamma}^{G}$ be the restriction of the groupoid $\cN_{\Gamma}^{G}$ to the open subset $\rN_V^M=N_V^M\setminus V$ of its unit space $N_V^M$.

\subsubsection{The connecting element}

As the groupoid $\Gamma$ is amenable we have exact sequences both for the reduced and for the maximal $C^*$-algebras: 
$$0\longrightarrow C^*(G_{\rM}^{\rM})\longrightarrow C^*(\SBlup_{r,s}(G,\Gamma))\longrightarrow C^*(\cS N_\Gamma^G)\longrightarrow 0\eqno {E^\partial_{\SBlup}}$$
and
$$0\longrightarrow C^*(G\times \R_+^*)\longrightarrow C^*(\DNC_+(G,\Gamma))\longrightarrow C^*(\cN_{\Gamma}^{G})\longrightarrow 0\eqno {E^\partial_{\DNC_+}}$$
By amenability, these exact sequences admit completely positive cross sections and therefore define elements
 $\partial_{\SBlup}^{G,\Gamma}= \partial_{\SBlup_{r,s}(G,\Gamma)}^{\rM} \in KK^1(C^*(\rcN_{\Gamma}^{G}/\R_+^*),C^*(G_{\rM}^{\rM}))$ and  $\partial_{\DNC_+}^{G,\Gamma} = \partial_{\DNC_+(G,\Gamma)}^{M\times \R^*_+} \in KK^1(C^*(\cN_{\Gamma}^{G}),C^*(G\times \R_+^*))$. 

\nomenclature[I, 20]{$\partial_{\SBlup}^{G,\Gamma}$, $\partial_{\DNC_+}^{G,\Gamma}$, $\partial_{\wDNC_+}^{G,\Gamma}$}{Respectively the element $\partial_{\SBlup_{r,s}(G,\Gamma)}^{\rM}$, $\partial_{\DNC_+(G,\Gamma)}^{M\times \R^*_+} $ and $\partial_{\wDNC_+(G,\Gamma)}^{\rM\times \R^*_+}$}
 
With the notation of section \ref{red}, write $\DNC_+$ for $\DNC$ restricted to $\R_+$ and ${\wDNC_+}$ for $\wDNC$ restricted to $\R_+$.

By section \ref{ctoi},  we have a diagram where the vertical arrows are $KK^1$-equivalences and the squares commute in $KK$-theory.
$$\xymatrix{0\ar[r] &C^*(G_{\rM}^{\rM})\ar[r]\ar@{-}[d]^{\beta'}&C^*(\SBlup_{r,s}(G,\Gamma))\ar[r]\ar@{-}[d]^{\beta}& C^*(\rcN_{\Gamma}^{G}/\R_+^*)\ar[r]\ar@{-}[d]^{\beta''}&0 & \hspace{1cm} {E^\partial_{\SBlup}}\\
0\ar[r] &C^*(G_{\rM}^{\rM}\times \R_+^*)\ar[r]&C^*(\wDNC_+(G,\Gamma))\ar[r]& C^*(\rcN_{\Gamma}^{G})\ar[r]&0 &   \hspace{1cm} {E^\partial_{\wDNC_+}} }$$
Denote by $\partial_{\wDNC_+}^{G,\Gamma}=\partial_{\wDNC_+(G,\Gamma)}^{\rM\times \R^*_+}$ the connecting element associated to $E^\partial_{\wDNC_+}$. We thus have, according to proposition \ref{etK1}:

\begin{fact}\label{fact1}
$\partial_{\SBlup}^{G,\Gamma}\otimes \beta'=-\beta''\otimes \partial_{\wDNC_+}^{G,\Gamma}\in KK^1(C^*(\cS N_{\Gamma}^G),C^*(G_{\rM}^{\rM}\times \R_+^*))$.
\end{fact}

\bigskip
We also have a commutative diagram where the vertical maps are inclusions:
\begin{equation}\label{diagram1}
\begin{gathered} 
\xymatrix{0\ar[r] &C^*(G_{\rM}^{\rM}\times \R_+^*)\ar[r]\ar[d]^{j'}&C^*({\wDNC_+}(G,\Gamma))\ar[r]\ar[d]^{j}& C^*(\rcN_{\Gamma}^{G})\ar[r]\ar[d]^{j''}&0\\
0\ar[r] &C^*(G\times \R_+^*)\ar[r] &C^*(\DNC_+(G,\Gamma))\ar[r] &C^*(\cN_{\Gamma}^{G})\ar[r]&0}
\end{gathered} 
\end{equation}

We thus find
\begin{fact}\label{fact2}
$(j'')^*(\partial_{\DNC_+}^{G,\Gamma })=j'_*(\partial_{\wDNC_+}^{G,\Gamma })\in KK^1(C^*(\rcN_{\Gamma}^{G}),C^*(G\times \R_+^*))$.
\end{fact}

 \subsubsection{The full symbol index}

 We now compare the elements $\wind_{\SBlup}^{G,\Gamma} =\wind_{full}^{\rM}(\SBlup_{r,s}(G,\Gamma))   \in KK^1(\Sigma_{\SBlup}(G,\Gamma),C^*(G_{\rM}^{\rM}))$ and $\wind _{\DNC_+}^{G,\Gamma} =\wind_{full}^{M\times \R_+^*}(\DNC_+(G,\Gamma)) \in KK^1(\Sigma_{\DNC_+}(G,\Gamma),C^*(G\times \R_+^*))$ defined by the semi-split exact sequences:
 $$0\longrightarrow C^*(G_{\rM}^{\rM})\longrightarrow \Psi^*(\SBlup_{r,s}(G,\Gamma))\longrightarrow \Sigma_{\SBlup}(G,\Gamma)\longrightarrow 0\eqno {E^{\wind }_{\SBlup}}$$
and
$$0\longrightarrow C^*(G\times \R_+^*)\longrightarrow \Psi^*(\DNC_+(G,\Gamma))\longrightarrow \Sigma_{\DNC_+}(G,\Gamma)\longrightarrow 0\eqno {E^{\wind }_{\DNC_+}}$$
 
\nomenclature[I, 21]{$\wind_{\SBlup}^{G,\Gamma}$, $\wind _{\DNC_+}^{G,\Gamma}$, $\wind_{\wDNC_+}^{G,\Gamma}$}{Respectively the elements $\wind_{full}^{\rM}(\SBlup_{r,s}(G,\Gamma))$, $\wind_{full}^{M\times \R_+^*}(\DNC_+(G,\Gamma))$ and $\wind_{full}^{\rM\times \R_+^*}(\wDNC_+(G,\Gamma))$}

By prop. \ref{betasigma},  we have a diagram where the vertical arrows are $KK^1$-equivalences and the squares commute in $KK$-theory.
$$\xymatrix{0\ar[r] &C^*(G_{\rM}^{\rM})\ar[r]\ar@{-}[d]^{\beta'}&\Psi^*(\SBlup_{r,s}(G,\Gamma))\ar[r]\ar@{-}[d]^{\beta_\Psi}&  \Sigma_{\SBlup}(G,\Gamma)\ar[r]\ar@{-}[d]^{\beta_{\Sigma}}&0\\
0\ar[r] &C^*(G_{\rM}^{\rM}\times \R_+^*)\ar[r]&\Psi^*(\wDNC_+(G,\Gamma))\ar[r]& \Sigma_{\wDNC_+}(G,\Gamma)\ar[r]&0}$$
We let $\wind_{\wDNC_+}^{G,\Gamma}=\wind_{full}^{\rM\times \R_+^*}(\wDNC_+(G,\Gamma))\in KK^1(\Sigma_{\wDNC_+}(G,\Gamma),C^*(\rG\times \R_+^*))$. We thus have:

\begin{fact}\label{fact7}
$\wind_{\SBlup}^{G,\Gamma}\otimes \beta'=-\beta_\Sigma \otimes \wind_{\wDNC_+}^{G,\Gamma}\in KK^1(\Sigma_{\SBlup}(G,\Gamma),C^*(G_{\rM}^{\rM}\times \R_+^*))$.
\end{fact}

We also have a commutative diagram  where the vertical maps are inclusions:
\begin{equation}\label{diagram2}
\begin{gathered} 
\xymatrix{0\ar[r] &C^*(G_{\rM}^{\rM}\times \R_+^*)\ar[r]\ar[d]^{j'}&\Psi^*(\wDNC_+(G,\Gamma))\ar[r]\ar[d]^{j_\Psi}& \Sigma_{\wDNC_+}(G,\Gamma)\ar[r]\ar[d]^{j_\Sigma}&0\\
0\ar[r] &C^*(G\times \R_+^*)\ar[r] &\Psi^*(\DNC_+(G,\Gamma))\ar[r] &\Sigma_{\DNC_+}(G,\Gamma)\ar[r]&0
}
\end{gathered} 
\end{equation}

We thus find:
\begin{fact}\label{fact8}
$j_\Sigma^*(\wind_{\DNC_+}^{G,\Gamma})=j'_*(\wind_{\wDNC_+}^{G,\Gamma})\in KK^1(\Sigma_{\wDNC_+}(G,\Gamma),C^*(G\times \R_+^*))$.
\end{fact}

\subsubsection{When $V$ is $\gA G$-small}

If $V$ is small in each $G$ orbit, \ie if the Lebesgue measure (in the manifold $G^x$) of $G_V^x$ is $0$ for every $x$, it follows from  prop. \ref{propstab} below  that the inclusion $i: C^*(G_{\rM}^{\rM})  \hookrightarrow C^*(G)$ is an isomorphism. Also, if $\rM$ meets all the orbits of $G$, the inclusion $i$ is a Morita equivalence. In these cases $\partial_{\DNC_+}^{G,\Gamma }$ determines $\partial_{\SBlup}^{G,\Gamma}$.

\begin{definition}\label{notation6}
We will say that $V$ is $\gA G$-small if for every $x\in V$, the composition $\gA G_x\overset{\natural_x}{\longrightarrow} T_xM\longrightarrow (N_V^M)_x$ is not the zero map
\end{definition}

If $V$ is $\gA G$-small, then the orbits of the groupoid $\cN_\Gamma^G$ are never contained in the $0$ section, \ie they meet the open subset $\rN_V^M$, and in fact the set $V\times \{0\}$ is small in every orbit of the groupoid $\DNC(G,\Gamma)$. It follows that the map $j$ is an isomorphism - as well of course as $j'$ and $j''$ of diagram (\ref{diagram1}). In that case, $\partial_{\DNC_+}^{G,\Gamma }$ and $\partial_{\SBlup}^{G,\Gamma}$ correspond to each other under these isomorphisms. 

\begin{proposition}
 \label{propstab} (\cf \cite{HilsSkstab, DSstab}) Let $\cG\rightrightarrows Y$ be a Lie groupoid and let $X\subset Y$ be a (locally closed) submanifold. Assume that, for every $x\in X$, the composition $\gA \cG_x\overset{\natural_x}{\longrightarrow} T_xY\longrightarrow (N_X^Y)_x$ is not the zero map. Then the inclusion $C^*(G_{Y\setminus X}^{Y\setminus X})\to C^*(G)$ is an isomorphism.
\begin{proof}
For every $x\in V$, we can find a neighborhood $U$ of $x\in M$, a section $X$ of $\gA G$ such that, for every $y\in U$, $\natural_y (X(y))\ne 0$ and, if $y\in U\cap V$, $\natural_y (X(y))\not\in T_y(V)$. Denote by $\cF$ the foliation of $U$ associated with the vector field $X$.  It follows from \cite{HilsSkstab} that $C_0(U\setminus V)C^*(U,\cF)=C^*(U,\cF)$; as $C^*(U,\cF)$ acts in a non degenerate way on the Hilbert-$C^*(G)$ module $C^*(G^U)$, we deduce that $C_0(U\setminus V)C^*(G^U)=C^*(G_U)$. We conclude using a partition of the identity argument that $C_c(M\setminus V)C^*(G)=C_c(M)C^*(G)$, whence $C_0(M\setminus V)C^*(G)=C_0(M)C^*(G)=C^*(G)$.
\end{proof}
\end{proposition}

\begin{proposition}\label{fullsymbfullsymb}
We assume that $\Gamma $ is amenable and that $V$ is $\gA G$-small. \\Then, the inclusions $j_\Sigma:\Sigma_{\wDNC_+}(G,\Gamma)\to \Sigma_{\DNC_+}(G,\Gamma)$, $j_\Psi: \Psi^*(\wDNC_+(G,\Gamma))\to \Psi^*(\DNC_+(G,\Gamma)) $ and $j_{symb}:  C_0(\bS^*\gA(\wDNC_+(G,\Gamma)))\to C_0(\bS^*\gA({\DNC_+}(G,\Gamma)))$ are $KK$-equivalences. 
\begin{proof}
We have  a diagram
$$\xymatrix{&&&0\ar[d]\\
0\ar[r] &C^*(\wDNC_+(G,\Gamma))\ar[r]\ar[d]^{j}&\Psi^*(\wDNC_+(G,\Gamma))\ar[r]\ar[d]^{j_\Psi}& C_0(\bS^*\gA(\wDNC_+(G,\Gamma)) \ar[r]\ar[d]^{j_{symb}}&0\\
0\ar[r] &C^*({\DNC_+}(G,\Gamma))\ar[r] &\Psi^*(\DNC_+(G,\Gamma))\ar[r] &C_0(\bS^*\gA({\DNC_+}(G,\Gamma))\ar[r]\ar[d]&0\\
&&&C_0(\bS^*\gA G_{|V}\times \R_+)\ar[d]\\
&&&0}$$
As $j$ is an equality, we find an exact sequence $$0\longrightarrow\Psi^*(\wDNC_+(G,\Gamma))\overset{j_\Psi}{\longrightarrow}\Psi^*({\DNC_+}(G,\Gamma))\overset{}{\longrightarrow} C_0(\bS^*\gA G_{|V}\times \R_+)\longrightarrow0.$$
As $j': C^*(\rG\times \R_+^*)\to C^*(G\times \R_+^*)$ is also an equality, we find (using diagram (\ref{diagram2})) an exact sequence $$0\longrightarrow\Sigma_{\wDNC_+}(G,\Gamma))\overset{j_\Sigma}{\longrightarrow}\Sigma_{\DNC_+}(G,V)\overset{}{\longrightarrow} C_0(\bS^*\gA G_{|V}\times \R_+)\longrightarrow0.$$
 As the algebra $C_0(\bS^*\gA G_{|V}\times \R_+)$ is contractible, we deduce that $j_{symb}$ and then $j_\Psi$ and $j_\Sigma$ are $KK$-equivalences.
\end{proof}
\end{proposition}

\bigskip
As a summary of these considerations, we find:

\begin{theorem}\label{DNC=SBlup}
Let $G\rightrightarrows M$ be a Lie groupoid and $\Gamma\rightrightarrows V$ a Lie subgroupoid of $G$. 
Assume that $\Gamma $ is amenable and put $\rM=M\setminus V$.
Let $i:C^*(G_\rM^\rM)\to C^*(G)$ be the inclusion. Put $\hat\beta''=j''_*(\beta'')\in KK^1(C^*(\cS N_\Gamma^G),C^*(\cN_\Gamma^G))$ and $\hat\beta_\Sigma=(j_\Sigma)_*(\beta_\Sigma)\in KK^1(\Sigma_{\SBlup}(G,\Gamma),\Sigma_{\DNC_+}(G,V))$.
\begin{enumerate}
\item We have equalities \begin{itemize}
\item $\partial_{\SBlup}^{G,\Gamma} \otimes[i]=\hat\beta''\otimes\partial _{\DNC_+}^{G,\Gamma}\in KK^1(C^*(\cS N_\Gamma^G),C^*(G))$ and
\item  $\wind_{\SBlup}^{G,\Gamma} \otimes [i]=\hat\beta_\Sigma\otimes\wind _{\DNC_+}^{G,\Gamma}\in KK^1(C^*(\Sigma_{\SBlup}(G,\Gamma),C^*(G))$

\end{itemize}
\item If $V$ is $\gA G$-small, then $i$ is an isomorphism and the elements $\hat\beta''$ and $\hat\beta_\Sigma$ are invertible. \end{enumerate}
\end{theorem}
\hfill$\square$

\subsection{The KK-element associated with DNC}

The connecting element $\partial_{\DNC_+}^{G,\Gamma }$ can be expressed in the following way: let  $\cG$ be the restriction of $\DNC(G,\Gamma)$ to $[0,1]$, \ie $\cG=\DNC_{[0,1]}(G,\Gamma)=\cN_\Gamma^G\times \{0\}\cup G\times (0,1]$. We have a semi-split exact sequence:  $$0\to C^*(G\times (0,1])\to C^*(\cG)\overset{\ev_0}\longrightarrow C^*(\cN_\Gamma^G)\to 0 \ .$$ As $C^*(G\times (0,1])$ is contractible, $\ev_0$ is a $KK$-equivalence. Let  $ev_1:C^*(\cG)\to C^*(G)$ be evaluation at $1$ and let $\delta_\Gamma^G=[ev_0]^{-1}\otimes [ev_1]\in KK(C^*(\cN_\Gamma^G),C^*(G))$.   Let $[Bott]\in KK^1(\C, C_0(\R_+^*))$ be the Bott element. We find  

\begin{fact}\label{fact3}
$\partial_{\DNC_+}^{G,\Gamma}=\delta_\Gamma^G \underset{\C}{\otimes} [Bott]$.
\end{fact}

\medskip Consider now the groupoid $\cG_{ad}^{[0,1]}$. It is a family of groupoids indexed by $[0,1]\times [0,1]$:\begin{itemize}
\item  its restriction to $\{s\}\times [0,1]$ for $s\ne 0$ is $G_{ad}^{[0,1]}$;
\item  its restriction to $\{0\}\times [0,1]$ is $(\cN_{\Gamma}^G)_{ad}^{[0,1]}$;
\item  its restriction to $ [0,1]\times \{s\}$ for $s\ne 0$ is $\cG=\DNC_{[0,1]}(G,\Gamma)$;
\item  its restriction to $ [0,1]\times \{0\}$ is the algebroid $\gA \cG=\DNC_{[0,1]}(\gA G,\gA \Gamma)$ of $\cG$.
\end{itemize}
For every locally closed subset $X\subset [0,1]\times [0,1]$, denote by $\cG_{ad}^X$ the restriction of $\cG_{ad}^{[0,1]}$ to $X$.

For every closed subset $X\subset [0,1]\times [0,1]$, denote by $q_X:C^*(\cG_{ad}^{[0,1]})\to C^*(\cG^{X}_{ad})$ the restriction map.

 We thus have the following commutative diagram:

$$\xymatrix{ C^*(\cN_{\Gamma}^G) \ar@/^1  pc/@{.>}[rr]^{\delta_\Gamma^G} & \DNC_{[0,1]}(G,\Gamma) \ar[r] \ar [l]& C^*(G)\\
 C^*((\cN_{\Gamma}^G)_{ad}^{[0,1]}) \ar[d]\ar[u] &C^*(\cG_{ad}^{[0,1]})\ar[ld]^{q_{(0,0)}} \ar[lu]_{q_{(0,1)}} \ar[ru]^{q_{(1,1)}} \ar[rd]^{q_{(1,0)}}\ar[r]^{q_{\{1\}\times [0,1]}}\ar[d]^{q_{[0,1]\times \{0\}}} \ar[l]_{q_{\{0\}\times [0,1]}} \ar[u]_{q_{[0,1]\times \{1\}}}& C^*(G_{ad}^{[0,1]}) \ar[u] \ar[d] \\
C^*(N_{\gA \Gamma}^{\gA G}) \ar@/^3pc/@{.>}[uu]^{\ind_{\cN_{\Gamma}^G}} \ar@/_1pc/@{.>}[rr]_{\delta_{\gA \Gamma}^{\gA G}} &C^*(\DNC_{[0,1]}(\gA G,\gA \Gamma))\ar[l] \ar[r]&C_0(\gA^* G) \ar@/_3 pc/@{.>}[uu]_{\ind_{G}}\\
}$$

For every locally closed subset $T\subset [0,1]$, the $C^*$-algebras $C^*(\cG_{ad}^{(0,1]\times T})$ and $C^*(\cG_{ad}^{T\times (0,1]})$ are null homotopic as well as $C^*(\cG_{ad}^{[0,1]^2\setminus \{0,0)\}})$. It follows that $q_{\{0\}\times [0,1]}$, $q_{[0,1]\times \{0\}}$ and $q_{\{(0,0)\}}$ are $KK$-equivalences.

Now $[q_{(0,0)}]^{-1}\otimes [q_{(0,1)}]=\ind_{\cN_{\Gamma}^G}$ and it follows that $[q_{(0,0)}]^{-1}\otimes [q_{(1,1)}]=\ind_{\cN_{\Gamma}^G}\otimes \delta_\Gamma^G$.

In the same way, $[q_{(0,0)}]^{-1}\otimes [q_{(1,0)}]=\delta_{\gA \Gamma}^{\gA G}$ and it follows that $[q_{(0,0)}]^{-1}\otimes [q_{(1,1)}]= \delta_{\gA \Gamma}^{\gA G}\otimes\ind_{G}$.

Finally, it follows from example \ref{uge} that $\delta_{\gA \Gamma}^{\gA G}$ is associated with a morphism $\varphi:C_0(\gA^*(\cN_V^G))\hookrightarrow C_0(\gA^* G)$ corresponding to an inclusion of $\gA^*(\cN_\Gamma^G)$ in $\gA^* G$ as a tubular neighborhood.

We thus have established:

\begin{fact}\label{fact4}
$\ind_{\cN_{\Gamma}^G}\otimes \delta_\Gamma^G=[\varphi]\otimes\ind_{G}$.
\end{fact}


\subsection{The case of a submanifold of the space of units}\label{section:suiteexacte}

Let $G$ be a Lie groupoid with objects $M$ and let $\Gamma=V\subset M$ be a closed submanifold of $M$. In this section, we push further the computations the connecting maps and indices \ie the connecting maps of the exact sequences $E^\partial_{\SBlup},\ E^\partial_{\DNC_+},\ E^{\wind}_{\SBlup}$ and $E^{\wind }_{\DNC_+}$.

\subsubsection{Connecting map and index map}

From propositions \ref{ConnectMapAbstract}, \ref{Sigma=cone}, \ref{avita} and fact \ref{fact4}, we find

\begin{proposition}\label{ComputeDNC}
\begin{enumerate}
\item The index element $\ind_{\cN_V^G}\in KK(C_0(\gA^* N_V^G),C^*(\cN_V^G))$ is invertible.
\item The inclusion $j:\Sigma_{N_V^M\times\{0\}}(\DNC_+(G,V))\hookrightarrow \Sigma_{\DNC_+}(G,V)$ is invertible in $KK$-theory.
\item The $C^*$-algebra $\Sigma_{\DNC_+}(G,V)$ is naturally $KK^1$-equivalent with the mapping cone of the map $\chi:C_0(\gA^*G\times \R_+^*)\to C_0(\DNC_+(M,V))$ defined by $\chi(f)(x)= \begin{cases}
f(x,0)& \hbox{if\ } x\in M\times \R_+^*\\0& \hbox{if\ } x\in N_V^M.
\end{cases}$
\item \label{noiseuse}The connecting element $\partial_{\DNC_+}^{G,V}\in KK^1(C^*(\cN_V^G),C^*(G\times \R_+^*))=KK(C^*(\cN_V^G),C^*(G))$ is  $\delta_V^G=\ind_{\cN_{V}^G}^{-1}\otimes [\varphi]\otimes\ind_{G}$ where $\varphi:C_0(\gA^* N_V^G)\to C_0(\gA^* G)$ is the inclusion using the tubular neighborhood construction.
\item \label{ComputeDNCIndex}Under the $KK^1$ equivalence of c), the full index element $$\wind _{\DNC_+}^{G,V}\in KK^1(\Sigma_{\DNC_+}(G,V),C^*(G\times \R_+^*))=KK^1(\Cn_\chi,C^*(G))$$ is   $q^*([\Bott]\underset{\C}{\otimes} \ind_G)$ where $q:\Cn_\chi\to C_0(\gA^*G\times \R_+^*)$ is evaluation at $0$.\hfill $\square$
\end{enumerate}

\end{proposition}

The element $[\chi]\in KK(C_0(\gA^*G\times \R_+^*),C_0(\DNC_+(M,V)))$ is the Kasparov product of the ``Euler element" of the bundle $\gA^*G$ which is the class in $KK(C_0(\gA^*G),C_0(M))=KK(C_0(\gA^*G\times \R_+^*),C_0(M\times \R_+^*))$ of the map $x\mapsto (x,0)$ with the inclusion $C_0(M\times \R_+^*)\to C_0(\DNC_+(M,V)) $. It follows that $[\chi]$ is often the zero element of $KK(C_0(\gA^*G\times \R_+^*),C_0(\DNC_+(M,V)))$. In particular, this is the case when the Euler class of the bundle $\gA^*G$ vanishes. In that case, the algebra $\Sigma_{\DNC_+}(G,V)$ is $KK$-equivalent to $C_0(\gA^*G)\oplus C_0(\DNC_+(M,V))$. 

If $V$ is $\gA G$ small, then, by theorem \ref{DNC=SBlup},  $\partial_{\SBlup}^{G,V}$ and $\wind _{\SBlup}^{G,V}$ are immediately deduced from proposition \ref{ComputeDNC}.

\begin{remark} 
Let $M_b$ be a manifold with boundary and  $V=\partial M_b$. Put $\rM=M_b\setminus V$. Let $G$ be a  piece of  Lie groupoid on $M_b$ in the sense of section \ref{BoundLiePoid}. Thus $G$ is the restriction of a Lie groupoid $\widetilde G \rightrightarrows M$, where $M$ is a neighborhood of $M_b$. Recall that in this situation, $\SBlup(M,V)=M_b \sqcup M_-$, where $M=M_b\cup M_-$ and $M\cap M_-=V$, and we let $\SBlup_{r,s}(G,V) \rightrightarrows M_b$ be the restriction of $\SBlup_{r,s}(\widetilde G,V)$ to $M_b$.

\smallskip Let us denote by $\rcN_{V}^{G}$ the open subset of $N_V^{\widetilde G}$ made of (normal) tangent vectors whose image under the differential of the source and range maps of $\widetilde G$ are non vanishing elements of $N_V^M$ pointing in the direction of $M_b$.  The groupoid $\SBlup_{r,s}(G,V)$  is the union $\rcN_{V}^{G}/\R_+^* \cup G_\rM^\rM$.

\smallskip
We have exact sequences $$0\to C^*(G_\rM^\rM)\to C^*(\SBlup_{r,s}(G,V))\to C^*(\rcN_{V}^{G}/\R_+^*)\to 0$$
$$0\to C^*(G_\rM^\rM)\to \Psi^*(\SBlup_{r,s}(G,V))\to \Sigma_{\SBlup}(G,V)\to 0.$$

As $V$ is of codimension 1, we find that $V$ is $\gA \widetilde G$-small if and only if it is transverse to $\widetilde G$. In that case, Proposition \ref{ComputeDNC} computes the $KK$-theory of $C^*(\rcN_{V}^{G}/\R_+^*)$ and of $\Sigma_{\SBlup}(G,V)$ and the $KK$-class of the connecting maps of these exact sequences.

In particular, we obtain a six term exact sequence $$\xymatrix{K_0(C(M_b))\ar[r] &K_0(\Sigma_{\SBlup}(G,V))\ar[r]&  K_1(C_0(\gA^*G_\rM^\rM))\ar[d]^{\chi}\\
 K_0(C_0(\gA^*G_\rM^\rM))\ar[u]^{\chi}&K_1(\Sigma_{\SBlup}(G,V))\ar[l]&K_0(C(M_b))\ar[l]}$$
and the index map $K_*(\Sigma_{\SBlup}(G,V))\to K_{*+1}(G_\rM^\rM)$ is the composition of $K_*(\Sigma_{\SBlup}(G,V))\to K_{*+1}(C_0(\gA^*G_\rM^\rM))$ with the index map of the groupoid $G_\rM^\rM$.

This holds, in particular, if $G=M_b\times M_b$ since the boundary $V=\partial M_b$ is transverse to $\widetilde G= M\times  M$. Note that in that case, $\chi=0$ (in $KK(C_0(T^*\rM),C_0(M_b))$) so that we obtain a (noncanonically) split short exact sequence: $$\xymatrix{0\ar[r]& K_*(C_0(M_b))\ar[r] &K_*(\Sigma_{\SBlup}(G,V))\ar[r]&  K_{*+1}(C_0(\gA^*G_\rM^\rM))\ar[r]&0.}$$
\end{remark}

\subsubsection{The index map via relative $K$-theory}

It follows now from prop. \ref{indrel}:

\begin{proposition} 
Let $\psi_{\DNC}:C_0(\DNC_+(M,V))\to \Psi^*(\DNC_+(G,V))$ be the inclusion map which associates to a (smooth) function $f$ the order $0$ (pseudo)differential operator multiplication by $f$ and $\sigma_{full}:\Psi^*(\DNC_+(G,V))\to \Sigma_{\DNC_+}(G,V)$ the full symbol map. Put $\mu_{\DNC}=\sigma_{full}\circ {\psi_{\DNC}}$. Then the relative $K$-group $K_*(\mu_{\DNC})$  is naturally isomorphic to $K_{*+1}(C_0(\gA^* G))$. Under this isomorphism, $\ind_{rel}:K_*(\mu_{\DNC})\to K_*(C^*(G\times \R_+^*))=K_{*+1}(C^*(G))$ identifies with  $\ind_{G}$. \hfill$\square$
\end{proposition}

Let us say also just a few words on the relative index map for $\SBlup_{r,s}(G,V)$, \ie for the map $\mu_{\SBlup}:C_0(\SBlup_+(M,V))\to \Sigma_{\SBlup}(G,V)$ which is the composition of the inclusion $\psi_{\SBlup}:C_0(\SBlup (M,V)\to \Psi^*(\SBlup_{r,s}(G,V))$ with the full index map $\sigma_{full}: \Psi^*(\SBlup_{r,s}(G,V))\to \Sigma_{\SBlup}(G,V))$, and the corresponding relative index map $\ind_{rel}:K_*(\mu_{\SBlup})\to K_{*}(C^*(\rG))$. Equivalently we wish to compute the relative  index map $\ind_{rel}:K_*(\mu_{\wDNC})\to K_{*+1}(C^*(\rG))$, where ${\mu_{\wDNC}}:C_0(\wDNC_+(M,V))\to \Sigma_{\wDNC_+}(G,V)$. We restrict to the case when $V$ is $\gA G$ small.
 
We have a diagram $$
\xymatrix{0\ar[r]&C_0(\wDNC_+(M,V))\ar[r] &C_0(\DNC_+(M,V))\ar[r] &C_0(V\times \R_+)\ar[r]&0}
$$
and it follows that the inclusion $C_0(\wDNC_+(M,V))\to C_0(\DNC_+(M,V))$ is a $KK$-equivalence. 

Since the inclusions $\Psi^*(\wDNC_+(G,V))\to \Psi^*(\DNC_+(G,V))$ and $\Sigma_{\wDNC_+}(G,V)\to \Sigma_{\DNC_+}(G,V)$ are also $KK$-equivalences (prop. \ref{fullsymbfullsymb}), it follows that the inclusion $\Cn_{\mu_{\wDNC}}\to \Cn_{\mu_{\DNC}}$ is a $KK$-equivalence - and therefore the relative $K$-groups $K_*(\mu_{\wDNC})$ and $K_*(\mu_{\DNC})$ are naturally isomorphic. Using this, together with the Connes-Thom isomorphism, we deduce:

\begin{corollary}\label{CorRelBlup}
We assume that $V$ is $\gA G$ small\begin{enumerate}
\item  The relative $K$-group $K_*(\mu_{\wDNC})$  is naturally isomorphic to $K_{*+1}(C_0(\gA^* G))$. Under this isomorphism, $\ind_{rel}:K_*(\mu_{\wDNC})\to K_*(C^*(G\times \R_+^*))=K_{*+1}(C^*(G))$ identifies with  $\ind_{G}$.

\item  The relative $K$-group $K_*(\mu_{\SBlup})$  is naturally isomorphic to $K_*(C_0(\gA^* G))$. Under this isomorphism, $\ind_{rel}:K_*(\mu_{\SBlup})\to K_*(C^*(G))$ identifies with  $\ind_{G}$. \hfill$\square$
\end{enumerate}

\end{corollary}

\section{A Boutet de Monvel type calculus}

From now on, we suppose that \(V\) is a \emph{transverse submanifold} of $M$ with respect to the Lie groupoid \(G\rightrightarrows M\). In particular $V$ is $\gA G$-small  - of course, we assume that (in every connected component of $V$), the dimension of $V$ is strictly smaller than the dimension of $M$.

\subsection{The \PT bimodule}

As $V$ is transverse to $G$, the groupoid $G_V^V$ is a Lie groupoid, so that we can construct its ``gauge adiabatic groupoid'' $(G_V^V)_{ga}$ (see section \ref{subsecgag}).

In \cite{DS1}, we constructed a bi-module relating the $C^*$-algebra of the groupoid $(G_V^V)_{ga}$ and the $C^*$-algebra of pseudodifferential operators of $G_V^V$.

\medskip
In this section,
\begin{itemize}
\item 
We first show that the groupoid $(G_V^V)_{ga}$, is (sub-) Morita equivalent to $\SBlup_{r,s}(G,V)$ (\cf also section \ref{subsecsubmorita} for a local construction). 

\item Composing the resulting bimodules, we obtain the ``Poisson-trace'' bimodule relating $C^*(\SBlup_{r,s}(G,V))$ and $\Psi^*(G_V^V)$.

\end{itemize}

\subsubsection{The $\SBlup_{r,s}(G,V)-(G_V^V)_{ga}$-bimodule $\scE(G,V)$}
Define the map $j:M\sqcup (V\times \R)\to M$ by letting $j_0:M\to M$ be the identity and $j_1:V\times \R\to M$ the composition of the projection $V\times \R\to V$ with the inclusion. Let $\cG=G_j^j$. As $V$ is assumed to be transverse, the map $j$ is also transverse, and therefore $\cG$ is a Lie groupoid.

It is the union of four clopen subsets\begin{itemize}
\item the groupoids $G_{j_0}^{j_0}=G=\cG_M^M$ and $G_{j_1}^{j_1}=G_V^V\times (\R\times\R)=\cG_{V\times \R}^{V\times \R}$.
\item the \emph{linking spaces} $G_{j_1}^{j_0}=\cG_{V\times \R}^M=G_{V}\times \R $ and $G_{j_0}^{j_1}=\cG^{V\times \R}_M=G^{V}\times \R $.
\end{itemize}

By functoriality, we obtain a sub-Morita equivalence of $\SBlup_{r,s}(G_V^V\times \R\times \R,V)$ and $\SBlup_{r,s}(G,V)$ (see section \ref{Moreq2}).\\ Let us describe this sub-Morita equivalence in a slightly different way:

Let also $\Gamma =V\times \{0,1\}^2$, sitting in $\cG$: 
\[\begin{array}{cc} V\times \{(0,0)\}\subset G=G_{j_0}^{j_0} \ ;  &\ V\times \{(0,1)\}\subset G_V\times \{ 0\} \subset G_{j_1}^{j_0}\ ; 
\\  V\times \{(1,0)\}\subset G^V\times \{ 0\} \subset G^{j_1}_{j_0}\ ;\  & \ V\times \{(1,1)\}\subset G_V^V\times \{(0,0)\} \subset G_{j_1}^{j_1} \ . \end{array}\]

It is a subgroupoid of $\cG$. The blowup construction applied to $\Gamma\subset \cG$ gives then a groupoid $\SBlup_{r,s}(\cG,\Gamma)$ which is the union of:
\[\begin{array}{cc} \SBlup_{r,s}(G,V)  \ ;  &\ \SBlup_{r,s}(G_V\times \R,V)\ ; 
\\  \SBlup_{r,s}(G^V\times \R,V)\ ;\  & \ \SBlup_{r,s}(G_V^V\times \R\times \R,V) \ . \end{array}\]

 Recall that $\SBlup (V\times \R,V\times\{0\})\simeq V\times (\R_-\sqcup \R_+)$. Thus $\SBlup_{r,s}(\cG,\Gamma)$ is a groupoid with objects $\SBlup(M,V)\sqcup V\times \R_-\sqcup V\times \R_+$.\\

The restriction of $\SBlup_{r,s}(\cG,\Gamma)$ to $V\times \R_+$ coincides with the restriction of $\SBlup_{r,s}(G_V^V\times \R\times \R,V)$ to $V\times \R_+$: it is the gauge adiabatic $(G_V^V)_{ga}$ groupoid of $G_V^V$ (\cf section \ref{subsecgag}).

\medskip Put $\SBlup_{r,s}(G_V\times \R,V)_+=\SBlup_{r,s}(\cG,\Gamma)_{V\times \R_+}^{\SBlup(M,V)}$. It is a linking space between the groupoids $\SBlup_{r,s}(G,V)$ and $(G_V^V)_{ga}$. Put also  $\SBlup_{r,s}(G^V\times \R,V)_+=\SBlup_{r,s}(\cG,\Gamma)^{V\times \R_+}_{\SBlup(M,V)}$.

\medskip With the notation used in fact \ref{otum}, we define the $C^*(\SBlup_{r,s}(G,V))-C^*((G_V^V)_{ga})$-bimodule $\scE(G,V)$ to be $C^*(\SBlup_{r,s}(G_V\times \R,V)_+)$.  It is the closure of $C_c(\SBlup_{r,s}(G_V\times \R,V)_+)$ in $C^*(\SBlup_{r,s}(\cG,\Gamma))$. It is a full Hilbert-$C^*(\SBlup_{r,s}(G,V))-C^*((G_V^V)_{ga})$-module.

The Hilbert-$C^*((G_V^V)_{ga})$-module $\scE(G,V)$ is full and $\cK(\scE(G,V))$ is the ideal $C^*(\SBlup_{r,s}(G_\Omega^\Omega,V))$ where $\Omega=r(G_V)$  is the union of orbits which meet $V$.

\smallskip Notice that $\Omega=M\setminus V\sqcup V\times \R^*$ and $F=\bS N_V^M \sqcup V\sqcup V$ gives a partition by respectively open and closed satured subsets of the units of  $\SBlup_{r,s}(\cG,\Gamma)$. Furthermore $\SBlup_{r,s}(\cG,\Gamma)_\Omega^\Omega=\cG_\Omega^\Omega$ and $C^*(\cG_\Omega^\Omega)=C^*(\cG)$ according to proposition \ref{propstab}. This decomposition gives rise to an exact sequence of C$^*$-algebras.

$$\xymatrix{0\ar[r] &C^*(\cG)\ar[r] &C^*(\SBlup_{r,s}(\cG,\Gamma))\ar[r] & C^*(\cS N_{\Gamma}^{\cG})\ar[r] &0  }$$

This exact sequence gives rise to an exact sequence of bimodules:

$$\xymatrix{0\ar[r] &C^*(G)\ar[r]\ar@{-}[d]^{\ronde \scE(G,V)}&C^*(\SBlup_{r,s}(G,V))\ar[r]\ar@{-}[d]^{\scE(G,V)}& C^*(\cS N_{V}^{G})\ar[r]\ar@{-}[d]^{\scE^{\partial}(G,\Gamma)}&0 \\
0\ar[r] &C^*(G_{V\times \R_+^*}^{V\times \R_+^*})\ar[r]&C^*((G_V^V)_{ga})\ar[r]& C^*(\gA G_V^V\rtimes \R_+^*)\ar[r]&0  }$$

where $\ronde \scE(G,V)=C^*(\cG^{M\setminus V}_{V\times \R_+^*})$ and $\scE^{\partial}(G,\Gamma)=C^*\big((\cS N_\Gamma^\cG)^{\bS N_V^M}_V\big)=\scE(G,V)/\ronde \scE(G,V)$.

\subsubsection{The \PT bimodule $\Ept$} 

In \cite{DS1}, we constructed, for every Lie groupoid $H$ a $C^*(H_{ga})-\Psi^*(H)$-bimodule $\scE_H$. \\ Recall that the Hilbert $\Psi^*(H)$-module $\scE_H$ is full and that $\cK(\scE _H)\subset C^*(H_{ga})$ is the kernel of a natural $*$-homomorphism $C^*(H_{ga})\to C_0(H^{(0)}\times \R)$. We also showed that the bimodule $\scE _H$ gives rise to an exact sequence of bimodule as above:  

$$\xymatrix{0\ar[r] &C^*(H\times \R_+^* \times \R_+^*)\ar[r]\ar@{-}[d]^{\ronde \scE _H}&C^*(H_{ga})\ar[r]\ar@{-}[d]^{\scE _H}& C^*(\gA H \rtimes \R_+^*)\ar[r]\ar@{-}[d]^{\scE ^{\partial}_H}&0 \\
0\ar[r] & C^*(H) \ar[r]& \Psi^*(H) \ar[r]& C_0(\bS^* \gA H) \ar[r]&0  }$$

Putting together the bimodule $\scE (G,V)$ and $\scE _{G_V^V}$ we obtain a $C^*(\SBlup_{r,s}(G,V))-\Psi^*(G_V^V)$ bimodule $\scE (G,V)\otimes_{C^*((G_V^V)_{ga})}\scE _{G_V^V}$ that we call the \emph{\PT} bimodule and denote by $\Ept(G,V)$ - or just $\Ept$. It leads to the exact sequence of bimodule:

$$\xymatrix{0\ar[r] &C^*(G)\ar[r]\ar@{-}[d]^{\ronde \Ept(G,V)}&C^*(\SBlup_{r,s}(G,V))\ar[r]\ar@{-}[d]^{\Ept(G,V)}& C^*(\cS N_{V}^{G})\ar[r]\ar@{-}[d]^{\Ept^{\partial}(G,V)}&0 \\
0\ar[r] & C^*(G_V^V) \ar[r]& \Psi^*(G_V^V) \ar[r]& C_0(\bS^* \gA G_V^V) \ar[r]&0    }$$

The \PT bimodule  is a full Hilbert $\Psi^*(G_V^V)$-module and $\cK(\Ept(G,V))$ is a two sided ideal of $C^*(\SBlup_{r,s}(G,V))$. Denote by $\Ept (G,V)^*$  its dual module, \ie the $\Psi^*(G_V^V)-C^*(\SBlup_{r,s}(G,V))$-bimodule $\cK(\Ept(G,V),\Psi^*(G_V^V))$.

\subsection{A  \BMT\ algebra}

The $C^*$-algebra $C^*_{BM}(G,V)=\cK\Big(C^*(\SBlup_{r,s}(G,V)) \oplus \Ept(G,V)^* \Big)$ is an algebra made of matrices of the form $
\begin{pmatrix} K&P\\
 T&Q
\end{pmatrix}
$ where $K\in C^*(\SBlup_{r,s}(G,V)),\ P\in \Ept(G,V),\ T\in \Ept(G,V)^*,\ Q\in \Psi^*(G_V^V)$.

We have an exact sequence (where $\rM\sqcup V\ne M$ denotes the topological disjoint union of $\mathring{M}$ with $V$):  $$0\to C^*(G_{\rM\sqcup V}^{\rM\sqcup V})\to C^*_{BM}(G,V)\overset{r_V^{C^*}}\longrightarrow \boundarysymb \to 0,$$
where the quotient $\boundarysymb$ is the algebra of the \emph{Boutet de Monvel type boundary symbols.} It is the algebra of matrices of the form $
\begin{pmatrix} k&p\\
 t&q
\end{pmatrix}
$ where $k\in C^*(\cS N_V^G)$, $q\in C(\bS^*\gA G_V^V)$, $p,t^*\in \Eptb(G,V):= \Ept(G,V)\otimes_{\Psi^*(G_V^V)}C(\bS^*\gA G_V^V)$. The map $r_V^{C^*}$ is of the form $$r_V^{C^*}\begin{pmatrix} K&P\\
 T&Q
\end{pmatrix}=\begin{pmatrix} r_V^{\Green}(K)&r_V^{\fish}(P)\\
 r_V^{\trace}(T)&\sigma_V(Q)
\end{pmatrix}$$
where:
\begin{itemize}
\item the quotient map $\sigma_V$ is the ordinary order $0$ principal symbol map on the groupoid $G_V^V$;

\item the quotient maps $r_V^{\Green},r_V^{\fish},r_V^{\trace}$ are restrictions to the boundary $N_V^M$: $$r_V^{\Green}:C^*(\SBlup_{r,s}(G,V))\to C^*(\bS N_V^G)=C^*(\SBlup_{r,s}(G,V))/C^*(G^{\rM}_{\rM}),$$ $$r_V^{\fish}:\Ept(G,V)\to \Eptb(G,V)=\Ept(G,V)/C^*(G^{\rM}_V),$$ and $r_V^{\trace}(T)=r_V^{\fish}(T^*)^*$.
\end{itemize}

The map $r_V^{C^*}$ is called the \emph{zero order symbol map of the Boutet de Monvel type calculus.}

\subsection{A  \BMT\  pseudodifferential algebra}

We denote by $\Psi^*_{BM}(G,V)$ the algebra of matrices $\begin{pmatrix}
 \Phi&P\\
 T&Q
\end{pmatrix}$ with $\Phi\in \Psi^*(\SBlup_{r,s}(G,V)),\ P\in \Ept({G,V}),\ T\in \Ept({G,V})^*$ and $Q\in \Psi^*(G_V^V)$.

Such an operator $R=\begin{pmatrix}
 \Phi&P\\
 T&Q
\end{pmatrix}$ has two symbols:\begin{itemize}
\item  the \emph{classical symbol} $\sigma_c:\Psi^*_{BM}(G,V)\to C_0(\bS^*\gA\SBlup_{r,s}(G,V))$ given by $\sigma_c\begin{pmatrix}
 \Phi&P\\
 T&Q
\end{pmatrix}=\sigma_c(\Phi)$;
\item  the \emph{boundary symbol} $r_V^{BM}:\Psi^*_{BM}(G,V)\to \Boundarysymb$ defined by $$r_V\begin{pmatrix} \Phi&P\\
 T&Q
\end{pmatrix}=\begin{pmatrix} r_V^{\pdo}(\Phi)&r_V^{\fish}(P)\\
 r_V^{\trace}(T)&\sigma_V(Q)
\end{pmatrix}$$
where $r_V^{\pdo}:\Psi^*(\SBlup_{r,s}(G,V))\to \Psi^*(\cS N_V^G)$ is the restriction.
\end{itemize}

Here $\Boundarysymb$ denotes the algebra of matrices of the form $\begin{pmatrix}
\phi&p\\
t&q
\end{pmatrix}$ with $\phi\in \Psi^*(\cS N_V^G)$, $p,t^*\in \Ept^V(G,V)$ and $q\in C(\bS^*\gA G_V^V)$.

\medskip The \emph{full symbol} map is the morphism  $$\sigma_{BM}:\Psi^*_{BM}(G,V)\to  \Sigma_{BM}(G,V):=C_0(\bS^*\gA\SBlup_{r,s}(G,V))\times _{C_0(\bS^*\gA \cS N_V^G)}\Boundarysymb$$ defined by $\sigma_{BM}(R)=(\sigma_c(R),r_V(R))$.

\medskip We have an exact sequence:
$$0\to C^*(G_{\rM\sqcup V}^{\rM\sqcup V})\to \Psi^*_{BM}(G,V)\overset{\sigma_{BM}}\longrightarrow \Sigma_{BM}(G,V) \to 0.\eqno{\EBM}$$

We may note that $\Psi^*(\SBlup_{r,s}(G,V))$ (\resp $\Psi^*(\cS N_V^G)$) identifies with the full hereditary subalgebra of $\Psi^*_{BM}(G,V)$ (\resp of $\Sigma_{BM}(G,V)$) consisting of elements of the form $ \begin{pmatrix}
 x&0\\
 0&0
\end{pmatrix} $.

\subsection{$K$-theory of the symbol algebras and index maps}

In this section we examine the index map corresponding to the Boutet de Monvel type calculus and in particular to the exact sequence $\EBM$. We compute the $K$-theory of the symbol algebra $\Sigma_{BM}$ and the connecting element $\wind_{BM}\in KK^1(\Sigma_{BM},C^*(G))$ (\footnote{We use the Morita equivalence of $C^*(G)$ with $C^*(G_{\rM\sqcup V}^{\rM\sqcup V})$.}).

We then extend this computation by including bundles into the picture \ie by computing a relative $K$-theory map.

\subsubsection{$K$-theory of $\Sigma_{BM}$ and computation of the index}

As the Hilbert $\Psi^*(G_V^V)$ module $\Ept({G,V})$ is full, \begin{itemize}
\item the subalgebra $\Big\{ 
\begin{pmatrix}
 K&0\\
 0&0
\end{pmatrix};\ K\in C^*(\SBlup_{r,s}(G,V))
 \Big\}$ is a full hereditary subalgebra of $C^*_{BM}(G,V)$;
 
\item the subalgebra $\Big\{ 
\begin{pmatrix}
 \Phi&0\\
 0&0
\end{pmatrix};\ \Phi\in \Psi^*(\SBlup_{r,s}(G,V))
 \Big\}$ is a full hereditary subalgebra of $\Psi^*_{BM}(G,V)$;
 
 \item the subalgebra $\Big\{ 
\begin{pmatrix}
 x&0\\
 0&0
\end{pmatrix};\ x\in \Sigma_{\SBlup}(G,V)
 \Big\}$ is a full hereditary subalgebra of $\Sigma_{BM}(G,V)$;

\item the subalgebra $\Big\{ 
\begin{pmatrix}
 k&0\\
 0&0
\end{pmatrix};\ k\in C^*(\cS N_V^G)
 \Big\}$ is a full hereditary subalgebra of $\boundarysymb$;

 \item the subalgebra $\Big\{ 
\begin{pmatrix}
\phi&0\\
 0&0
\end{pmatrix};\ \phi\in \Psi^*(\cS N_V^G)
 \Big\}$ is a full hereditary subalgebra of $\Boundarysymb$.
 \end{itemize}

We have a diagram of exact sequences where the vertical inclusions are Morita equivalences:
$$\xymatrix{0\ar[r] &C^*(G_{\rM}^{\rM})\ar[r]\ar@{^{(}->}[d]&\Psi^*(\SBlup_{r,s}(G,V))\ar[r]^{ \sigma_{full}}\ar@{^{(}->}[d]& \Sigma_{\SBlup}(G,V)\ar[r]\ar@{^{(}->}[d]&0\\
0\ar[r] &C^*(G_{\rM\sqcup V}^{\rM\sqcup V})\ar[r] &\Psi^*_{BM}(G,V)\ar[r]^{\sigma_{BM}} &\Sigma_{BM}(G,V)\ar[r]&0}$$

We thus deduce immediately from theorem \ref{DNC=SBlup} and prop. \ref{ComputeDNC}:

\begin{corollary}
The algebra $\Sigma_{BM}(G,V))$ is $KK$-equivalent with the mapping cone $\Cn_\chi$ and, under this $K$-equivalence,  the index $\wind_{BM}$ is $q^*([Bott]\otimes_\C \ind_G)$ where $q:\Cn_\chi\to C_0(\gA^*G\times \R_+^*)$ is evaluation at $0$.
\end{corollary}

\subsubsection{Index in relative $K$-theory}

One may also consider more general index problems, which are concerned with generalized boundary value problems in the sense of \cite{Schrohe, MSS1, MSS2}: those are concerned with index of \emph{fully elliptic} operators of the form 
$
R=\begin{pmatrix}
 \Phi&P\\
 T&Q
\end{pmatrix}
$, where we are given hermitian complex vector bundles $E_\pm$ over $\SBlup(M,V)$ and $F_{\pm}$ over $V$, and 
\begin{itemize}
\item $\Phi$ is an order $0$ pseudodifferential operator of the Lie groupoid $\SBlup_{r,s}(G,V)$  from sections of $E_+$ to sections of $E_-$;
\item $P$ is an order $0$ ``Poisson type'' operator from sections of $F_+$ to sections of $E_-$;
\item $T$ is an order $0$ ``trace type'' operator from sections of $E_+$ to sections of $F_-$;
\item  $Q$ is an order $0$ pseudodifferential operator of the Lie groupoid $G_V^V$ from sections of $F_+$ to sections of $F_-$.
\end{itemize}

In other words, writing $E_\pm$ as associated with projections $p_\pm\in M_N(C^\infty(\SBlup(M,V)))$ and $F_\pm$ as associated with projections $q_\pm\in M_N(C^\infty(V))$, then $R\in (p_-\oplus q_-)M_N(\Psi^*_{BM}(G,V))(p_+\oplus q_+)$. 

Full ellipticity for $R$ means just that the \emph{full symbol} of $R$ is invertible, \ie that there is a quasi-inverse $R'\in (p_+\oplus q_+)M_N(\Psi^*_{BM}(G,V))(p_-\oplus q_-)$, such that $(p_+\oplus q_+)-R'R\in M_N(C^*(G_{\rM\sqcup V}^{\rM\sqcup V}))$ and $(p_-\oplus q_-)-RR'\in M_N(C^*(G_{\rM\sqcup V}^{\rM\sqcup V}))$.

\medskip 
In other words, we wish to compute the morphism $\ind_{rel}:K_*(\mu_{BM})\to K_*(C^*_{BM}(G,V))$ where $\mu_{BM}$ is the  natural morphism $\mu_{BM}:C_0(\SBlup(M,V))\oplus C_0(V)\to \Sigma_{BM}(G,V)$.  

\medskip Let us outline here this computation. We start with a remark.

\begin{remark}
Let $H\rightrightarrows  V$ be a Lie groupoid. The bimodule $\scE ^{\partial}_H$ is a Morita equivalence of an ideal $C^*(\gA H \rtimes \R_+^*)$ with $C_0(\bS^* \gA H)$ and therefore defines an element  $\zeta_H\in KK(C_0(\bS^* \gA H),C^*(\gA H \rtimes \R_+^*))$. 
Let $\mu_H:C_0(V)\to C_0(\bS^* \gA H)$ be the inclusion (given by the map $\bS^* \gA H\to V$). The composition $\mu_H^*(\zeta_H)$  is the zero element in $KK(C_0(V),C^*(\gA H \rtimes \R_+^*))$. Indeed $\mu_H^*(\zeta_H)$ can be decomposed as\begin{itemize}
\item  the Morita equivalence $C_0(V)\sim C_0((V\times \R_+^*)\rtimes \R_+^*)$,
\item the inclusion $C_0(V\times \R_+^*)\rtimes \R_+^* \subset C_0(V\times \R_+)\rtimes \R_+^* $,
\item the inclusion $C_0(V\times \R_+)\rtimes \R_+^*\to C_0(\gA^* H) \rtimes \R_+^*$ corresponding to the map $(x,\xi)\mapsto (x,\|\xi\|)$ from $\gA^* H$ to $V\times \R_+$.
\end{itemize}
Now, the Toeplitz algebra $C_0(\R_+)\rtimes \R_+^*$ is $K$-contractible.
\end{remark}

From this remark, we immediately deduce:

\begin{proposition}
The inclusion $C_0(V)\to \Sigma_{BM}(G,V)$ is the zero element in $KK$-theory. \hfill$\square$
\end{proposition}

We have a diagram 
$$\xymatrix{C_0(\SBlup(M,V))\oplus C_0(V)\ar[rr]^{\mu_{\SBlup}\oplus \mu_V}\ar@{=}[d]&&\Sigma_{\SBlup}(G,V)\oplus C_0(\bS^* \gA G_V^V)\ar[d]\\
C_0(\SBlup(M,V))\oplus C_0(V)\ar[rr]^{\hspace{1.4cm}\psi_{BM}}&&\Sigma_{BM}(G,V)
}$$

The mapping cone $\Cn_{\check \mu_{\SBlup}}$ of the morphism $\check \mu_{\SBlup}: C_0(\SBlup(M,V))\oplus 0\to \Psi^*_{BM}(G,V)$ is Morita-equivalent to the mapping cone of the morphism $\mu_{\SBlup}:C_0(\SBlup_+(M,V))\to \Sigma_{\SBlup}(G,V))$ and therefore it is $KK$-equivalent to $C_0(\gA^* G\times \R)$ by Cor. \ref{CorRelBlup}. 

We then deduce:

\begin{theorem}
\begin{enumerate}
\item The relative $K$-theory of  $\mu_{BM}$ is naturally isomorphic to $K_*(\gA^* G)\oplus K_{*+1}(C_0(V))$.
\item Under this equivalence, the relative index map identifies with  $\ind_{G}$ on $K_*(\gA^* G)$ and the zero map on $K_{*+1}(C_0(V)$.  \hfill$\square$
\end{enumerate}
\end{theorem}

\section{Appendix}

\subsection{A characterization of groupoids via elements composable to a unit}

\begin{remark}
Let $G$ be a groupoid. For $n\in \N$, we may define the subsets $U^{n}(G)=\{(x_1,\ldots,x_n)\in G^n; s(x_i)=r(x_{i+1});\ x_1\cdot x_2\ldots x_n\in G^{(0)}$. We have $U^1(G)=G^{(0)}$, the sets $U^k(G)$ are invariant under cyclic permutations; moreover, we have  natural maps $\delta_k:U^{n}(G)\to U^{n+1}(G)$ ($1\le k\le n$) defined by $\delta_k(x_1,\ldots,x_n)=(x_1,\ldots,x_k,s(x_k),x_{k+1},\ldots ,x_n)$ and boundaries $b_k:U^{n+1}(G)\to U^{n}(G)$ defined by $b_k(x_1,\ldots ,x_{n+1})=(x_1,\ldots,x_{k-1},x_kx_{k+1},x_{k+2},\ldots x_{n+1})$ if $k\ne n+1$ and $b_{n+1}(x_1,\ldots ,x_{n+1})=(x_{n+1}x_1,x_2,\ldots ,x_n)$.
\end{remark}

Let $G$ be a set. For $(x,y,z)\in G^3$, put $q_1(x,y,z)=x$ and $q_{1,2}(x,y,z)=(x,y)$.

\begin{proposition}\label{VBgpdprop1}
Let $G$ be a set and $U^{1}(G)\subset G$, $U^3(G) \subset G^3$ be subsets satisfying the following conditions.\begin{enumerate}
\item The subset $U^{3}(G)$ of $G^3$ is invariant under cyclic permutation.\label{condition1}
\item For all $x\in U^{1}(G)$, $(x,x,x)\in U^{3}(G)$.\label{condition2}
\item The map $q_{12}:(x,y,z)\mapsto (x,y)$ from $U^{3}(G)$ to $G^2$ is injective.\label{condition3}
\item  Let $R=\{(x,y,z)\in U^3(G);\ z\in U^{1}(G)\}$. The map $q_1:R\to G$ is a bijection and $U^2(G)=q_{12}(R)$ is invariant under (cyclic) permutation.\label{condition4}
\item  The subset $U^4(G)=\{(x,y,z,t)\in G^4;\ \exists (u,v)\in U^2(G); (x,y,v)\in U^{3}(G)\ \hbox{and}\ (u,z,t)\in U^{3}(G)\}$ is invariant under cyclic permutation in $G^4$.\label{condition5}
\end{enumerate}
 Then there is a unique groupoid structure on $G$ such that $U^{1}(G)$ is its set of units, and $U^{3}(G)=\{(x,y,z)\in G^3;\ (x,y)\in G^{(2)},\ z=(xy)^{-1}\}$.
\end{proposition}

\begin{proof}
Uniqueness is easy:  one defines the range and the inverse of $x$ by saying that $(x,x^{-1},r(x))$ is the unique element $w$ in $R$ such that $q_1(w)=x$; the source is defined by $s(x)=r(x^{-1})$; the product for composable elements $(x,y)$ is then defined by the fact $(x,y,(xy)^{-1})\in U^{3}(G)$.

\medskip Let us pass to existence.

\begin{itemize}
\item By condition (\ref{condition4}), $U^2(G)$ is the graph of an involution that we denote by $x\mapsto x^{-1}$. By condition (\ref{condition2}), if $x\in U^1(G)$, then $x^{-1}=x$.

\item Define also $r:G\to U^{1}(G)$ to be the (unique) element in $U^{1}(G)$ such that $(x,x^{-1},r(x))\in U^{3}(G)$ and put $s(x)=r(x^{-1})$.

\item Put $G^{(2)}=q_{12}(U^{3}(G))$. If $(x,y)\in G^{(2)}$, then there exists $z$ such that $(x,y,z)\in U^{3}(G)$. As $(x,s(x),x^{-1})\in U^{3}(G)$, it follows that $(x,s(x),y,z)\in U^4(G)$ and thus $(z,x,s(x),y)\in U^4(G)$, and therefore $(s(x),y)\in G^{(2)}$, and $r(y)=s(x)$.

Conversely, if $(x,y)\in G^2$ satisfy $s(x)=r(y)$, as $(x^{-1},x,s(x))\in U^{3}(G)$ and $(s(x),y,y^{-1})\in U^{3}(G)$, it follows that $(x^{-1},x,y,y^{-1})\in U^4(G)$, whence $(x,y,y^{-1},x^{-1})\in U^4(G)$ and $(x,y)\in G^{(2)}$.

In other words, $G^{(2)}=\{(x,y)\in G^2;\ s(x)=r(y)\}$.

\item For $(x,y)\in G^{(2)}$, we may define thanks to condition (\ref{condition3}) the element $xy\in G$, by the requirement $(x,y,(xy)^{-1})\in U^{3}(G)$. 

\item Since $(y,(xy)^{-1})\in G^{(2)}$ and $((xy)^{-1},x)\in G^{(2)}$, it follows that $s(xy)=r((xy)^{-1})=s(y)$ and $r(xy)=s((xy)^{-1})=r(x)$.

\item For $x\in G$, since $(r(x),x,x^{-1})\in U^{3}(G)$ and $(x,s(x),x^{-1})\in U^{3}(G)$, it follows that $r(x)x=x$ and $xs(x)=x$ - thus units are units. As $(x,x^{-1},r(x))\in U^{3}(G)$ and $(x^{-1},x,s(x))\in U^{3}(G)$ we find $xx^{-1}=r(x)$ and $x^{-1}x=s(x)$ and thus $x^{-1}$ is the inverse of $x$.

\item Finally, let $(x,y,z)\in G^3$ be such that $(x,y)\in G^{(2)}$ and $(y,z)\in G^{(2)}$. We saw that $s(xy)=s(y)=r(z)$. Put $w=((xy)z)^{-1}$. Then $(x,y,(xy)^{-1})\in U^{3}(G)$ and $(xy,z,w)\in U^{3}(G)$, and thus $(x,y,z,w)\in U^4(G)$, whence $(y,z,w,x)\in U^4(G)$ and therefore $(yz,w,x)\in U^{3}(G)$ and finally, $(x,yz,w)\in U^{3}(G)$, which means that $x(yz)=w^{-1}=(xy)z$.
\qedhere
\end{itemize}
\end{proof}

\subsection{\VBGs\ and their duals (\cite{Pra,Mack})}
A \VBG\ over a groupoid $G$ is a vector bundle $E$ over $G$ with a groupoid structure such that $E^{(0)}\subset E$ is a vector subbundle of the restriction of $E$ to $G^{(0)}$ and such that all the structure maps of the groupoid $E$ ($r_E,s_E$, $x\mapsto x^{-1}$ and the composition) are linear bundle maps and $r_E:E\to r_G^*(E^{(0)})$ is surjective.

\begin{proposition}
Let $E\to G$ be a \VBG. For all $k\in \N$ ($k\ge 1$) $U^k(E)\to U^{k}(G)$ is a subbundle of the restriction to $U^{k}(G)\subset G^k$ of the bundle $E^k\to G^k$. We identify the dual bundle of $E^k\to G^k$ with $(E^*)^k$. Then the dual bundle $E^*$ is a VB-groupoid over $G$ with $U^k(E^*)=U^k(E)^\perp$ for all $k$.
\begin{proof}
We prove that $U^1(E^*)=(E^{(0)})^\perp$ and $U^3(E^*)=U^3(E)^\perp$ satisfy the conditions of prop. \ref{VBgpdprop1}.  

\begin{enumerate}\renewcommand\labelenumi{\rm ({\theenumi})}
\item  Condition (\ref{condition1}) is obvious: since $U^3(E)$ is invariant under cyclic permutations, so is $U^3(E)^\perp$.

\item  Taking the restriction of $E$ over a point of $G^{(0)}$, we have a linear groupoid, and we have already proved that its orthogonal, is a linear groupoid. Condition (\ref{condition2}) follows immediately.

\item By condition (\ref{condition4}) for $E$, it follows that $q_1:U^3(E)\to E$ is onto, whence (by condition \ref{condition1}) $q_3:U^3(E)\to E$ is onto. Therefore $q_{1,2}:U^3(E)^\perp\to E^*\times E^*$ is injective.

\item  Since $q_{2,3}:U^3(E)\to E\times E$ injective, it follows that  $q_1:U^3(E)^\perp\to E^*$ is onto. Since $U^2(E)$ is the graph of an involution, the same holds for $U^2(E)^\perp$. Note also that condition (\ref{condition4}) ensures that $q_1:U^2(E)^\perp\to E^*$ is an isomorphism. We then just have to show that $\{(x,y,z)\in U^3(E)^\perp;\ z\in (U^1(E))^\perp\}=\{(x,y,z)\in U^3(E)^\perp; \ (x,y)\in U^2(E)^\perp\}$. The first term is the orthogonal of $U^3(E)+F_1$ where $F_1=\{(0,0,z);\ z\in U^1(E)\}$ and the second the orthogonal of $U^3(E)+F_2$ where $F_2=\{(x,y,0);\ (x,y)\in U^2(E)\}$.  Now, for every $(x,y)\in U^2(E)$ there exists $z\in U^1(E)$, namely $z=r_E(x)=s_E(y)$ such that $(x,y,z)\in U^3(E)$; by surjectivity of $r_E$, it follows that for every $\gamma\in G$ and every $z\in E^{(0)}_\gamma$ there exists $x\in E_\gamma$ such that $r_E(x)=z$, thus $(x,x^{-1},z)\in U^3(E)$. In other words, $U^3(E)+F_1=U^3(E)+F_2$. Condition (\ref{condition4}) follows.

\item  We just need to show that $U^4(E)^\perp=\{(w,x,y,z)\in (E^*)^4;\ \exists (u,u')\in U^2(E)^\perp, \ (w,x,u)\in U^3(E)^\perp\ \hbox{and} \ (u',y,z)\in U^3(E)^\perp\}$. As $U^4(E)^\perp$ is cyclicly invariant, condition (\ref{condition5}) will follow.

If there exists $(u,u')\in U^2(E)^\perp$ such that $(w,x,u)\in U^3(E)^\perp$ and $(u',y,z)\in U^3(E)^\perp$, then, for every $(a,b,c,d)\in U^4(E)$, there exists $(v,v')\in U^2(E)$ such that $(a,b,v)\in U^3(E)$ and $(v',c,d)\in U^3(E)$. It follows that $(w,x,y,z)\in U^4(E)^\perp$.
The inclusion $\{(w,x,y,z)\in (E^*)^4;\ \exists (u,u')\in U^2(E)^\perp, \ (w,x,u)\in U^3(E)^\perp\ \hbox{and} \ (u',y,z)\in U^3(E)^\perp\}\subset U^4(E)^\perp$ follows.

Now, as vector bundles, if $\dim E=n$ and $\dim E^{(0)}=p$, it follows that $\dim (U^k(E))=(k-1)n-(k-2)p$ (for all $k\ge 1$); therefore $\dim U^3(E)^\perp=n+p$ and $\dim U^4(E)^\perp=n+2p$. As the projection $q_1:U^3(E)^\perp\to E^*$ is onto, we find that for $(\gamma_1,\gamma_2,\gamma_3,\gamma_4)\in U^4(G)$, we have $\dim \{(w,x,u),(v,y,z)\in U^3(E)^\perp_{(\gamma_1,\gamma_2,\gamma_3\gamma_4)}\times U^3(E)^\perp_{(\gamma_1\gamma_2,\gamma_3,\gamma_4)};\ v=u^{-1}\}$ is $(n+p)+(n+p)-n$ and we find the desired equality by dimension equality.  
\end{enumerate}

It is then quite immediately seen, using induction and dimension equality, that, for every $k$, we have $U^k(E^*)=U^k(E)^\perp$.
\end{proof}
\end{proposition}

\subsection{Fourier transform}

\begin{remark}\label{rem:Fourier}
Let $F_1,F_2$ be real vector spaces and let $H$ be a subspace of $F_1\times F_2$. Assume that $p_1:H\to F_1$ is injective and $p_2:H\to F_2$ is surjective. For $f\in \scS(F_1)$ we put $q_H(f)=(p_2)_!(p_1^*)(f)$.  Then, for every $f\in \scS(F_1)$, $\widehat{q_H(f)}=q_{H^\perp}(\hat f)$ (taking a good normalization for the Fourier transform).

Indeed we can write $F_1=F_2\times L\times K$, $H=\{((x,y,0),x),\ x\in F_2,\ y\in L\}$. Then $H^\perp =\{((\xi,0,\eta),-\xi),\ \xi\in F_2^*,\ \eta\in K^*\}$. 

For $x\in F_2$, we have $q_H(f)(x)=\int _{L}f(x,y,0)\,dy$.

For $\xi \in F_2^*$, we have $$\widehat{q_H(f)}(\xi)=\int _{F_1\times L}f(x,y,0)e^{-i\langle x|\xi\rangle}\,dx\,dy$$ and $$q_{H^\perp}(\hat f)(\xi)=\int_{K^*}\hat f(\xi,0,\eta)\,d\eta=\int_{K^*}\Big(\int_{F_2\times L\times K} f(x,y,z)e^{-i(\langle x|\xi\rangle+\langle z|\eta\rangle)}dx\,dy\,dz\Big)d\eta.$$
But $\int_{K^*}\Big(\int_{K} f(x,y,z)e^{-i \langle z|\eta\rangle}dz\Big)d\eta=f(x,y,0)$.

\end{remark}

\begin{remark}
Let $E\to G$ be a \VBG. For $\gamma\in G$, $E_x$ is a linking space between the groupoids $E_{s(x)}$ and $E_{r(x)}$. The family of Hilbert bimodules $C^*(E_x)_{x\in G}$ is a \emph{Fell bundle} and $C^*(E)$ is the $C^*$-algebra associated with this Fell bundle (\cite{Kumj, Muh, MuWi, IonWil}).
\end{remark}

\begin{proposition}
Let $E\to G$ be a \VBG, and let $E^*$ be the dual \VBG. Then $C^*(E)\simeq C^*(E^*)$ - via Fourier transform. 
\begin{proof}
For $(\gamma,\gamma')\in G^{(2)}$, let $F_1\subset E_{\gamma}\times E_{\gamma'}$ be the set of composable elements; let $F_2=E_{(\gamma\gamma')^{-1}}$ and $H\subset F_1\times F_2$ the set of $(x,y,z)\in E_{\gamma}\times E_{\gamma'}\times E_{(\gamma\gamma')^{-1}}$ that compose to a unit. Remark \ref{rem:Fourier} implies that for $f\in \scS(E_\gamma)$ and $g\in \scS(E_{\gamma})$, we have $\widehat {(f\cdot g)}=\hat f\cdot \hat g$ - where $f\cdot g\in E_{\gamma\gamma'}$ is the ``Fell bundle product''. In other words, the Fourier transform map is  an isomorphism of the Fell-bundles and therefore the corresponding $C^*$-algebras are isomorphic.

\end{proof}
\end{proposition}

\renewcommand{\nomname}{List of Symbols}

\printnomenclature[3cm]

\bibliography{BiblioMB.bib} 
\bibliographystyle{amsplain}

\end{document}